\documentclass[11pt]{article}

\usepackage[margin=1in]{geometry}
\usepackage{amsmath}
\usepackage{amssymb}
\usepackage{amsfonts}
\usepackage{amsthm}
\usepackage{keytheorems}
\usepackage{graphicx}
\usepackage{hyperref}
\usepackage{cleveref}
\usepackage{quiver}
\usepackage{graphicx}
\usepackage{enumitem}

\title{
    \bfseries
    \Large
    Cerf Diagrams and Hatcher-Wagoner
    Invariants for Barbell Maps
}
\author{Xiayu Tan}
\date{}

\newcounter{thmenvcnt}[section]

\NewDocumentCommand{\newtheoremalias}{m m O{} O{\MakeLowercase{#2}} O{#4s} O{#2} O{#6s}}{%
  \newkeytheorem{#1}[
    title={#2},
    sibling=thmenvcnt,
    refname={#4,#5},
    Refname={#6,#7},
    qed={#3} 
  ]
  \newkeytheorem{#1*}[
    title={#2},
    numbered=no,
  ]
}

\theoremstyle{plain}
\newtheoremalias{theorem}{Theorem}
\newtheoremalias{lemma}{Lemma}
\newtheoremalias{proposition}{Proposition}
\newtheoremalias{corollary}{Corollary}[][corollary][corollaries][Corollary][Corollaries]

\theoremstyle{definition}
\newtheoremalias{definition}{Definition}
\newtheoremalias{example}{Example}
\newtheoremalias{remark}{Remark}[$\lrcorner$]
\newtheoremalias{question}{Question}
\newtheoremalias{acknowledgement}{Acknowledgement}

\graphicspath{{figures/}}

\DeclareMathOperator{\Diff}{Diff}
\DeclareMathOperator{\Emb}{Emb}
\DeclareMathOperator{\id}{id}
\DeclareMathOperator{\Wh}{Wh}
\newcommand{\Z}{\mathbb{Z}}

\begin{document}

\maketitle


\begin{abstract}
    For a half-unknotted implanted $(i,n-i)$-barbell $\beta=\beta_{i,n-i}$ in $M^n$, we construct two specific pseudo-isotopies, which we denote by standard barbell pseudo-isotopies, both resulting in that barbell diffeomorphism, each having a Cerf diagram only containing a single eye and with easily computable Hatcher-Wagoner invariants. We give an explicit formula for $\beta_{2,n-2}$ and a special class of $\beta_{3,n-3}$. Using this we show that for $n\geq 6$, every pseudo-isotopy with vanishing first Hatcher-Wagoner invariant can be isotoped to a composition of standard barbell pseudo-isotopies with $i=2$ or $3$. In dimension $n=4$, we further generalize the constructions and computations to half-unknotted immersed barbell diffeomorphisms and prove that for every $s\in \mathbb{Z}_2, \sigma\in \pi_2 M,\gamma\in \pi_1 M$ with $s=0 \text{ or }w_2^M(\sigma)\neq0$, there exists a standard immersed barbell pseudo-isotopy $f_\beta$ with the second induced Hatcher-Wagoner invariant $\Theta(f_\beta)=(s,\sigma)\cdot [\gamma]$.
\end{abstract}

\section{Introduction}

Given a smooth oriented manifold $X$, let $\mathcal{P}(X)$ be the pseudo-isotopy group of $X$ defined by 
$$\mathcal{P}(X):=\{f\in \Diff(X\times I)|f_{\partial X\times I\cup X\times{0}}=\id\}$$
The pseudo-isotopy group $\mathcal{P}(X)$ is closely related to the diffeomorphism group of $X$ by $\mathcal{P}(X)\xrightarrow{F} \Diff(X, \partial): f\to F_f=f|_{X\times 1}$, in that case we say that $f$ results in $F_f$.

Hatcher and Wagoner first studied $\pi_0\mathcal{P}(X)$ in their collected work \cite{AST_1973__6__1_0} and it was later studied by Igusa in \cite{K.Igusa}. They found two invariants for the group, namely, $$\Sigma: \pi_0(P)\to \Wh_2(\pi_1X)\text{ and }\Theta: \ker\Sigma\to \Wh_1(\pi_1X; \mathbb{Z}_2\times \pi_2 X)/\chi(K_3\mathbb{Z}[\pi_1 X])$$
where we have $\Wh_1(\pi_1X; \mathbb{Z}_2\times \pi_2 X)=(\mathbb{Z}_2\times \pi_2 M)[\pi_1 M]/(\beta\cdot [1], \alpha\cdot [\sigma]-\alpha^\tau\cdot [\tau\sigma\tau^{-1}], \alpha,\beta\in \mathbb{Z}_2\times \pi_2 M, \tau,\sigma\in \pi_1 M)$.  These two maps are always well-defined whenever $\text{dim}X\geq 4$, and they showed that $\Sigma$ is always surjective when $\text{dim}X\geq 5$, and $\Theta$ is surjective when $\text{dim}X\geq 5$ and bijective when $\text{dim}X\geq 6$. In \cite{singh2022pseudoisotopiesdiffeomorphisms4manifolds} Singh studied the partial images of $\Sigma$ and $\Theta$ in dimension 4 and showed that both $\Sigma$ and $\Theta$ are stably surjective.

The pseudo-isotopy group is closely related to the diffeomorphism group, for there is a natural fiber bundle 
$$\mathcal{J}=\Diff(X\times I,\partial)\to \mathcal{P}(X)\xrightarrow{F} \Diff_{PI}(X,\partial)$$
where $\Diff_{PI}(X,\partial)\subset \Diff(X,\partial) $ denotes all the diffeomorphisms of $X$ which are pseudo-isotopic to identity. Thus we have \emph{induced} Hatcher-Wagoner invariants $$\Sigma: \pi_0\Diff_{PI}(X,\partial)\to \Wh_2(\pi_1 X)/\Sigma(\mathcal{J})$$ $$\Theta: \pi_0F(\ker\Sigma)\to \Wh_1(\pi_1 X; \mathbb{Z}_2\times \pi_2 X)/(\Theta(\mathcal{J}\cap \ker\Sigma)+\chi(K_3\mathbb{Z}[\pi_1 X])) $$ on the diffeomorphism group of $X$, where $\pi_0F(\ker\Sigma)\subset \pi_0\Diff_{PI}(X,\partial)$. Then one naturally asks: 
\begin{question*}
How to give an expilict construction for a pseudo-isotopy (diffeomorphism) with prescribed (induced) Hatcher-Wagoner invariants?
\end{question*}

One candidate for the above question is the implanted barbell diffeomorphisms which were constructed by Gabai and Budney in \cite{Budney_2025}, where the $n$-dimensional standard $(i,n-i)$-barbell diffeomorphism is a specific nontrivial element in $\pi_0\Diff(\mathcal{B}_{i,n-i}^n,\partial)$ following the notations in \cite{Budney_2025}. Gabai, Budney, Gay and Hartman studied the implanted $(2,2)$-barbell diffeomorphisms in a general $X^4$, especially in $S^1\times D^3$(\cite{budney2021knotted3ballss4}) and $S^4$(\cite{Gay_2025}). One of the remarkable results is that Budney and Gabai constructed an infinitely generated subset in $\pi_0(\Diff(S^1\times D^3,\partial)/\Diff(D^4,\partial))$, more explicitly, an inclusion $$\bigoplus_{k\geq 4\in \mathbb{N}} \mathbb{Z} \to  \pi_0(\Diff(S^1\times D^3,\partial)/\Diff(D^4,\partial)): 1_k\to \delta_k$$
where $\{\delta_k\}_{k\in \mathbb{N}}$ are specific implanted barbell diffeomorphisms in $S^1\times D^3$ which we will carefully describe later in this paper.

Suppose we have an implanted barbell in a general $X^4$, i.e. $\beta:S^2\times D^2 \natural S^2\times D^2\hookrightarrow X$, and therefore an implanted barbell diffeomorphism in $\Diff(X,\partial)$, it is known that (see \cite[Proposition 2.6]{Budney_2025}) if one of the implanted core spheres is unknotted in $X^4$, then the implanted barbell diffeomorphism is pseudo-isotopic to identity. One natural question is:

\begin{question*}
    How to compute the induced Hatcher-Wagoner invariants for a given half-unknotted implanted barbell diffeomorphism?
\end{question*}

Once we get an $f_\beta\in \pi_0\mathcal{P}$ resulting in that implanted barbell diffeomorphism, $\Sigma(f_\beta), \Theta(f_\beta)$ are representatives of the two induced Hatcher-Wagoner invariants. 

In this article, we will not restrict to 4-manifolds. Budney and Gabai present a general definition of barbell diffeomorphisms in \cite{Budney_2025}, namely, for any implanted $(i,j)$-barbell $\beta=\beta_{i,j}: \mathcal{B}_{i,j}^n=S^i\times D^{-i}\natural S^j\times D^{n-j}\hookrightarrow M^n$ with $i+j\geq n$, we can naturally associate it with an element $F_\beta\in \pi_{i+j-n}\Diff(M,\partial)$. In this article, we restrict to the case when $i+j=n$, therefore the implanted barbell $\beta$ gives the corresponding barbell diffeomorphism $F_\beta\in \pi_0\Diff(X,\partial)$. In fact, for any implanted half-unknotted $(i,n-i)$-barbell $\beta$ (when we say an implanted $(i,n-i)$-barbell is half-unknotted in $M$, we usually mean the $(n-i)$-sphere is unknotted in $M$) in $M^n,n\geq 4$, we construct two pseudo-isotopies $g_\beta,f_\beta\in \pi_0\mathcal{P}$, both resulting in the implanted barbell diffeomorphism with respect to $\beta$. $g_\beta$ has a Cerf diagram containing a single eye of $(i-1,i)$-handle pair and $f_\beta$ has a Cerf diagram containing a single eye of $(n-i,n-i+1)$-handle pair. For any half-unknotted implanted $(2,n-2)$-barbell $\beta$, the main result will give an explicit formula for the Hatcher-Wagoner invariants of $f_\beta$, which is stated as follows:

\begin{theorem}
    For a half-unknotted implanted $(2,n-2)$-barbell $\beta=\beta_{2,n-2}=(R_0, S, \gamma)$ in $M^n,n\geq 4$ with $S=\partial \beta_0^\bullet$ where $\beta_0^\bullet: D^{n-1}\hookrightarrow M$, by finger-pushing $R_0$ along the arc, we can make $\gamma$ short enough such that $\text{int}(\beta_0^\bullet)\cap \gamma =\emptyset$. Now suppose that $\beta_0^\bullet \cap R_0=\bigsqcup _{i=1}^k S_i^1$. Choose $p_i\in S_i^1$ and let $*_0=\gamma\cap S$ be the base point. For each $i$, find a path $\delta^B_i\subset \beta=(R_0, S,\gamma)$ which is a path from $*_0$ to $p_i\in S_i^1$. Also, for each $i$, $S_i^1$ divides $R_0$ into two embedded disks $D_i$ and $D_i'$, where $D_i'$ is the one connected to the arc $\gamma$. Let $D^B_i=D_i$ (see \Cref{fig:general-barbell} for an illustration). Then there is a pseudo-isotopy $f_\beta\in \pi_0\mathcal{P}$ resulting in that implanted barbell diffeomorphism with the Cerf diagram of $f_\beta$ being a single eye of $(n-2,n-1)$-handle pair such that:

    $$\Sigma(f_\beta)=0, \Theta(f_\beta)=\sum_{i=1,...,k} (0, [D^B_i]^{\delta_i^B})\cdot  [\delta_i^B] $$
    Here we identify $\pi_i(M,*_0)$ with $\pi_i(M, \beta_0^\bullet)$ so that $[D_i]\in \pi_2(M,\beta_0^\bullet)$ and $\delta_i^B\in \pi_1(M,\beta_0^\bullet)$. 
\end{theorem}

\begin{figure}[!ht]
    \centering
    \includegraphics[width=0.5\textwidth]{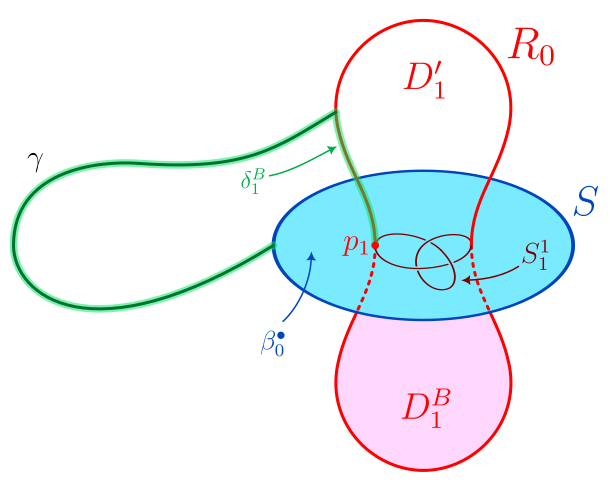}
    \caption{A general barbell and data needed to compute Hatcher-Wagoner invariants}
    \label{fig:general-barbell}
\end{figure}

\begin{remark}
    In particular, when the given barbell $\beta=(R_0,S,\gamma)$ with $S=\partial\beta_0^\bullet$ and $\gamma:[0,1]\to M,\gamma(0)=*_0=\gamma\cap S$ satisfies $R_0\cap \beta_0^\bullet=\emptyset$, suppose $\gamma|_{(0,1)}\cap \beta_0^\bullet=\{p_i\}_{i=1}^k$ with $p_i=\gamma(t_i)$, then the $f_\beta\in \pi_0\mathcal{P}$ we constructed in the theorem satisfies:$$\Sigma(f_\beta)=0, \Theta(f_\beta)=(0,[R_0^\gamma])\sum_{i=1}^k [\gamma_i]$$where $\gamma_i:=\gamma|_{[0,t_i]}\in \pi_1(M,\beta_0^\bullet)$ and $[R_0^\gamma]\in \pi_2(M, \beta_0^\bullet)$ is obtained by pulling $R_0$ back to $*_0$ along $\gamma$.
\end{remark}

\begin{remark} \label{rmk:involution}
    Along the way we also compute the Hatcher-Wagoner invariants for $g_\beta$: $\Sigma(g_\beta)=0, \Theta(g_\beta)=(-1)^{n+1}\bar{\Theta}(f_\beta)$, where $\bar{\cdot}$ is an involution on $\Wh_1(\pi_1 M, \mathbb{Z}_2\times \pi_2 M)$ which is described in both \cite{AST_1973__6__1_0} and \cite[section 9.1]{singh2022pseudoisotopiesdiffeomorphisms4manifolds}. We review the involution here:
    $$\bar{\cdot}: (n,\sigma)\cdot[\gamma] \mapsto (n+w_2^X(\sigma), -w_1^X(\gamma)\sigma^{\gamma^{-1}})\cdot [\gamma^{-1}],$$where $w_2^X(\sigma)\in \{0,1\}$ is the second Stiefel-Whitney class of the normal bundle of $\sigma$ in $X$, and $w_1^X(\gamma)\in \{-1,1\}$ is the first Stiefel-Whitney class of the normal bundle of $\gamma$ in $X$. To be careful that when we write $\bar{\Theta}(f_\beta)$, we first view $\Theta(f_\beta)\in Wh_1(\pi_1 M; \mathbb{Z}_2\times \pi_2 M)$, then apply the involution $\bar{\cdot}$ to it, and finally view $\bar{\Theta}(f_\beta)$ as an element in $\Wh_1(\pi_1 M; \mathbb{Z}_2\times \pi_2 M)/\chi(K_3\mathbb{Z}[\pi_1 M])$. We cannot view $\Theta(f_\beta)$ as an element in $\Wh_1(\pi_1 M, \mathbb{Z}_2\times \pi_2 M)/\chi(K_3\mathbb{Z}[\pi_1 M])$ directly and do the involusion $\bar{\cdot}$ on $\Wh_1(\pi_1 M, \mathbb{Z}_2\times \pi_2 M)/\chi(K_3\mathbb{Z}[\pi_1 M])$ since by \cite[Conjecture 9.7]{singh2022pseudoisotopiesdiffeomorphisms4manifolds} it is still a conjecture whether $\bar{\cdot}$ is an involusion on this quotient group, i.e. we don't know whether $\overline{\chi(K_3\Z [\pi_1 M])}=\chi(K_3\Z [\pi_1 M])$.
\end{remark}

\begin{remark}
    In particular, in dimension $n=4$, consider Budney and Gabai's examples $\delta_k$, then all $\{f_{\delta_k}\}_{k\in \mathbb{N}}$ in $\mathcal{P}(S^1\times D^3)$ lie in the kernel of $\Sigma$ and $\Theta$, this answers a question raised by Powell in \cite[Question 12.6]{powellmapping}. But $\{\delta_k, k\geq 4\}$ are nontrivial diffeomorphisms of $S^1\times D^3$, so $[f_{\delta_k}, k\geq 4]\neq 0\in \pi_0\mathcal{P}(S^1\times D^3)$, which equivalently means that $\{f_{\delta_k},k\geq 4\}$ are nontrivial pseudo-isotopies of $S^1\times D^3$ which can not be detected by Hatcher-Wagoner invariants.
\end{remark}

In dimension $n=4$, we generalize the calculation to \emph{half-unknotted immersed barbell diffeomorphisms} which will be defined in \cref{sec:generalizations-to-immersed-barbell-diffeomorphisms}. The result is exactly the same:

\begin{theorem}
    For a half-unknotted immersed barbell $\beta=(R_0,S,\gamma)$ with $\beta_0^\bullet: D^3\hookrightarrow M, S=\partial\beta_0^\bullet $, perturb self-intersections of $R_0$ away from $\beta_0^\bullet$. By finger-pushing $R_0$ along arc, we can make $\gamma$ short enough such that $\text{int}(\beta_0^\bullet)\cap \gamma =\emptyset$. Now suppose that $\beta_0^\bullet \cap R_0=\bigsqcup _{i=1}^k S_i^1$. Choose $p_i\in S_i^1$ and let $*_0=\gamma\cap S$ be the base point. For each $i$, find a path $\delta^B_i\subset \beta=(R_0, S,\gamma)$ which is a path from $*_0$ to $p_i\in S_i^1$. Also, for each $i$, $S_i^1$ divides $R_0$ into two immersed disks $D_i$ and $D_i'$, where $D_i'$ is the one connected to the arc $\gamma$. Let $D^B_i=D_i$. Then the $f_\beta\in \pi_0\mathcal{P}$ we constructed which results in the immersed barbell diffeomorphism with respect to $\beta$ satisfies:
    
    $$\Theta(f_\beta)=\sum_{i=1,...,k} (0, [D^B_i]^{\delta_i^B})\cdot  [\delta_i^B]$$
    Here we identify $\pi_i(M,*_0)$ with $\pi_i(M, \beta_0^\bullet)$ so that $[D_i]\in \pi_2(M,\beta_0^\bullet)$ and $\delta_i^B\in \pi_1(M,\beta_0^\bullet)$.    
\end{theorem}

From the theorems we deduce the \emph{barbell realization theorem} in dimension 4: 
\begin{theorem} \label{thm:barbell-realization-dim-4}
    For a 4-dimensional manifold $M$, given any $(s,\sigma)\cdot[\gamma]\in \Wh_1(\pi_1 M; \mathbb{Z}_2\times \pi_2 M)$ with $s=0$ or $w_2^M(\sigma)\neq 0$,

    \begin{enumerate}[label=\em(\arabic*)]
        \item If $s=0$, there exists a half-unknotted immersed barbell $\beta$ and $f_\beta\in \ker\Sigma\subset \pi_0\mathcal{P}$ resulting in the immersed barbell diffeomorphism with respect to $\beta$, whose Cerf diagram contains a single eye of (2,3)-handle pair, satisfying $\Theta(f_\beta)=(s,\sigma)\cdot[\gamma]$.
        \item If $s=1$ and $w_2^M(\sigma)=1$, there exists a half-unknotted immersed barbell $\beta$ and $g_\beta\in \ker\Sigma\subset \pi_0\mathcal{P}$ resulting in the immersed barbell diffeomorphism with respect to $\beta$, whose Cerf diagram contains a single eye of (1,2)-handle pair, satisfying $\Theta(g_\beta)=(s,\sigma)\cdot[\gamma]$. \qedhere
    \end{enumerate}
\end{theorem}

In particular, it covers all second Hatcher-Wagoner invariants realizable in Singh's paper \cite{singh2022pseudoisotopiesdiffeomorphisms4manifolds}.

\Cref{thm:barbell-realization-dim-4} has a parallel version for dimension $n\geq 5$:

\begin{theorem} \label{thm:barbell-realization-dim-n}
    For a $n$-dimensional manifold $M$ with $n\geq 5$, given any $(s,\sigma)\cdot[\gamma]\in \Wh_1(\pi_1 M; \mathbb{Z}_2\times \pi_2 M)$ with $s=0$ or $w_2^M(\sigma)\neq 0$,

    \begin{enumerate}[label=\em(\arabic*)]
        \item If $s=0$, there exists an implanted barbell $\beta=\beta_{2,n-2}$ with the $(n-2)$-sphere unknotted in $M$ and $f_\beta\in \ker\Sigma\subset \pi_0\mathcal{P}$ resulting in the implanted barbell diffeomorphism with respect to $\beta$, whose Cerf diagram contains a single eye of $(n-2,n-1)$-handle pair, satisfying $\Theta(f_\beta)=(s,\sigma)\cdot[\gamma]$.
        \item If $s=1$ and $w_2^M(\sigma)=1$, there exists an implanted barbell $\beta=\beta_{2,n-2}$ and $g_\beta\in \ker\Sigma\subset \pi_0\mathcal{P}$ resulting in the implanted barbell diffeomorphism with respect to $\beta$, whose Cerf diagram contains a single eye of $(1,2)$-handle pair, satisfying $\Theta(g_\beta)=(s,\sigma)\cdot[\gamma]$. \qedhere
    \end{enumerate}
\end{theorem}

But as we mentioned above, Hatcher and Wagoner showed that the second Hatcher-Wagoner invariants $\Theta: \ker\Sigma\to \Wh_1(\pi_1 M; \mathbb{Z}_2\times \pi_2 M)/\chi(K_3\mathbb{Z}[\pi_1 M])$ is a surjection for $n\geq 5$ and bijection for $n\geq 6$. To give a refinement of that result, for any $\gamma\in \pi_1 M$, we construct a specific half-unknotted implanted $(3,n-3)$-barbell for $n\geq 6$ and calculate its Hatcher-Wagoner invariants:

\begin{proposition}
    Given a $n$-dimensional manifold $M$ with $n\geq 6$, for every $\gamma\in \pi_1 M$, there is a half-unknotted $(3,n-3)$-barbell $\beta=(R_0, S,\gamma)$ and a pseudo-isotopy $f_\beta\in \ker\Sigma$ resulting in the barbell diffeomorphism with respect to $\beta$, whose Cerf diagram contains a single eye of $(n-3,n-2)$-handle pair, satisfying that $\Theta(f_\beta)=(1,0)\cdot[\gamma]$.

\end{proposition}

Note that $\{(0,\sigma)\cdot[\gamma], (0,1)\cdot[\gamma]|\sigma\in \pi_2 M, \gamma\in \pi_1M\}$ generates $\Wh_1(\pi_1 M; \mathbb{Z}_2\times \pi_2M)$. Finally we obtain the \emph{barbell realization theorem} for dimension $n\geq 6$ as follows:

\begin{theorem} \label{thm:diffeomorphism-decomposed-by-barbells}
    For a $n$-dimensional manifold $M$ with $n\geq 6$,

    \begin{enumerate}[label=\em(\arabic*)]
        \item For any pseudo-isotopy $f$ of $M$ with $\Sigma(f)=0$, there is a series $\beta_j, j=1,...,m$ of half-unknotted implanted $(2,n-2)$-barbells or half-unknotted implanted $(3,n-3)$-barbells such that $f$ is isotopic to $f_{\beta_1}\circ ...\circ f_{\beta_m}$.
        \item For any diffeomorphism $F\in \Diff_{PI}(M,\partial)$ such that there exists a pseudo-isotopy $f\in \pi_0\mathcal{P}$ resulting in $F$ with $\Sigma(f)=0$, $F$ is isotopic to a finite composition of half-unknotted $(2,n-2)$-barbell diffeomorphisms and half-unknotted $(3,n-3)$-barbell diffeomorphisms. \qedhere
    \end{enumerate}
\end{theorem}

\begin{corollary}
    For a manifold $M^n,n\geq 6$ with $\Wh_2(\pi_1 M)=0$, for example, $M=S^1\times D^{n-1}$ or $S^1\times S^{n-1}$, any diffeomorphism $F\in \Diff(M,\partial)$ which is pseudo-isotopic to identity can be isotoped into a composition of half-unknotted barbell diffeomorphisms.   
\end{corollary}

The paper is organized as follows: In \cref{sec:cerf-theory-and-parameterised-handle-constructions} we recall the Cerf theory and the parameterised handle constructions version of it, especially we explain how a loop of handle constructions of $X\times I$ can result in a pseudo-isotopy $f\in \pi_0\mathcal{P}$ of $X$, and how to see the resulting $F_f\in\Diff(X,\partial)$ directly from the loop of handle constructions. In \cref{sec:twin-twists-and-relation-with-barbell-diffeomorphisms} we recall the general notion of Montesino $(i,n-i)$-twins in $X^n,n\geq 4$ and the resulting twin twists in $\Diff(X,\partial)$ which were studied for $X=S^4$ in \cite{Gay_2025}. Then we show that if one of the spheres, let it be the $(n-i)$-sphere, in the Montesino twins is unknotted, then $\exists f\in \mathcal{P}(X)$ resulting in that twin twist with Cerf diagram being a loop of a single cancelling $(i-1,i)$-handle pair. Then we show that any implanted barbell diffeomorphism is a twin twist, so together with the above result, for a specific implanted $(i,n-i)$-barbell $\beta$ with the $(n-i)$-sphere unknotted in $X$, we find $g_\beta\in \mathcal{P}(X)$ with Cerf diagram being a loop of a single cancelling $(i-1,i)$-handle pair resulting in that implanted barbell diffeomorphism. In \cref{sec:changing-from-i-1-i-handle-pair-to-n-i-n-i-1-handle-pair} we follow the essential lemma \cite[Lemma 17]{Gay_2025} to develop a technique of \emph{one parameter version of 0-framed and dotted replacement} to change the pseudo-isotopy $g_\beta\in \pi_0\mathcal{P}$ with Cerf diagram being the above loop of cancelling $(i-1,i)$-handle pair to $f_\beta\in \pi_0\mathcal{P}$ with Cerf diagram being a loop of cancelling $(n-i,n-i+1)$-handle pair. Therefore we get a loop of cancelling $(n-i,n-i+1)$-handle pair resulting in implanted barbell diffeomorphism of $\beta$. In \cref{sec:computing-second-hatcher-wagoner-invariants-for-beta-2-n-2} we use the technology in \cref{sec:changing-from-i-1-i-handle-pair-to-n-i-n-i-1-handle-pair} to complete the computations of the second Hatcher-Wagoner invariant for $f_\beta$ with a given half-unknotted implanted $(2,n-2)$-barbell diffeomorphism. In \cref{sec:generalizations-to-immersed-barbell-diffeomorphisms}, for $n=4$ and $i=2$, we generalize to \emph{immersed barbell diffeomorphisms} and compute the induced Hatcher-Wagoner invariants for a half-unknotted immersed barbell diffeomorphism. As corollaries, we deduce \Cref{thm:barbell-realization-dim-4} as the \emph{barbell realization theorem} in dimension 4 and \Cref{thm:barbell-realization-dim-n} for dimension $n\geq 5$. In \cref{sec:a-special-barbell-beta-3-n-3-realizes-surjectivity-when-n-geq-6}, we show that for $n\geq 6$, $\forall \gamma\in \pi_1 M$, we can construct a half-unknotted $(3,n-3)$-barbell $\beta$ with the second Hatcher-Wagoner invariant $\Theta(f_\beta)=(1,0)\cdot[\gamma]$, together with \Cref{thm:barbell-realization-dim-n}, we know that $\ker\Sigma$ is generated by $\{f_\beta|\beta $ half-unknotted $(2,n-2)$-barbell or half-unknotted $(3,n-3)$-barbell$\}$, which implies \Cref{thm:diffeomorphism-decomposed-by-barbells}. 

\subsection*{Acknowledgements}

The author would like to thank Jean Cerf, Allen Hatcher and John Wagoner for developing the excellent theory about pseudo-isotopy. She also thanks David Gabai, David Gay, Daniel Hartman, Ryan Budney and Olivier Singh for their remarkable works on diffeomorphisms and pseudo-isotopies of 4-manifolds. She also thanks Daniel Hartman and Yicheng Yang for useful conversations along the way. She also thanks Yifei Fan for some of the illustrations in this paper. Most especially, she thanks Jianfeng Lin for being her mentor, making many inspiring discussions,  reading the first version of this paper and giving useful advice. 

The work was carried out during the author's fourth year in Qiuzhen College, Tsinghua University, and thus she expresses her deep gratitude to the institution.

\section{Cerf theory and parameterised handle constructions version}
\label{sec:cerf-theory-and-parameterised-handle-constructions}

We have defined the pseudo-isotopy group $\mathcal{P}(X)$ and want to study $\pi_0\mathcal{P}(X)$. Cerf discovered a great correspondence: 
$$\pi_0\mathcal{P}\cong \pi_1(\mathcal{F},\mathcal{E})$$
where $\mathcal{F}=\mathcal{F}(X):=\{f: X\times I\to I| f=\text{standard projection on a neighborhood of }\partial(X\times I)\}$ and $\mathcal{E}=\mathcal{E}(X)\subset \mathcal{F}(X)$ contains all such $f$ that have no critical points at all. The explicit correspondence is: Given $f\in \mathcal{P}(X)$, let $p$ be the standard projection from $X\times I$ to $I$. Then $p\circ f\in \mathcal{E}(X)$ since $f$ is a diffeomorphism of $X\times I$. It is not hard to show that $\mathcal{F}(X)$ is contractible, therefore we can choose a path $\gamma=\{\gamma_t:t\in I\}$ in $\mathcal{F}$ with $\gamma_0=p, \gamma_1=p\circ f$ both in $\mathcal{E}$. Then $\gamma$ is the corresponding element in $\pi_1(\mathcal{F},\mathcal{E})$. Conversely, given $f_t: X\times I\to I$ starting at $f_0=p$ and ending at $f_1=q\in \mathcal{E}$, consider the gradient vector field $V_q$ of $q$ on $X\times I$, integrating $V_q$ gives a diffeomorphism of $X\times I: (x,t)\to \phi_{V_q}(x,0)(t)$.

Then we reduce the question to the study of \emph{one parameter family of functions} on $X\times I$, namely, $f_t: X\times I\to I$ with $f_0,f_1$ having no critical points at all. Like what we did in Morse theory, we can perturb the path fixing $f_0,f_1$ such that for all but finitely many $t\in I$, $f_t$ is a Morse function with ordered, distinct critical values. At the exceptional points, there appears either an $(i,i+1)$-birth/death point (two critical points of index $i$ and $i+1$ have the same critical value, but by the ordering condition, they must cancel in one direction) or an $(i,i)$-crossing point (two non-degenerate critical points of same index $i$ with crossing critical values). The Cerf diagram records the changes of critical values in $f_t: X\times I$. Here's one example (see \Cref{fig:cerf-diagram-with-2-eyes}):

\begin{figure}[!ht]
    \centering
    \includegraphics[width=0.5\textwidth]{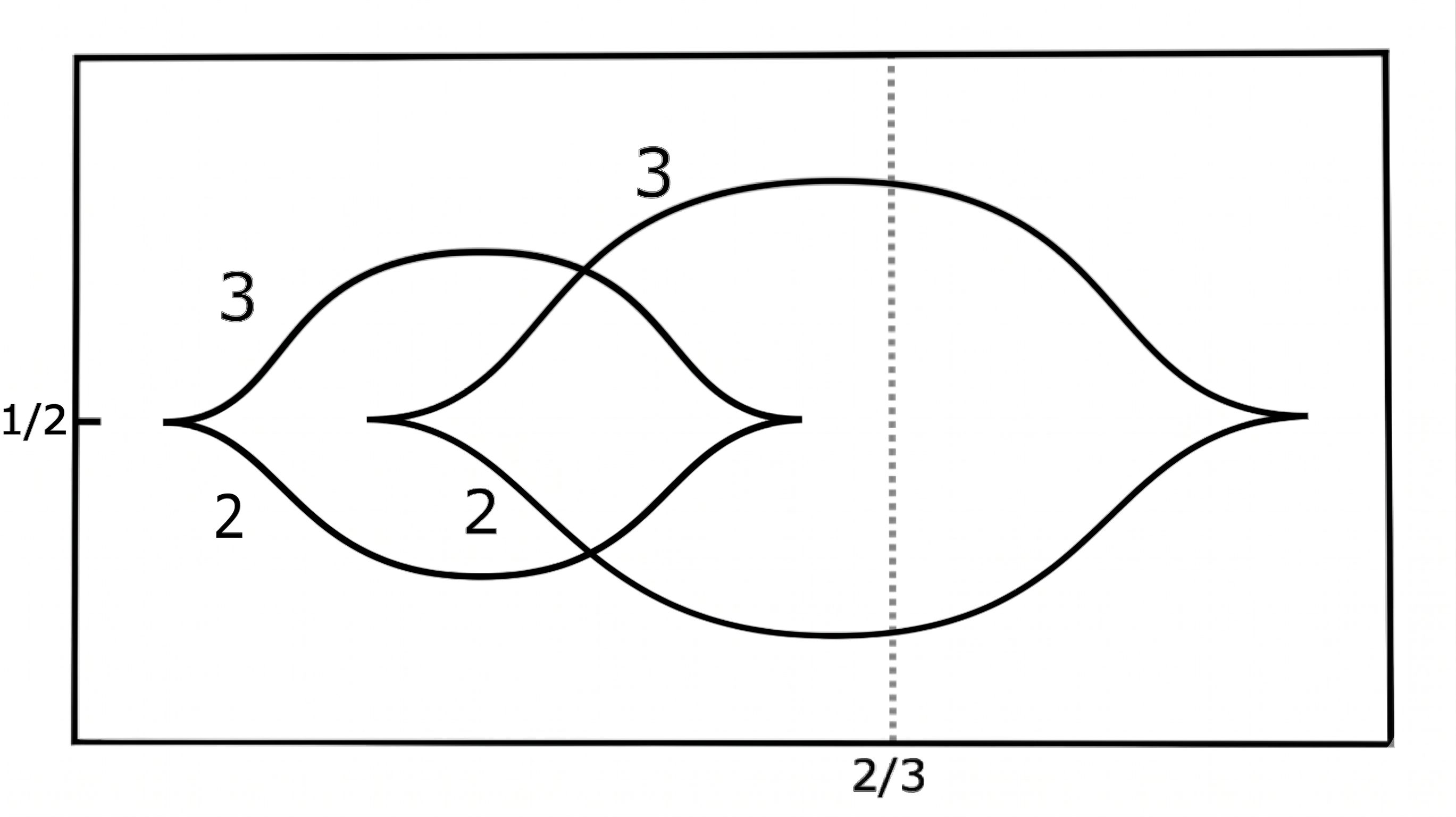}
    \caption{This is a Cerf diagram with 2 eyes, each eye begins with the birth and finishes with the death. All births and deaths happen in $g^{-1}_t(1/2)$. Except the four birth/death points, every $g_t$ is Morse. For example, $g_{2/3}$ has two non-degenerate critical points of index 2 and 3. There are two crossing points in the diagram. }
    \label{fig:cerf-diagram-with-2-eyes}
\end{figure}

In this paper we are trying to find a Cerf diagram for some pseudo-isotopy $f_\beta$ with only one eye of indices $k$ and $k+1$ for some $k$. Then let $g_t, t\in I$ be the corresponding one parameter family of functions on $X\times I$ with $g_0=p, g_1=p\circ f_\beta$. If we successfully find such a family $g_t$ with such a Cerf diagram, then the only things that happen here are as follows: For some small $\epsilon$, $g_\epsilon$ births a degenerate critical point(standard local model is $g_\epsilon=x^3-\sum_{i=1}^ky_i^2+\sum_{j=k+1}^{n-k}y_j^2$ near the origin $p=(0,...,0
)$ with $g_\epsilon(p)=1/2$), and for $t> \epsilon$, the critical point splits into two non-degenerate critical points of indices $k$ and $k+1$(standard local model is $g_{\epsilon+s}=x^3-sx-\sum_{i=1}^ky_i^2+\sum_{j=k+1}^{n-k} y_j^2$). Right after the birth point at $g_\epsilon^{-1}(1/2)$, the level set is $g_{\epsilon+s}^{-1}(1/2)=X\# S^k\times S^{n-k}$, with the dual sphere of the unique $k$-handle being $*\times S^{n-k}$ and the attaching sphere of the unique $k+1$-handle being $S^k\times *$. Then the two spheres isotope in $X\# S^k\times S^{n-k}$ during the interval $t\in (\epsilon, 1-\epsilon)$. They may intersect each other at more than one point during the isotopy, but finally, right before $t=1-\epsilon$, they return to the dual position, i.e. intersect at exactly one point. Then general Morse theory states that the ($k$,$k+1$)-handle pair can be cancelled, so a death occurs at $t=1-\epsilon$. But via an isotopy, we can assume one of the spheres is fixed throughout. In short, the $\{g_t,t\in I\}\in \pi_1(\mathcal{F},\mathcal{E})$ is completely governed by an element in $\pi_1(\Emb(S^k\times D^{n-k}, X\# S^k\times S^{n-k}), \Emb_0(S^k\times D^{n-k}, X\# S^k\times S^{n-k}))$ where $\Emb_0(S^k\times D^{n-k}, X\# S^k\times S^{n-k})$ denotes all the embedded $S^k\times D^{n-k}$ such that $S^k\times 0$ transversely intersect $*\times S^{n-k}$ at a single point.

In particular, whenever we have a loop of framed embedded $S^k$ in $X\# S^k\times S^{n-k}$, which means the element lies not only in the relative $\pi_1$, but also in $\pi_1(\Emb(S^k\times D^{n-k}, X\# S^k\times S^{n-k}),*)$, tracing the flow lines it is not hard to show that the resulting diffeomorphism on $X$ is just the following composition:
$$\pi_1(\Emb(S^k\times D^{n-k}, X\# S^k\times S^{n-k}),*)\to \pi_0\Diff(X\setminus \nu S^{n-k-1},\partial)\xrightarrow{\cup \id_{\nu S^{n-k-1}}}\pi_0\Diff(X,\partial)$$
where the first map is the isotopy extension at $t=1$ and removing $\nu S^k=\nu (S^k\times *)$ in $X\# S^k\times S^{n-k}$, which is diffeomorphic to $X\setminus \nu S^{n-k-1}$, where the $S^{n-k-1}$ denotes the dual sphere of the $(k+1)$-handle.

Just like \emph{"Building a manifold $X$ from a Morse function is equivalent to finding a handle decomposition for $X$"}, we have an equivalent version for pseudo-isotopy: \emph{Finding a one-parameter family of functions on $X\times I$ is equivalent to finding a one-parameter family of handle decompositions for $X\times I$}.

Thus if we have a loop $\mathcal{H}=\{H_t, t\in I|H_0=H_1\}$ of handle decompositions of $X\times I$ which starts and ends at a standard cancelling position which can be made null, by the above statement, this corresponds to an element in $\pi_1(\mathcal{F},\mathcal{E})=\pi_0\mathcal{P}$, where the resulting manifold $Z_t$ is diffeomorphic to $X\times I$. Since $H_0=H_1$, we have a natural way to identify $Z_0$ with $Z_1$ (since they are totally the same manifold). Then we build a cobordism from $S^1\times X$ to $Y=X\times I/((x,0)\sim (f_\mathcal{H}(x), 1))$ where $f_\mathcal{H}$ is the resulting diffeomorphism on $X$ with respect to $\mathcal{H}\in \pi_1(\mathcal{F},\mathcal{E})$. In handlebody language, we describe what $f_\mathcal{H}$ is: Each $H_t$ gives a series of surgeries $X=X_{0,t}\to X_{1,t}\to ...\to X_{n,t}$, where $X_{n,t}$ is diffeomorphic to $X$ but not in a natural way. But again, since $H_0=H_1$, we have $X_{n,0}=X_{n,1}\cong X$. But as $t$ varies, since $H_t$ moves smoothly, every attaching region moves smoothly, thus we have a diffeomorphism $\phi_{m,t}: X_{m,0}\to X_{m,t}, \forall m,t$. In particular $\phi_{n,1}=f_\mathcal{H}:X=X_{n,0}\to X_{n,1}=X$ is the desired diffeomorphism on $X$.

In particular, when $\mathcal{H}$ is a loop of a single $(k,k+1)$-cancelling pair, $Z_t=X\times I\cup h_{k,t}\cup h_{k+1,t}$. Since $h_{k,0}$ and $h_{k+1,0}$ form a standard cancelling pair, the attaching sphere of $h_{k,0}$ must be unknotted. Then let $X_{1,0}=(X\setminus  \nu S^{k-1}) \cup D^k\times S^{n-k}=X\# S^k\times S^{n-k}$ be the manifold after doing surgery on $h_{k,0}$. If we further assume $h_{k,t}=h_{k,0}$ throughout, then $X_{1,t}=X_{1,0}=X\# S^k\times S^{n-k}$. The only $t$-dependent data is the attaching region of $h_{k+1,t}$, which corresponds to $\pi_1(\Emb(S^k\times D^{n-k}, X_{1,0}),*)$, then one can directly see that the resulting $f_\mathcal{H}: X=X_{2,0}\to X_{2,1}=X$ is the composition map mentioned above.

\section{Twin twists and relation with barbell diffeomorphisms}
\label{sec:twin-twists-and-relation-with-barbell-diffeomorphisms}

First we derive a standard model for Montesino $(i,n-i)$-twins: For $i\geq 2$ and $n-i\geq 1$, consider $X=S^1\times S^{i-1}\times D^{n-i}$ with boundary $\partial X= S_l\times S_S\times S_R$, with $S_l=S^1\times 1\times 1$ a 1-sphere, $S_S=-1\times S^{i-1}\times 1$ an $(i-1)$-sphere and $S_R={-1}\times 1\times \partial D^{n-i}$ an $(n-i-1)$-sphere. Then we do surgery on $\nu(1\times S^{i-1}\times 0)$, and let $X'=(X\setminus \nu(1\times S^{i-1}\times 0))\cup D^i\times S^{n-i}$ be the resulting manifold. Then the core $T_R=(S^1\times S^{i-1}\times 0)$ of $X$ becomes $R=(T_R\setminus I\times S^{i-1} )\cup \partial I\times D^i$, which is a $i$-sphere with normal sphere $S_R$, and $S=0\times S^{n-i}\subset D^i\times S^{n-i}$ is a $(n-i)$-sphere with normal sphere $S_S$. Moreover, $R$ intersects with $S$ transversely at 2 points. The resulting manifold $X'$ is just a tubular neighborhood of $R\cup S$ with the same boundary $\partial X'=\partial X=S_l\times S_S\times S_R$. Based at $*=(-1,1,1)$, $S_l$ is homotopic to a longitude of $X'$, that is, $\pi_1(\partial X',*)\to \pi_1(X',*)\cong \pi_1(S^i\vee S^{n-i}\vee S^1,*)\to \pi_1(S^1, *)$ sends $[S_l]$ to the third component $1\in \pi_1(S^1,*)$ after a suitable orientation. An easy way to see that is that when $n-i-1=1$, $\pi_1(\partial X,*)\to \pi_1(X,*)$ simply kills $S_R$ and sends others via identity, when $n-i-1\geq 2$, $\pi_1(\partial X,*)=\pi_1(X,*)$. And by surgering $X$ to $X'$, if $i-1=1$, $\pi_1(X')=\pi_1(X)/[S_S]$ since it is a parallel copy of $1\times S^{i-1}\times 0$, if $i-1\geq 2$, $\pi_1(X')=\pi_1(X)$. Moreover, a simple way to construct a longitude in $X'$ is to consider $\{p,q\}=R\cap S$, and choose $\gamma_R\subset R, \gamma_S\subset S$ connecting $p$ and $q$, then the resulting $\gamma_R\cup_{p,q} \gamma_S$ is a longitude of $X'$. 

A standard twin twist is a Dehn twist near the boundary of $X'$, that is, choose a neighborhood of $\partial X'=I\times S_l\times S_S\times S_R$, and let $\phi_D\in \Diff(I\times S_l,\partial)$ be the standard Dehn twist, then twin twist $\tau$ is identity-extension of $\phi_D\times \id\times \id\in \Diff(\nu(\partial X'),\partial)$ in $\Diff(X',\partial)$.

We provide a more natural way to construct the twin twist $\tau$: Consider a loop of $S^{i-1}$ in $X$: $\gamma_t: S^{i-1}\to e^{2\pi i t}\times S^{i-1}\times 0$, the isotopy extension of $\gamma_t$ at $t=1$ induces a diffeomorphism on $X$ which fixes $\gamma_0$. Moreover, it can be made to fix the whole core $T_R$, that is, it induces a diffeomorphism on $X\setminus \nu T_R=\nu (\partial X)$, which is exactly $\phi_D\times \id\times \id$. Therefore $\tau$ is the induced diffeomorphism after surgering along $\gamma_0$. 

\begin{definition}
   A \emph{Montesino $(i,n-i)$-twin} in $M^n$ is a pair $W=(R,S)$ where $R$ is an embedded $S^i$ in $M$ with trivial normal bundle, $S$ is an embedded $S^{n-i}$ in $M$ also with trivial normal bundle and $R$ intersects $S$ transversely at 2 points. Equivalently, a Montesino twin in $M^n$ is an embedding $i_W: X'\hookrightarrow M$. The twin twist induced by $W$ is just the implanted diffeomorphism $\tau$, denoted by $\tau_W$. 
\end{definition}

Any twin twist is induced by a \emph{parameterised surgery of index $(i-1)$} (see \cite{kosanović2025diffeomorphisms4manifoldsgraspers} for details) which we will briefly describe below:

\begin{definition} \label{def:parameterised-surgery}
    Given a framed embedded $S^{n-i}$ in $M^n$, that is, $\nu S: S^{n-i}\times D^i\to M^n$, let $M_{\nu S}:=(M\setminus \nu S) \cup D^{n-i+1}\times S^{i-1}$ be the manifold obtained by performing surgery on $S$. We say a diffeomorphism $\phi$ of $M$ is induced by a \emph{parameterised surgery of index $(i-1)$} if $\phi$ is in the image of $\text{ps}_{\nu S}$:
    $$\text{ps}_{\nu S}: \pi_1(\Emb(\nu S^{i-1}, M_{\nu S}), i_0)\to \pi_0\Diff(M\setminus \nu S, \partial)\xrightarrow{\cup \id_{\nu S}} \pi_0\Diff(M,\partial)$$
    where $i_0=0\times S^{i-1}\subset M_{\nu S}$ with natural framing is the base point of the embedding space and the first map is induced by isotopy extension.
\end{definition}

\begin{proposition} \label{prop:twin-twist-is-induced-by-parameterised-surgery}
    For any Montesino $(i,n-i)$-twin $W=(R,S)$ in $M$, the twin twist $\tau_W$ is induced by a \emph{parameterised surgery of index $(i-1)$} on $S$. 
\end{proposition}

\begin{proof}
    By the standard model of Montesino twins we have constructed, doing surgery on $S$ yields an embedding $i: X=S^1\times S^{i-1}\times D^2=\nu T_R\to M_{\nu S}$ with $i_0=i(1\times S^{i-1}\times 0)$. By the second way we described the twin twist $\tau$, $\tau_W=\text{ps}_{\nu S}(\gamma_t)$ where $\gamma_t: S^{i-1}\to i(e^{2\pi it}\times S^{i-1}\times 0)\subset M_{\nu S}$.
\end{proof}

\begin{proposition} \label{prop:twin-twist-is-pseudo-isotopic-to-identity-when-s-unknotted}
    When $W=(R,S)$ with $S$ unknotted, then there exists a loop $\mathcal{H}\in \pi_0\mathcal{P}$ of handle decompositions of $M\times I$ with a single $(i-1, i)$-cancelling pair resulting in $\tau_W$. In particular, $\tau_W$ is pseudo-isotopic to identity with the corresponding Cerf diagram being a single eye of $(i-1,i)$-handle pair. 
\end{proposition}

\begin{proof}
    When $S$ is an unknotted $S^{n-i}$, $M_{\nu S}=M\# S^{n-i+1}\times S^{i-1}=M\# S^{i-1}\times S^{n-i+1}$, this can be obtained by attaching a trivial $(i-1)$-handle $h_{i-1,t}=h_{i-1,0}$ on $M\times I$. Then \Cref{prop:twin-twist-is-induced-by-parameterised-surgery} proved that $\tau_W$ is induced by an element $\gamma\in \pi_1(\Emb(\nu S^{i-1}, M\# S^{i-1}\times S^{n-i+1}), i_0)$. We attach the $i$-handle $h_{i,t}$ along $\gamma_t$, then $\mathcal{H}=\{H_t,t\in S^1\}=\{(h_{i-1,t}, h_{i,t}), t\in S^1\}$ forms a pseudo-isotopy of $M$ with the Cerf diagram being a single eye of $(i-1,i)$-handle pair. Then $\tau_W=\phi_\gamma(1)|_{M\# S^{i-1}\times S^{n-i+1}\setminus \nu S^{i-1}}\cup \id_{\nu S}=f_\mathcal{H}$ where $\phi_\gamma (t)$ denotes the time $t$ isotopy extension on $M\# S^{i-1}\times S^{n-i+1}$.
\end{proof}

Then we recall the definition of an (implanted) barbell diffeomorphism and prove that any implanted barbell diffeomorphism is a special twin twist.

\begin{definition}
    The \emph{standard $(i,n-i)$-barbell diffeomorphism} $\beta=\beta_{i,n-i}$ on $S^{i}\times D^{n-i} \natural S^{n-i}\times D^{i}$ is defined as follows: Consider $D^n$ and 2 disjoint properly embedded $D^{n-i-1}, D^{i-1}$. Let $D^{i-1}$ rotate around a normal sphere $S^i\times *\subset SN(D^{n-i-1})=S^i\times D^{n-i-1}$ and come back. This isotopy extension induces an element in $\Diff(D^n\setminus (\nu D^{n-i-1}\cup \nu D^{i-1})=S^{i}\times D^{n-i} \natural S^{n-i}\times D^{i}, \partial)$. The \emph{implanted $(i,n-i)$-barbell diffeomorphism} is obtained by implanting a barbell $\beta=\beta_{i,n-i}=(S_1, S_2, \gamma)$, i.e. $S_1$ a framed embedded $S^i$ and $S_2$ a framed embedded $S^{n-i}$ with an arc $\gamma$ connecting them, and extending the barbell diffeomorphism by $\id$ on $M\setminus \nu \beta$. Later on when we say a implanted $(i,n-i)$-barbell $\beta=(S_1,S_2,\gamma)$ is \emph{half-unknotted}, we usually mean $S_2$ is unknotted in $M$, i.e. $S_2=\partial\beta_0^\bullet$ for some $\beta_0^\bullet: D^{n-i+1}\hookrightarrow M$.
\end{definition}

\begin{proposition}
    Let $\beta=(R_0,S,\gamma)$ be an implanted $(i,n-i)$-barbell in $M$. Finger-push $R_0$ along $\gamma$ into $S$ to get another sphere $R$; then $R$ intersects $S$ transversely at 2 points. Then the implanted barbell diffeomorphism by $\beta$ is isotopic to twin twist $\tau_W$ where $W=(R,S)$.
\end{proposition}

\begin{proof}
    We first describe $\tau_W$ with $W=(R,S)$: After surgering along $S$, we get an embedding $\nu T_R\to M_{\nu S}$ where $T_R=S^1\times S^{i-1}$ which has a framing induced by $\nu R_0$, this corresponds to an element $[\nu T_R]$ in $\pi_1(\Emb(\nu S^{i-1}, M_{\nu S}),*)$, then we have:
    \[\begin{tikzcd}
        & {\pi_0\text{Diff}(M,\partial)} \\
        {\pi_1(\Emb(S^{i-1}, M_{\nu S}),*) } & {\pi_0\text{Diff}(M_{\nu S}\setminus \nu S^{i-1}=M\setminus\nu S, \partial)} \\
        {\pi_1(\Emb(D^{i-1}, D^n\setminus \nu D^{n-i-1}=S^i\times D^{n-i}),*)} & {\pi_0\text{Diff}(S^i\times D^{n-i}\natural S^{n-i}\times D^i,\partial)}
        \arrow[from=2-1, to=2-2]
        \arrow["{\cup \id_{\nu S}}", from=2-2, to=1-2]
        \arrow[from=3-1, to=2-1]
        \arrow[from=3-1, to=3-2]
        \arrow[from=3-2, to=2-2]
    \end{tikzcd}\]

    The first line is just $\text{ps}_{\nu S}$ so by \Cref{prop:twin-twist-is-induced-by-parameterised-surgery} it sends $[\nu T_R]$ to $\tau_W$. On the other hand, when $S^{i-1}$ goes around $T_R$, it just goes along $\gamma$, winds $R_0$ around and then comes back. This is exactly what $D^{i-1}$ does in the standard barbell. Thus we just implant $(\nu D^{i-1}, S^i\times D^{n-i})$ into $M_{\nu S}$ (see Figures in \Cref{eg:example-of-S1xD3} below). The resulting diffeomorphism in $M_{\nu S}\setminus \nu S^{i-1}=M\setminus \nu S$ is the barbell diffeomorphism with respect to theimplanted $\beta=(R_0, \text{normal sphere of }S^{i-1}=S, \gamma)$, when extended to $M$ it is still the same barbell diffeomorphism.
\end{proof}

Combining the above two propositions we get:
\begin{corollary}
    For any implanted $(i,n-i)$-barbell $\beta=(R,S,\gamma)$ with $S$ unknotted in $M$, the implanted barbell diffeomorphism is pseudo-isotopic to identity by $g_\beta\in \pi_0\mathcal{P}$ with the corresponding Cerf diagram being a single eye of $(i-1,i)$-handle pair. Moreover, it is induced from an element $[\nu T_R]\in \pi_1(\Emb(S^{i-1}\times D^{n-i+1}, M\# S^{i-1}\times S^{n-i+1}), *)$.
\end{corollary}

\begin{example} \label{eg:example-of-S1xD3}
    Here we draw the corresponding loop in $\pi_1(\Emb(S^1, S^1\times D^3\# S^1\times S^3), *)$ (for simplicity we omit the framing but it's just the
     natural framing induced by $\nu T_R|_{S^1}$ and $\nu (S^1,T_R)$) which corresponds to Gabai's constructions $\delta_k$ (see \Cref{fig:implanted-barbell-in-S1xD3,fig:finger-push-of-barbell,fig:surgery-along-S-to-get-torus}).
\end{example}

\begin{figure}[!ht]
    \centering
    \includegraphics[width=0.5\textwidth]{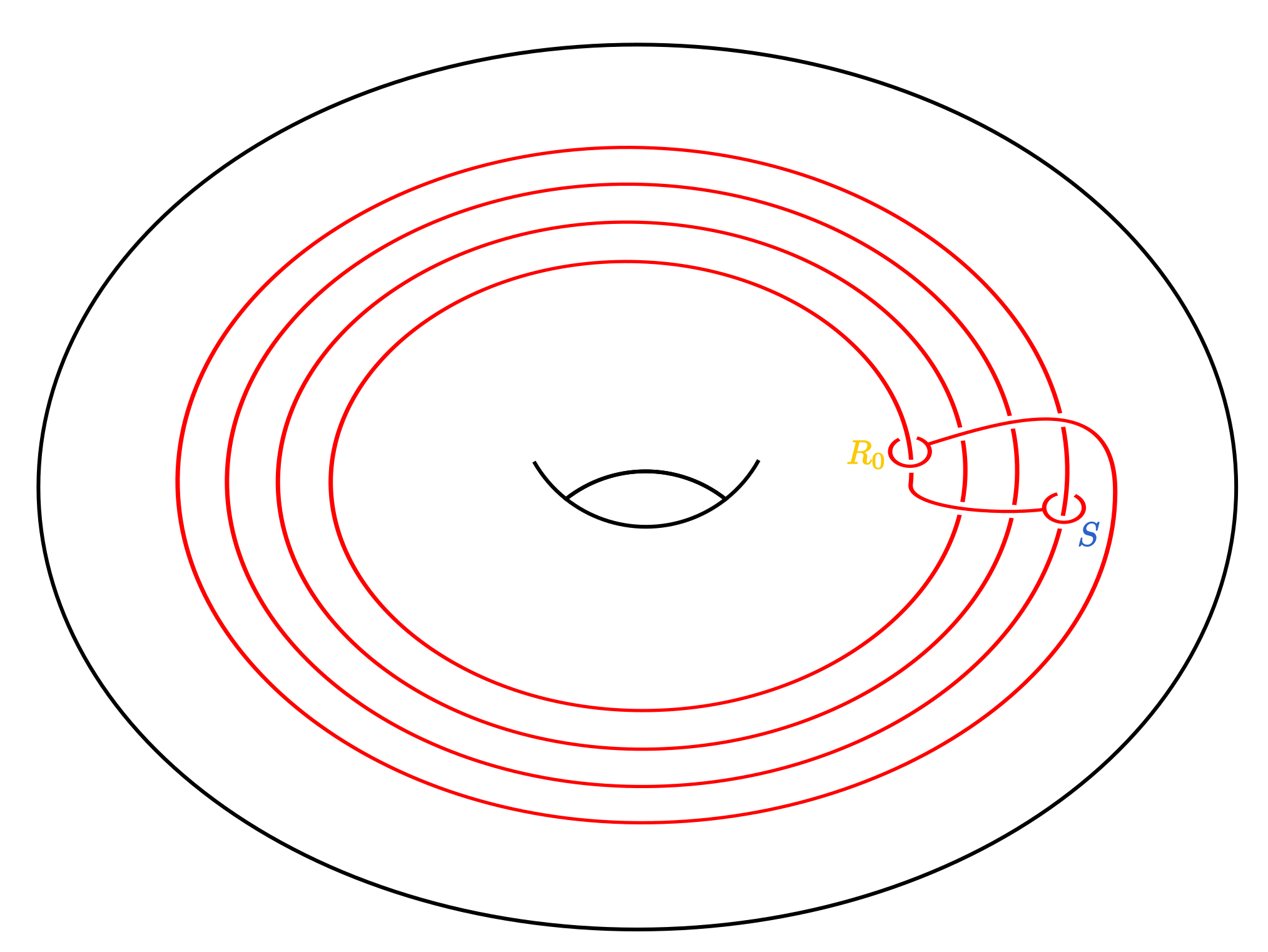}
    \caption{The implanted barbell $\delta_4$ in $M=S^1\times D^3$}
    \label{fig:implanted-barbell-in-S1xD3}
\end{figure}

\begin{figure}[!ht]
    \centering
    \includegraphics[width=0.5\textwidth]{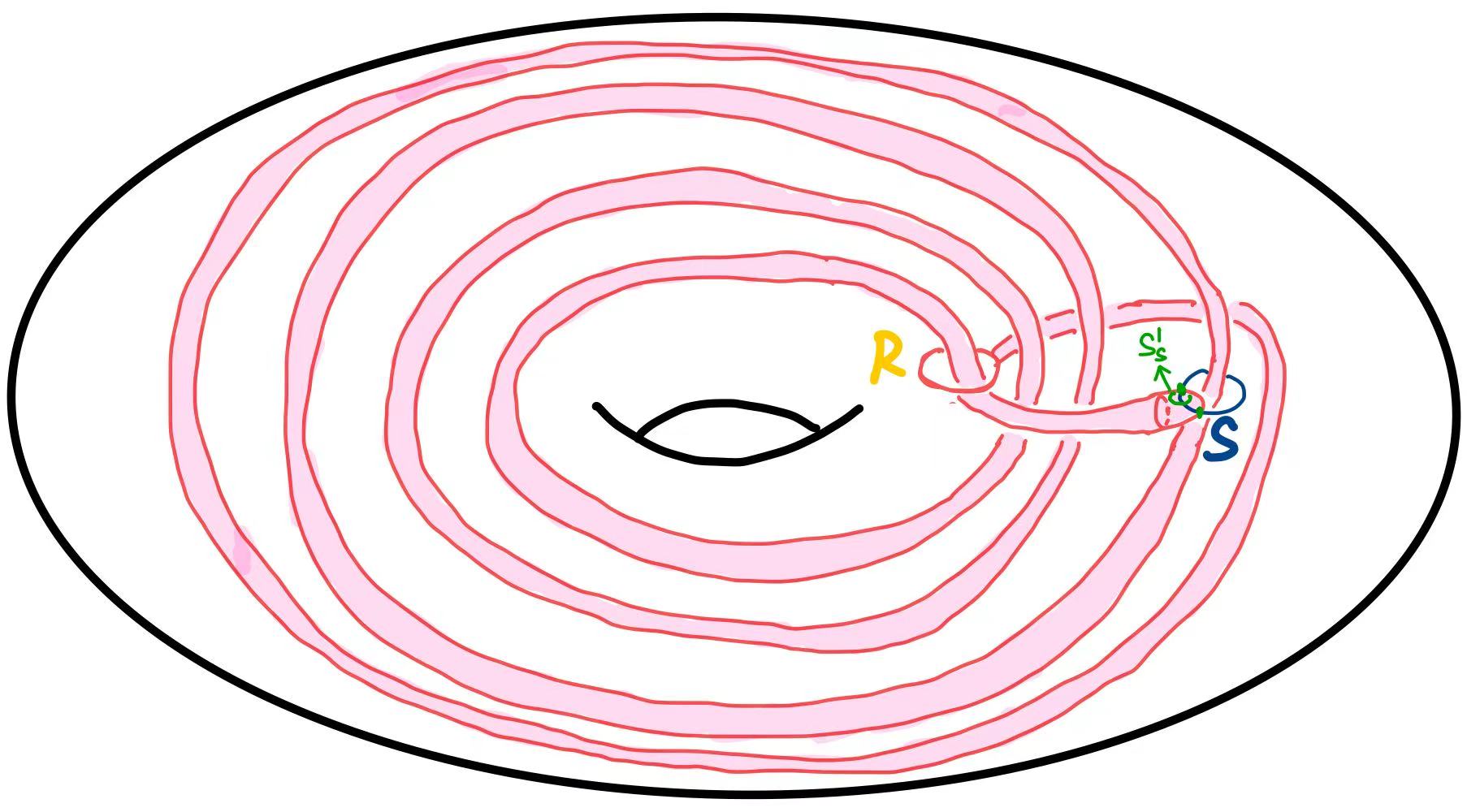}
    \caption{Finger-push $R_0$ to get $R$ such that $R\cap S=2$ points}
    \label{fig:finger-push-of-barbell}
\end{figure}

\begin{figure}[!ht]
    \centering
    \includegraphics[width=0.7\textwidth]{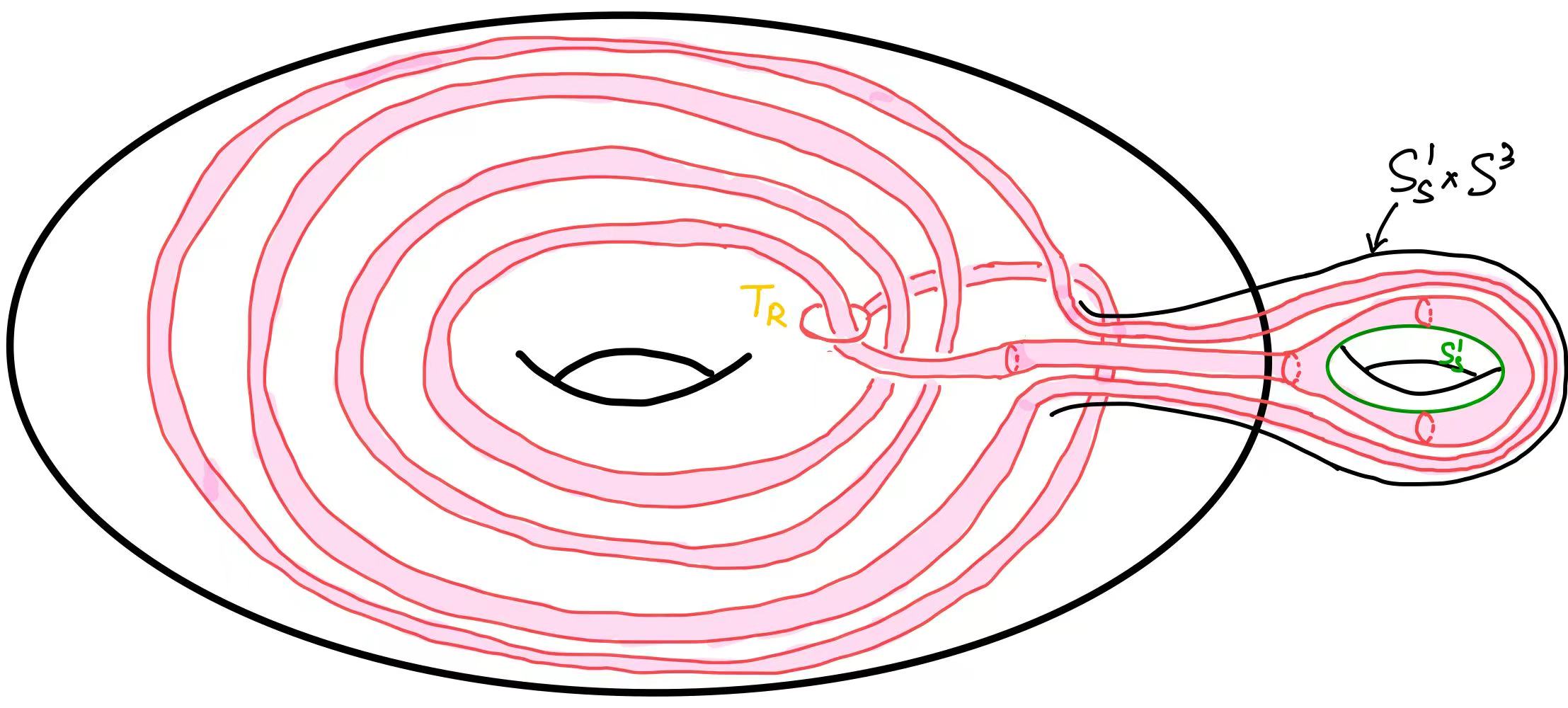}
    \caption{Do surgery along $S$ to get an embedded torus $T_R$ in $S^1\times D^3\# S^1\times S^3$ with the standard $S^1$ (green one), this $T_R$ determines an element in $\pi_1(\Emb(S^1, S^1\times D^3\# S^1\times S^3),*)$ which is a loop of handle decompositions $\mathcal{H}=\{H_t,t\in S^1\}$ (a loop of (1,2)-handle pair) of $S^1\times D^3\times I$ which results in the barbell diffeomorphism $\delta_4=\tau_W$ where $W=(R,S)$.}
    \label{fig:surgery-along-S-to-get-torus}
\end{figure}
\section{Changing from \texorpdfstring{$(i-1,i)$}{(i-1,i)}-handle pair to \texorpdfstring{$(n-i,n-i+1)$}{(n-i,n-i+1)}-handle pair}
\label{sec:changing-from-i-1-i-handle-pair-to-n-i-n-i-1-handle-pair}

The author first found the excellent idea from \cite{Gay_2025}. It is essentially a higher dimensional generalizations of Kirby calculus, where in dimension 4, we use dotted unknotted $S^1=\partial D^2\subset M^3$ to represent a 1-handle, and framed $S^1\subset M$ to represent 2-handle. The central idea is that when the framed $S^1$ is unknotted and 0-framed, \emph{swapping the dotted circle and the 0-framed circle} changes the cobordism but results in the same boundary manifold. 

In this section we first describe the dotted version of trivial $k$-handles in a $(n+1)$-dimensional handle construction $H$ on $X^{n+1}=M\times I$, then derive the strategy of \emph{dotted and 0-framed replacement}. After that, we move on to the one-parameter versions of both, namely, dotted version of a Cerf diagram containing a single eye of $(k-1,k)$-handle pair which corresponds to an element in $\pi_1(\Emb(S^{k-1}\times D^{n-k+1}, M\# S^{k-1}\times S^{n-k+1}),\Emb_0(S^{k-1}\times D^{n-k+1}, M\# S^{k-1}\times S^{n-k+1}))$, and carefully develop the \emph{one-parameter version of dotted and 0-framed replacement}.

Consider the trivial cobordism $Z=M\times I$ from $M$ to $M$, if we attach a trivial $k$-handle on top of $Z$, which means the attaching sphere $S^{k-1}=\partial D^k$ for an embedded $D^k\hookrightarrow M$ and the framing of $S^{k-1}$ is induced by this embedded $D^k$, we get $Z_1=Z\cup h_k$ which is a cobordism from $M$ to $M_1=(M\setminus S^{k-1}\times D^{n-k}) \cup D^k\times S^{n-k}=M\# S^k\times S^{n-k}$. If we push the one handle into original $Z=M\times I$, we know that it is the same as cutting a neighborhood of an unknotted properly embedded $D_{in}^{n-k}$ in $M\times I$, that is, $Z_1=M\times I\setminus \nu D_{in}^{n-k}$ where $D^{n-k}_{in}$ is obtained as follows: Choose an embedding $i: D^{n-k}\to M\times 1$, fix $i(\partial D^{n-k})$ then push $i(D^{n-k})$ into $M\times I$ to get $D_{in}^{n-k}$. Then when the attaching sphere $S^k$ of a $(k+1)$-handle goes through $S^k\times S^{n-k}\setminus D^n\subset M_1$, if we can isotope this $S^k$ such that it is disjoint from $S^k\times *\in S^k\times S^{n-k}$ (when $2k<n$, we can always do that by slightly perturbing the attaching region), then in the dotted version, the $S^k$ can be seen in $M\setminus \nu (\partial D^{n-k})$. And whenever $S^k$ goes through the belt sphere $*\times S^{n-k}$, in the dotted version, $S^k\subset M\setminus \nu (\partial D^{n-k})$ intersects with that $D^{n-k}$. In short, the dotted version of a $k$-handle can be described as:

\begin{definition}
    For a $(n+1)$-dimensional cobordism $Z$ from $\partial_- Z$ to $\partial_+ Z=M$ with $M$ connected, the \emph{dotted version of describing a new cobordism} $Z_1=Z\cup h_k$ with $h_k$ a \emph{trivial} $k$-handle, is by choosing an embedding $\beta^\bullet :D^{n-k}\hookrightarrow M\times 1$ with $\beta:=\partial D^{n-k}$ denoting a dotted $(n-k-1)$-sphere, then $Z_1$ is diffeomorphic to $Z\setminus \nu \beta^\bullet _{in}$ with $\beta_{in}^\bullet : D^{n-k}\hookrightarrow Z$ obtained by pushing interior of $\beta^\bullet (D^{n-k})$ into $Z$ a bit. Any framed embedded $S^k$ in $M\setminus \nu \beta$ represents a new $(k+1)$-handle $h_{k+1}$ on $Z_1$. Whenever the attaching sphere of $h_{k+1}$ intersects $\beta^\bullet$ in the dotted version (or equivalently, it links with $\beta$), it runs over the belt sphere of $h_k$.
\end{definition}

Now if we are giving \emph{both} data for a pair of trivial $(k,n-k-1)$ handles, we can perform the \emph{dotted and 0-framed replacement}:

\begin{proposition}
    For an $(n+1)$-dimensional cobordism $Z$ with $\partial_+ Z=M$, suppose we are given $\beta^\bullet: D^{n-k}\hookrightarrow M$, $\gamma^\bullet: D^{k+1}\hookrightarrow M$ with $\beta,\gamma$ the corresponding boundary embeddings. If $\gamma\cap \beta=\emptyset$, then $\partial_+(Z\setminus \nu \beta^\bullet_{in} \cup _{(\gamma, 0)}  h_{k+1})=\partial_+(Z\setminus \nu \gamma^\bullet _{in} \cup_{(\beta, 0)} h_{n-k})  $. Here $(\gamma, 0)$ means that the attaching sphere of $h_{k+1}$ is $\gamma$ and the framing is induced by $\nu (\gamma, \gamma^\bullet)\oplus \nu (\gamma^\bullet, M)|_{\gamma}$, similar for $(\beta,0)$.
\end{proposition}

\begin{proof}
    Both sides are $(M\setminus (\nu \beta \cup \nu \gamma))\cup_{(\beta,0)} D^{n-k}\times S^{k} \cup _{(\gamma, 0)} D^{k+1}\times S^{n-k-1}$.
\end{proof}

\begin{remark}
    Later on in the dotted version, when we say the attaching region $S^{i-1}\times D^{n-i+1}$ of some handle $h_i$ is \emph{0-framed}, we will specify a $D^i\hookrightarrow M$ with the attaching sphere $S^{i-1}=\partial D^i$, and the framing is naturally induced by $\nu (S^{i-1},D^i)\oplus \nu(D^i,M)|_{S^{i-1}}$. Usually the $D^i$ is clear from the context.
\end{remark}

\begin{remark} \label{rmk:orientation-reversing-cobordism}
    Let $Z'=Z\setminus \nu \beta^\bullet_{in} \cup _{(\gamma, 0)}  h_{k+1}$ and $Z''=Z\setminus \nu \gamma^\bullet _{in} \cup_{(\beta, 0)} h_{n-k}$. Note that the two cobordisms result in the same boundary manifolds but they are different cobordisms themselves. In fact, when $Z=M\times I$, we can embed $Z'$ into a bigger $ M\times I$ such that $Z'\cup_{\partial_+ Z'=\partial_-\bar{Z''}} \bar{Z''}=M\times I$ where $\bar{Z''}$ denotes the orientation-reversing cobordism from $\partial_+ Z''$ to $\partial_- Z''$ (see \Cref{fig:orientation-reversing-cobordism}).

    \begin{figure}[!ht]
        \centering
        \includegraphics[width=0.3\textwidth]{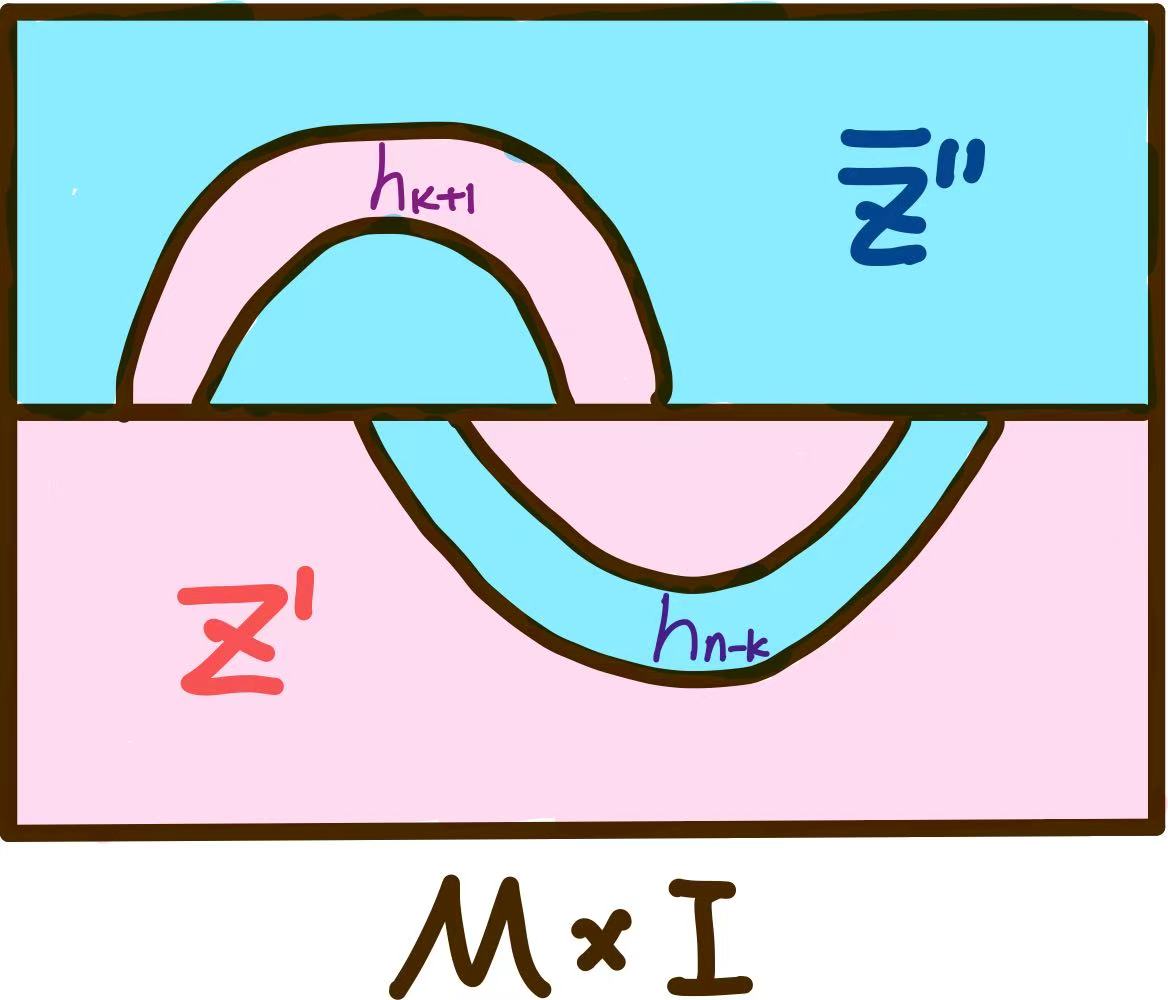}
        \caption{$Z'\cup \bar{Z}''=M\times I$}
        \label{fig:orientation-reversing-cobordism}
    \end{figure}
\end{remark}

Recall what we obtained in the last section: For any half-unknotted barbell $\beta=\beta_{i,n-i}=(R,S,\gamma)$ in $M^n$, we find a loop of framed $S^{i-1}$, i.e. an element in $\pi_1(\Emb(S^{i-1}\times D^{n-i+1}, M\# S^{i-1}\times S^{n-i+1}),*)$ which represents (by the map $\text{ps}_{\nu S}$) that barbell diffeomorphism. And thus, it gives a Cerf diagram containing a single eye of $(i-1,i)$-handle pair which results in the barbell diffeomorphism. Since in the definition of Hatcher-Wagoner invariants, we need the simple connectivity for both the dual sphere of the $(i-1)$-handle and the attaching sphere of the $i$-handle, i.e. $i-1\geq 2, n-i+1\geq 2$, thus when $i=2$, we need to find another pseudo-isotopy resulting in the same barbell diffeomorphism with the Cerf diagram being a single eye of higher index handle pair. As above we see that doing \emph{dotted and 0-framed replacement} allows us to change a $(k,k+1)$-handle pair to a $(n-k-1,n-k)$-handle pair, so the natural idea is to do the \emph{one-parameter version of dotted and 0-framed replacement}. To do that, we need the following lemma (this appeared in the work of Gay in \cite{Gay_2025} when $M=S^4$ and $k=1$, and there Gay was going from the opposite direction, namely, to change a certain eye of (2,3)-handle pair into an eye of (1,2)-handle pair, but the idea of the proof is totally the same):

\begin{lemma} \label{lem:one-parameter-version-of-dotted-and-0-framed-replacement}
    Consider a one-parameter family of embeddings: $\beta^\bullet_t: D^{n-k}\hookrightarrow M$(denote $\beta_t=\partial\beta_t^\bullet$) and $\gamma_t: S^{k}\hookrightarrow M, t\in [0,1]$ with the following properties:

    \begin{enumerate}[label=\em(\arabic*)]
        \item $\gamma_0$ extends to $\gamma_0^\bullet: D^{k+1}\hookrightarrow M$, i.e. $\gamma_0^\bullet|_{\partial D^{k+1}}=\gamma_0$, moreover, $\gamma_0^\bullet$ intersects $\beta_0$ transversely at a single point.
        \item $\beta_1^\bullet =\beta_0^\bullet$ and $\gamma_1=\gamma_0$.
        \item For all $t\in [0,1]$, $\gamma_t \cap \beta_t=\emptyset$.
        \item $\gamma_i$ intersects $\beta_i^\bullet$ transversely at a single point, for $i\in\{0,1\}$.
    \end{enumerate}

    Then there exists an extension of $\beta_t^\bullet$ and $\gamma_t$ to $t\in [0,3]$ and a one-parameter family of embeddings $\gamma_t^\bullet: D^{k+1}\hookrightarrow M, t\in [0,3]$ satisfying:

    \begin{enumerate}[label=\em(\arabic*)]
        \item For all $t\in [0,1]$, $\gamma_t$ and $\beta_t^\bullet$ stay the same as given.
        \item For $t\in[0,3]$, $\partial\gamma_t^\bullet=\gamma_t$.
        \item $\beta_3^\bullet=\beta_0^\bullet$ and $\gamma_3^\bullet=\gamma_0^\bullet$.
        \item For all $t\in [1,3]$, $\gamma_t$ intersects $\beta_t^\bullet$ transversely at a single point.
        \item $\beta_0=\beta_3$ intersects $\gamma_0^\bullet =\gamma_3^\bullet$ transversely at a single point.
        \item The path $\gamma_t^\bullet, t\in [0,3]$ is null homotopic rel $t\in \{0,3\}$ in $\Emb(D^{k+1}, M)$. \qedhere
    \end{enumerate}
\end{lemma}

\begin{proof}
    First we construct $\gamma_t^\bullet$ for $t\in [0,1]$: Since $\beta_t \cap \gamma_t=\emptyset$, choose an isotopy $\psi_t: M\to M$ which sends $\gamma_0$ to $\gamma_t$ and sends $\beta_0$ to $\beta_t$, then let $\gamma_t^\bullet :=\psi_t(\gamma_0^\bullet)$. 

    But now $\gamma_1^\bullet$ may not be equal to $\gamma_0^\bullet$. Then we construct $\gamma_t^\bullet$ and $\beta_t^\bullet$ for $t\in [1,2]$ such that $\gamma_0^\bullet=\gamma_2^\bullet$: Let $\phi_t: M\to M$ be the isotopy such that $\phi_t(\gamma_{1}^\bullet)=\gamma_{2-t}^\bullet$ then define $\gamma_t^\bullet=\phi_t(\gamma_1^\bullet)$ and $\beta_t^\bullet=\phi_t(\beta_1^\bullet)$.

    But now $\beta_2^\bullet$ may not be equal to $\beta_0^\bullet$. Then we construct $\gamma_t^\bullet$ and $\beta_t^\bullet$ for $t\in [2,3]$ such that $\gamma_t^\bullet=\gamma_2^\bullet=\gamma_0^\bullet$ for $t\in [2,3]$ and $\beta_3^\bullet=\beta_0^\bullet$: By condition (4), $\beta_0^\bullet$ is a normal disk of $\gamma_0$ and $\beta_1^\bullet$ is a normal disk of $\gamma_1$, since $\gamma_2=\phi_2(\gamma_1)=\gamma_0$ and $\beta_2^\bullet =\phi_2(\beta_1^\bullet)$, so $\beta_2^\bullet$ is also a normal disk of $\gamma_2=\gamma_0$. That means, both $\beta_2^\bullet$ and $\beta_0^\bullet$ are normal disks of $\gamma_0$, then we can isotope $\beta_2^\bullet$ to $\beta_0^\bullet$ during $t\in [2,3]$, one way is to shrink $\beta_2^\bullet$ to $\beta_t^\bullet$ inside $\beta_2^\bullet$ to a standard normal ball of $\gamma_0$ and then expand back to $\beta_0^\bullet$.

    Then from the construction one can see directly that conditions (1), (2), (3), (5) are satisfied. For condition (4), for $t\in [1,2], |\gamma_t\cap \beta_t^\bullet|=|\phi_t\gamma_1\cap \phi_t\beta_1^\bullet|=|\gamma_1\cap \beta_1^\bullet|=1$, for $t\in [2,3]$, $\beta_t^\bullet$ is always a normal disk of $\gamma_t=\gamma_0$ so satisfied. For condition (6), by $\gamma_t^\bullet=\gamma_{2-t}^\bullet, t\in [0,2]$ and $\gamma_t^\bullet=\gamma_0^\bullet, t\in [2,3]$, then of course the condition is satisfied.
\end{proof}

\begin{proposition}[one-parameter version of dotted and 0-framed replacement] \label{prop:one-parameter-version-of-dotted-and-0-framed-replacement}
    Suppose that $\beta_t^\bullet: D^{n-k}\hookrightarrow M$, $\gamma_t: S^k\hookrightarrow M$, $t\in [0,1]$ satisfying the four conditions in the above lemma, moreover, $\beta_t^\bullet=\beta_0^\bullet$, and $\gamma_0^\bullet$ is a standard normal disk of $\beta_0$. By using the above lemma, we extend $\beta_t^\bullet$ and $\gamma_t$ to $t\in[0,3]$ and get $\gamma_t^\bullet$ with $\gamma_t=\partial \gamma_t^\bullet$ for $t\in[0,3]$. We further assume $\gamma_1^\bullet$ and $\gamma_0^\bullet$ yield the same framing for $\gamma_1=\gamma_0$. Since $\beta_0^\bullet=\beta_3^\bullet$ and $\gamma_0^\bullet=\gamma_3^\bullet$,

    \begin{enumerate}[label=\em(\arabic*)]
        \item For $t\in [0,1]$, regard $\beta_t^\bullet=\beta_0^\bullet$ as dotted version of a trivial $k$-handle, regard $\gamma_t=\partial\gamma_t^\bullet$ as a loop of 0-framed $(k+1)$-handle. Then it gives $\mathcal{H}_1\in \pi_0\mathcal{P}$ with Cerf diagram being a single eye of $(k,k+1)$-handle pair and results in $f_{\mathcal{H}_1}\in \pi_0\Diff(M,\partial)$.
        \item For $t\in [0,3]$, regard $\gamma_t^\bullet$ as dotted version of a trivial $(n-k-1)$-handle (the $(n-k-1)$-handle can be made not moving for $t\in[0,3]$ by condition (6) in above lemma), regard $\beta_t=\partial\beta_t^\bullet$ as a loop of framed $(n-k)$-handle. Then it gives $\mathcal{H}_2\in \pi_0\mathcal{P}$ with Cerf diagram being a single eye of $(n-k-1,n-k)$-handle pair and results in $f_{\mathcal{H}_2}\in \pi_0\Diff(M,\partial)$.
    \end{enumerate}

    Then:  $[f_{\mathcal{H}_1}]=[f_{\mathcal{H}_2}]\in \pi_0\Diff(M,\partial)$.
\end{proposition}

\begin{proof}
    First consider $\beta_t^\bullet, \gamma_t=\partial\gamma_t^\bullet,t\in[0,3]$, this loop of handle decompositions of $M\times I$ yields $\mathcal{H}_0\in \pi_0\mathcal{P}$ with Cerf diagram being a single eye of $(k,k+1)$-handle pair. And by doing dotted and 0-framed replacement \emph{pointwise} for $t\in [0,3]$ from $\mathcal{H}_0$, we get another pseudo-isotopy which is exactly $\mathcal{H}_2\in \pi_0\mathcal{P}$, thus $f_{\mathcal{H}_0}=f_{\mathcal{H}_2}$. 
    
   But note that $\beta_1^\bullet=\beta_3^\bullet$ and $\gamma_1=\gamma_3$ with the same framing induced by $\gamma_1^\bullet$ and $\gamma_3^\bullet$, so for $t\in [1,3]$, $\beta_t^\bullet$ and $\gamma_t=\partial\gamma_t^\bullet$ also determine a loop of handle constructions for $M\times I$, thus a pseudo-isotopy $\mathcal{H}\in \pi_0\mathcal{P}$. But here $\beta_t^\bullet$ intersects $\gamma_t$ transversely at a single point for all $t\in [1,3]$, which means they are all at cancelling position for all $t$, thus $[\mathcal{H}]=[\id]\in \pi_0\mathcal{P}$. So $\beta_t^\bullet, \gamma_t=\partial\gamma_t^\bullet,t\in[0,3]$ together give a pseudo-isotopy $[\mathcal{H}_0=\mathcal{H}_1*\mathcal{H}]=[\mathcal{H}_1]\in \pi_0\mathcal{P}$.

   Thus $[f_{\mathcal{H}_1}]=[f_{\mathcal{H}_0}]=[f_{\mathcal{H}_2}]\in \pi_0\Diff(M,\partial)$.
\end{proof}

For any half-unknotted implanted barbell $\beta=\beta_{i,n-i}=(R,S,\gamma)$, from the last section, we get a loop of $(i-1,i)$-handle pair $g_\beta\in \pi_0\mathcal{P}$ resulting in the barbell diffeomorphism. By first changing to the dotted version and then applying \Cref{prop:one-parameter-version-of-dotted-and-0-framed-replacement}, we can see that there is a natural extension $\gamma_t^\bullet,t\in [0,1]$ by pulling $\gamma_0^\bullet$ along $\gamma_t$ in $T_R$, i.e. $\gamma_t^\bullet=\gamma_0^\bullet\cup_{s\in [0,t]} \gamma_s$. Thus $\gamma_1^\bullet=\gamma_0^\bullet \text{ tube}_{\gamma} R$. Thus framing of $ \gamma_1=\gamma_0$ induced by $\gamma_1^\bullet$ is the same as that induced by $\gamma_0^\bullet$ (For moving pictures of $\gamma_t^\bullet$ in the case of $\beta=\delta_k$, see \Cref{fig:movement-of-disk}). Thus we have:

\begin{corollary} \label{cor:changing-from-i-1-i-handle-pair-to-n-i-n-i-1-handle-pair}
    For any half-unknotted implanted barbell $\beta=\beta_{i,n-i}=(R,S,\gamma)$ in $M^n$, there exists $f_\beta\in \pi_0\mathcal{P}$ whose Cerf diagram contains only a single eye of $(n-i,n-i+1)$-handle pair resulting in that barbell diffeomorphism. In particular, the first Hatcher-Wagoner invariant $\Sigma(f_\beta)=0\in \Wh_2(\pi_1M)$.
\end{corollary}

\begin{remark}
    We are not going into the definition of the first Hatcher-Wagoner invariant $\Sigma:\pi_0\mathcal{P}\to \Wh_2(\pi_1M)$(to see a clear definition, read \cite{singh2022pseudoisotopiesdiffeomorphisms4manifolds}), but briefly speaking, it records the handle slides that happen during $\{f_t: M\times I\to I\}\in \pi_1(\mathcal{F},\mathcal{E})$. So when there is only a single eye of $(n-i, n-i+1)$-handle pair, there's at most one handle for each index, so no handle slides at all time.
\end{remark}

In Hatcher's work \cite[Lemma 4.3 and Duality Formula 4.4]{AST_1973__6__1_0}, he showed that there exists an involution $\bar{\cdot}$ on both $\Wh_2(\pi_1 M)$ and $\Wh_1(\pi_1 M; \mathbb{Z}_2\times \pi_2M)$ such that for any pseudo-isotopy $f$ of $M$, let $\bar{f}=(F_f^{-1}\times \id)\circ R\circ f\circ R$ where $R: X\times I\to X\times I: (x,s)\to (x,1-s)$, then $\Sigma(\bar{f})=(-1)^n\bar{\Sigma}(f)$, $\Theta(\bar{f})=(-1)^n\bar{\Theta}(f)$. Let $\mathcal{P}_0$ be the connected component of $\id\in\mathcal{P}=\mathcal{P}(M)$, then $\Theta: \ker\Sigma \to \Wh_1(\pi_1 M; \mathbb{Z}_2\times \pi_2M)/\Theta(\mathcal{P}_0\cap \ker \Sigma)$ are well-defined with $\Theta(\mathcal{P}_0\cap \ker \Sigma)\subset\chi(K_3\mathbb{Z}[\pi_1 M])$ since $f\in \mathcal{P}_0\Leftrightarrow \bar{f}\in \mathcal{P}_0$. Therefore $\bar{\cdot}$ induces an involution on both $\Wh_2(\pi_1 M)$ and $\Wh_1(\pi_1 M; \mathbb{Z}_2\times \pi_2M)/\Theta(\mathcal{P}_0\cap \ker \Sigma)$, i.e. $\overline{\Theta(\mathcal{P}_0\cap \ker \Sigma)}=\Theta(\mathcal{P}_0\cap \ker \Sigma)$.

\begin{proposition} \label{rmk:orientation-reversing-cobordism-for-hatcher-wagoner-invariants}
   For a half-unknootted barbell $\beta$ in $M^n$, we have $\Sigma(g_\beta)=(-1)^{n+1}\bar{\Sigma}(f_\beta)=0$, $\Theta(g_\beta)=(-1)^{n+1}\bar{\Theta}(f_\beta)$.
\end{proposition}

\begin{proof}
    Note that in the proof of the above proposition, we obtain $\mathcal{H}_2=\{H_{2,t},t\in [0,3]$ which is a loop of $(n-i,n-i+1)$-handle constructions of $M\times I\}$ by doing dotted and 0-framed replacement \emph{pointwise} from $\mathcal{H}_0=\{H_{0,t},t\in S^1$ which is a loop of $(i-1,i)$-handle constructions of $M\times I\}$. Denote $Z_{\mathcal{H}_0}$(resp. $Z_{\mathcal{H}_2}$) as the corresponding cobordism from $S^1\times M$ to $Y_\mathcal{H}=M\times I/((x,0)\sim (f_{\mathcal{H}_0}(x),1))$ (resp. $Y_{\mathcal{H}_2}=M\times I/((x,0)\sim (f_{\mathcal{H}_2}(x),1))$). Since $f_{\mathcal{H}_0}=f_{\mathcal{H}_2}$, by \Cref{rmk:orientation-reversing-cobordism} we know that $Z_{\mathcal{H}_0}\cup \bar{Z}_{\mathcal{H}_2}=S^1\times X\times I$, thus $0=\Sigma(\mathcal{H}_0)+\Sigma(\bar{\mathcal{H}_2})=\Sigma(\mathcal{H}_1)+\Sigma(\bar{\mathcal{H}_2})$. Similarly, $\Theta(\mathcal{H}_1)+\Theta(\bar{\mathcal{H}_2})=0\in \Wh_1(\pi_1 M; Z_2\times \pi_2 M)/\Theta(\mathcal{P}_0\cap \ker \Sigma)$. Thus $\Sigma(g_\beta)=-\Sigma(\bar{f}_\beta)=(-1)^{n+1}\bar{\Sigma}(f_\beta)=0$ and $\Theta(g_\beta)=-\Theta(\bar{f}_\beta)=(-1)^{n+1}\bar{\Theta}(f_\beta)\in \Wh_1(\pi_1 M; \mathbb{Z}_2\times \pi_2M)/\Theta(\mathcal{P}_0\cap \ker \Sigma)$, therefore $\Theta(g_\beta)=(-1)^{n+1}\bar{\Theta}(f_\beta)\in \Wh_1(\pi_1 M; \mathbb{Z}_2\times \pi_2M)/\chi(K_3\mathbb{Z}[\pi_1 M])$. 
\end{proof}

\begin{example}
    We apply the above procedure to Budney and Gabai's barbell diffeomorphisms $\delta_k$, changing its corresponding loop of (1,2)-handle pair which we obtained in the last section into the loop of (2,3)-handle pair.

    Recall the loop of (1,2)-handle constructions for Cerf diagram resulting in $\delta_k$ (again we let $k=4$ for simplicity), as shown in \Cref{fig:element-in-pi1-representing-barbell}.

    \begin{figure}[!ht]
        \centering
        \includegraphics[width=0.7\textwidth]{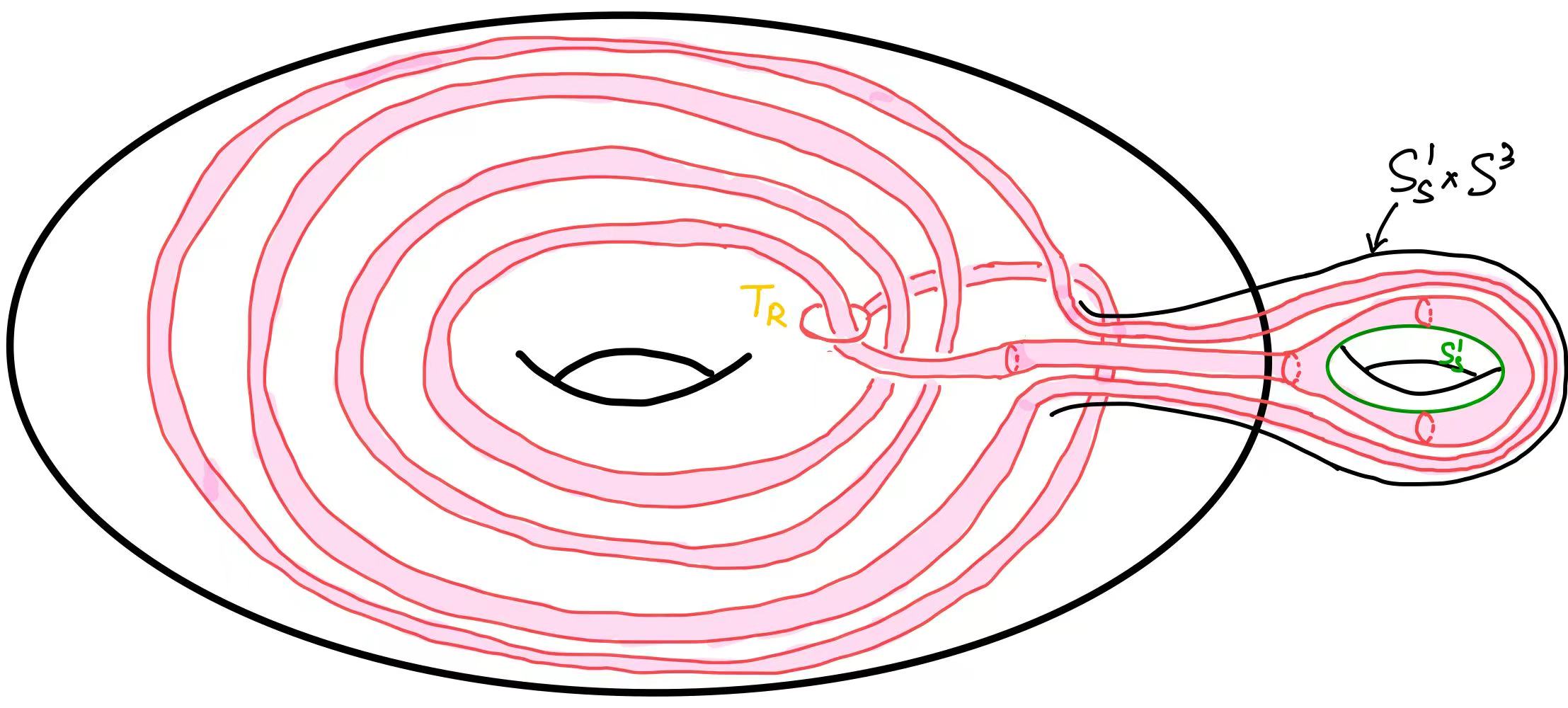}
        \caption{Element in $\pi_1(\Emb(S^1\times D^3, S^1\times D^3\# S^1\times S^3), *)$ which represents barbell $\delta_k$.}
        \label{fig:element-in-pi1-representing-barbell}
    \end{figure}

    Change it to the one-parameter dotted version, i.e.  $\beta_t^\bullet=\beta_0^\bullet: D^3\hookrightarrow S^1\times D^3$, $\gamma_t: S^1\hookrightarrow S^1\times D^3$, $\gamma_0=\gamma_1$ a standard meridian of $\beta_0=\partial\beta_0^\bullet$, $t\in [0,1]$ (see \Cref{fig:dotted-version-of-12-pair}).

    \begin{figure}[!ht]
        \centering
        \includegraphics[width=0.5\textwidth]{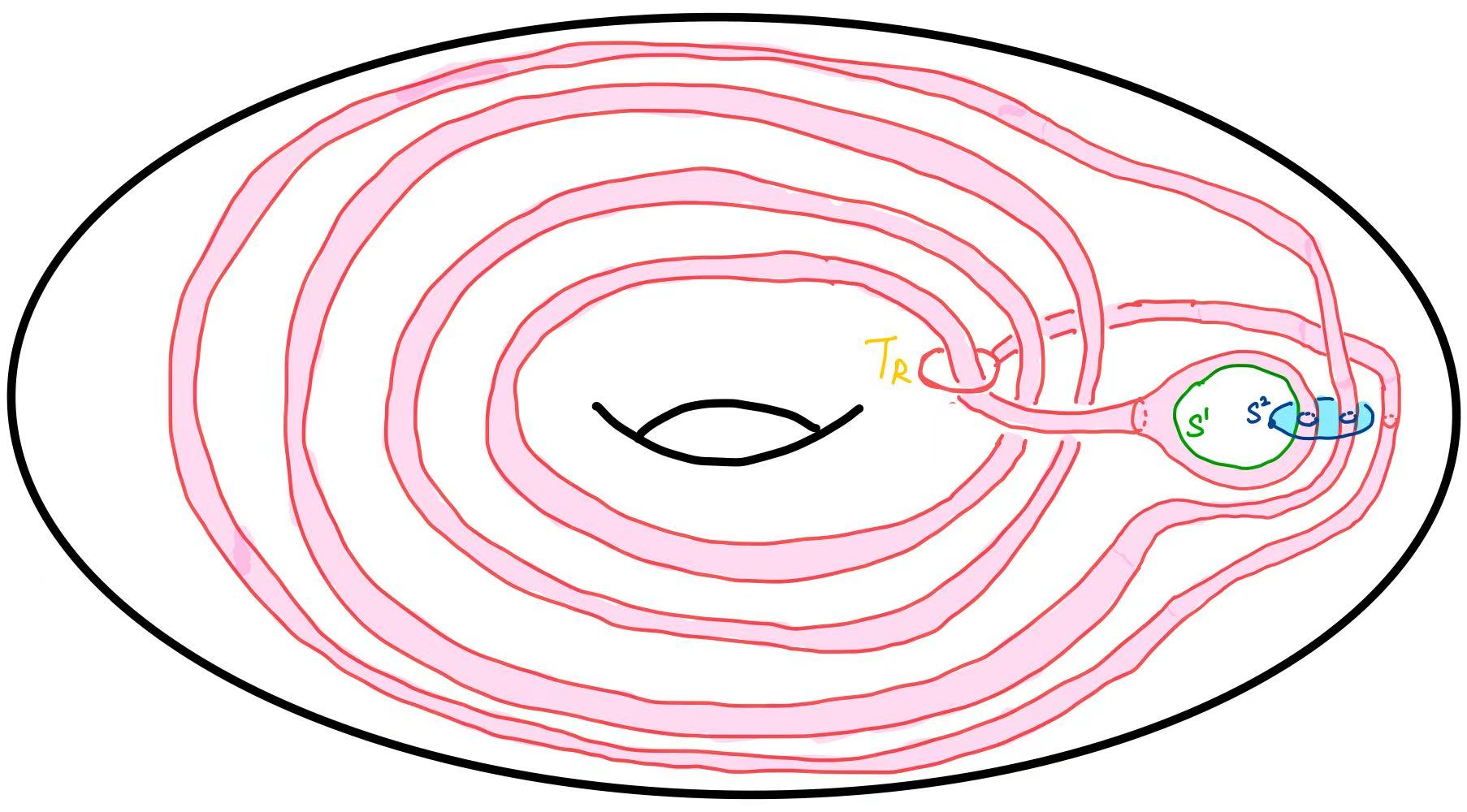}
        \caption{Dotted version of the loop of (1,2)-handle pair. The dark blue is the dotted $S^2=\partial \beta_0^\bullet$, the light blue region is $\beta_t^\bullet=\beta_0^\bullet: D^3\hookrightarrow S^1\times D^3$ which represents the 1-handle.}
        \label{fig:dotted-version-of-12-pair}
    \end{figure}

    Now following the constructions we made in \Cref{lem:one-parameter-version-of-dotted-and-0-framed-replacement}, we need to find $\gamma_t^\bullet, t\in [0,1]$. Let $\gamma_0^\bullet$ just be the standard normal disk of $\partial \beta_0^\bullet$, then $\gamma_t^\bullet$ is just dragging $\gamma_0^\bullet$ onto $T_R$ along the trajectory of $S^1$ in $T_R$, finally $\gamma_1^\bullet=\gamma_0^\bullet \#_{tube} R$ (see \Cref{fig:movement-of-disk}).

    \begin{figure}[!ht]
        \centering
        \includegraphics[width=0.45\textwidth]{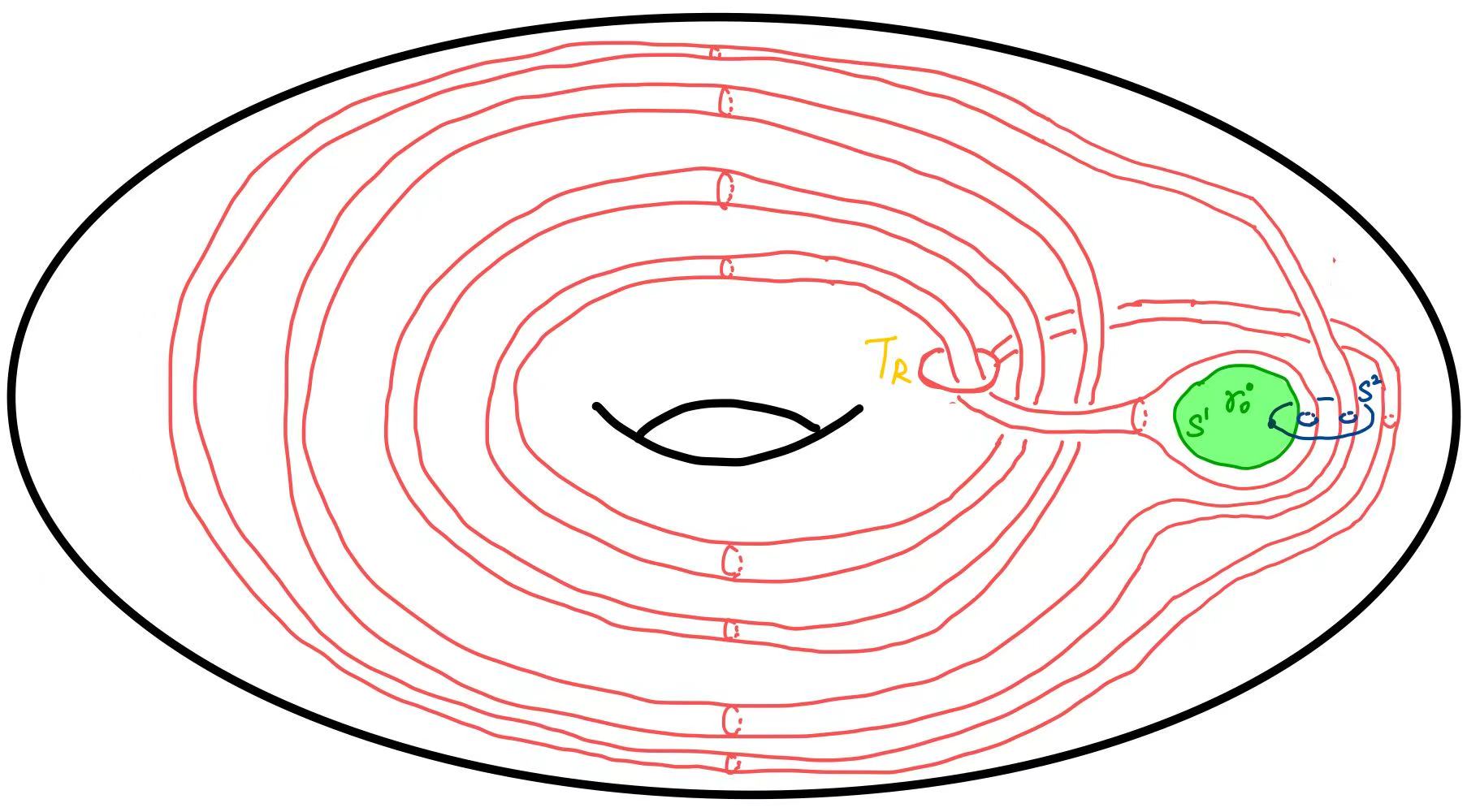}
        \includegraphics[width=0.45\textwidth]{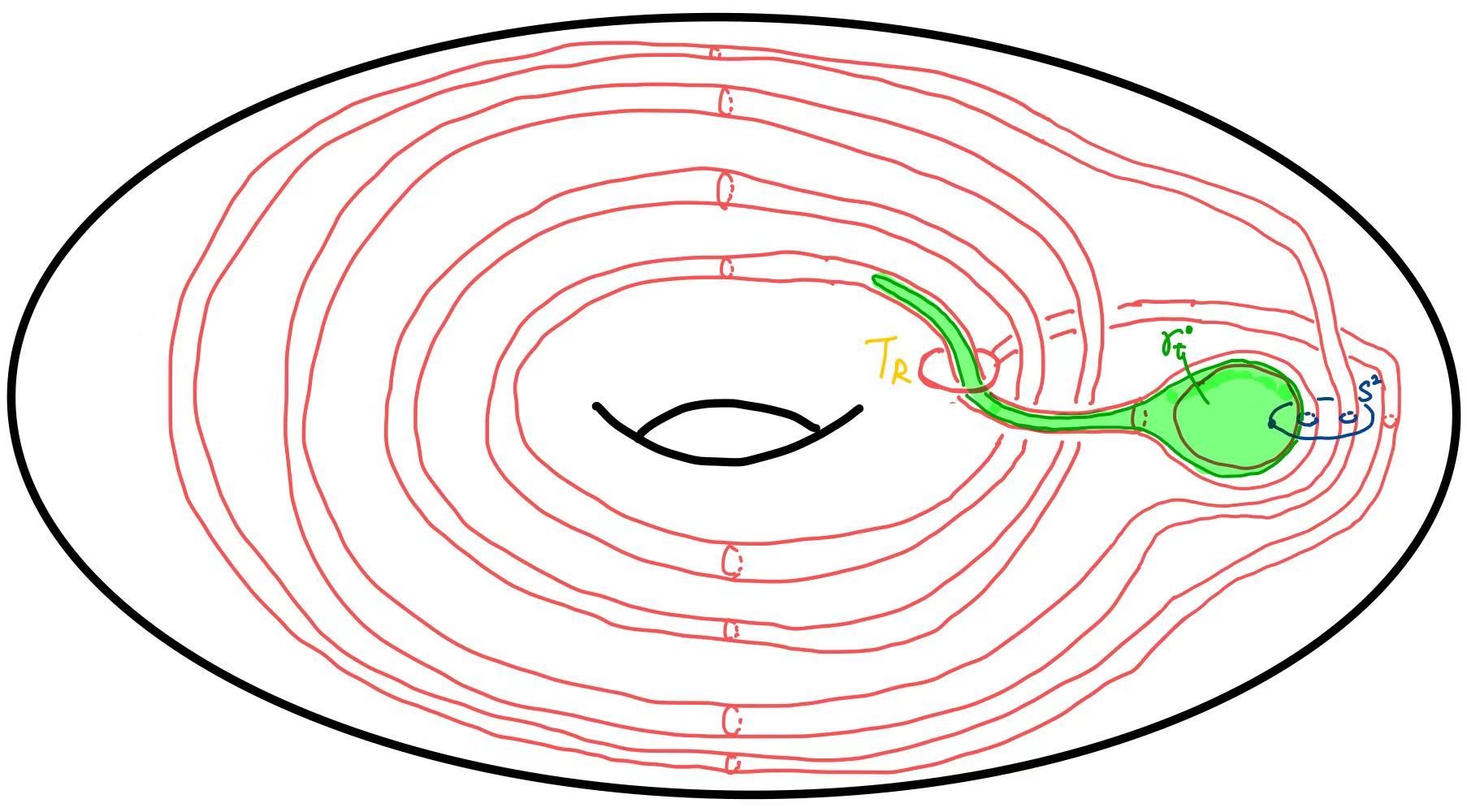}
        \includegraphics[width=0.5\textwidth]{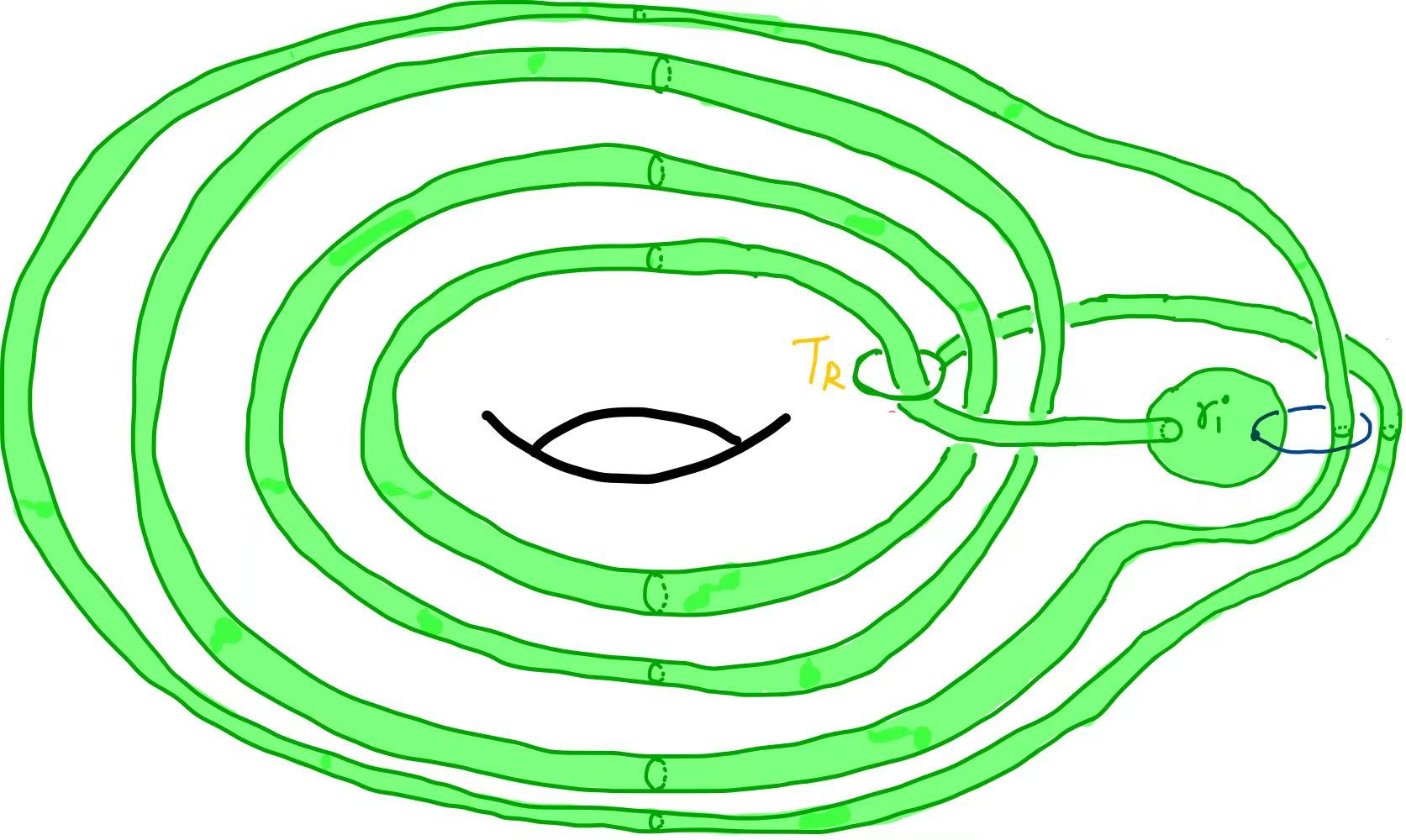}
        \caption{Movement of $\gamma_t^\bullet: D^2\hookrightarrow S^1\times D^3$. The figures illustrate $\gamma_0^\bullet, \gamma_t^\bullet, \gamma_1^\bullet$.}
        \label{fig:movement-of-disk}
    \end{figure}

    Now at $t\in[1,2]$, as the constructions in \Cref{lem:one-parameter-version-of-dotted-and-0-framed-replacement} read, consider $\phi_t: S^1\times D^3\to S^1\times D^3$ such that $\gamma_t^\bullet=\phi_t(\gamma_1^\bullet)=\gamma_{2-t}^\bullet$, $\beta_t^\bullet=\phi_t(\beta_1^\bullet=\beta_0^\bullet)$. In the interval $t\in [1,2]$, $\beta_t=\beta_0$ intersects $\gamma_t^\bullet$ at a single point. While $\beta_2^\bullet$ intersects $\gamma_2^\bullet$ at $S^1\cup I^1$, so when $t\in [2,3]$, $\beta_t$ shrinks in $\beta_2^\bullet$ to a standard normal sphere of $\gamma_2=\gamma_0$ so that $\beta_3=\beta_0$. Therefore, when during $t\in [2,3]$, $\beta_t$ will intersect $\gamma_t^\bullet=\gamma_0^\bullet$ first at 1 point, then at 3 points, and finally at 1 point, which will become essential to our calculations of Hatcher-Wagoner invariants. So in order to understand how $\beta_t,t\in [2,3]$ shrinks, we only need to understand $\beta_2^\bullet$:

    $\beta_2^\bullet$ is the isotopy image of $\beta_1^\bullet=\beta_0^\bullet$ when pulling $\gamma_1$ back to standard $\gamma_0$ along $R$, but this is just the \emph{backward barbell diffeomorphism} with data $\beta=(R_0, S, \gamma)$. Furthermore, the $\beta_1^\bullet$ is just a mid-ball of this implanted barbell (see \Cref{fig:barbell-equivalent-to-pulling-back})! Thus in \cite{Budney_2025} we know the embedded surgery description of the mid-ball after performing a barbell. In \Cref{fig:mid-ball-transformation} we draw the embedded surgery explicitly to get $\beta_2^\bullet$.

    Then performing \emph{one-parameter version of dotted and 0-framed replacement} we regard $\gamma_t^\bullet,t\in [0,3]$ as the dotted version of $h_{2,t},t\in [0,3]$ and regard $\beta_t,t\in [0,3]$ as the attaching spheres of $h_{3,t}, t\in[0,3]$. Thus we change the Cerf diagram of a (1,2)-handle pair to the Cerf diagram of a (2,3)-handle pair, both resulting in barbell $\delta_k$.

    \begin{figure}[!ht]
        \centering
        \includegraphics[width=0.5\textwidth]{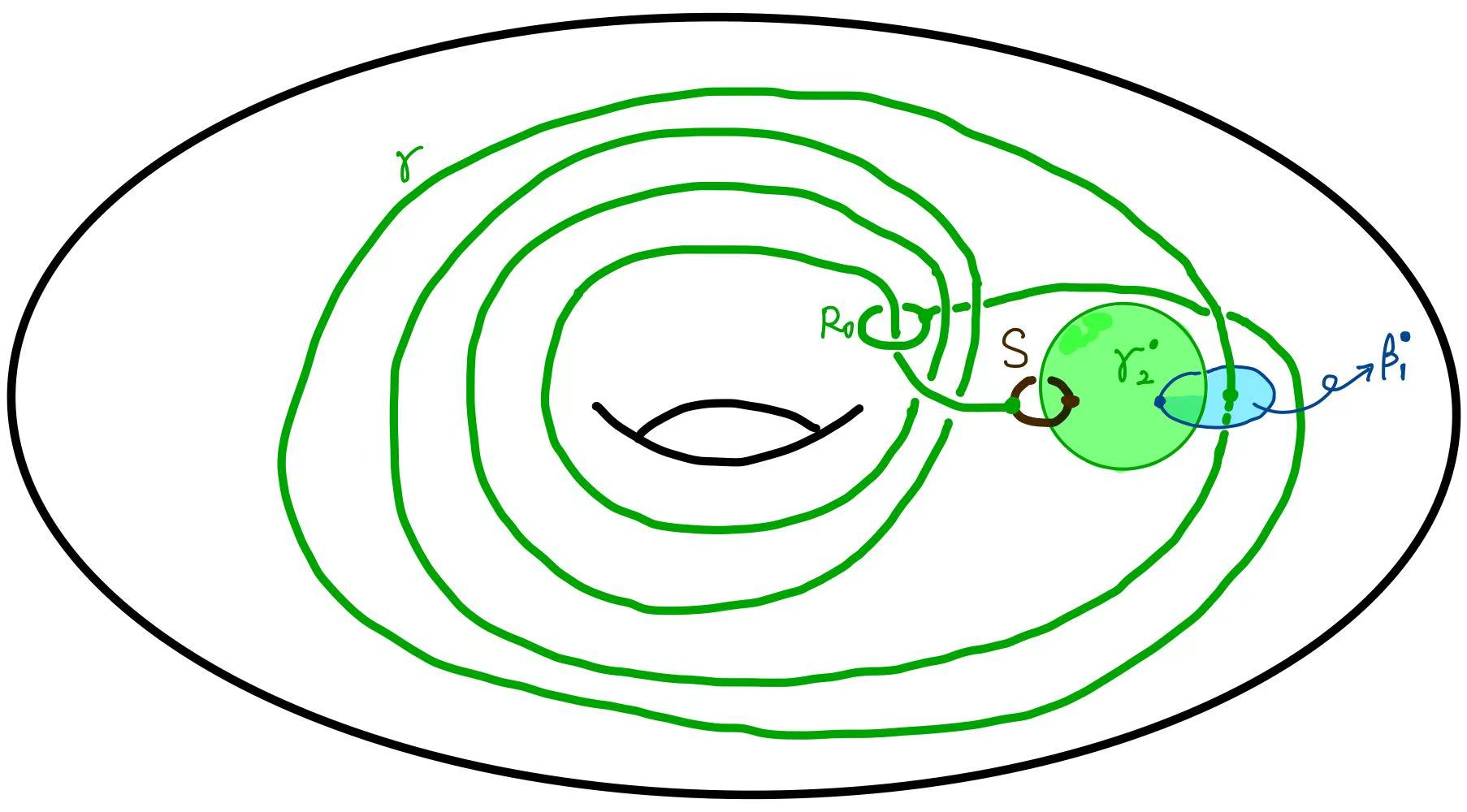}
        \caption{Pulling $\gamma_1$ back to $\gamma_0$(therefore pulling $\gamma_1^\bullet$ back to $\gamma_0^\bullet$) is equivalent to doing the barbell diffeomorphism $\beta=(R_0, S, \gamma)$.}
        \label{fig:barbell-equivalent-to-pulling-back}
    \end{figure}

    \begin{figure}[!ht]
        \centering
        \includegraphics[width=0.5\textwidth]{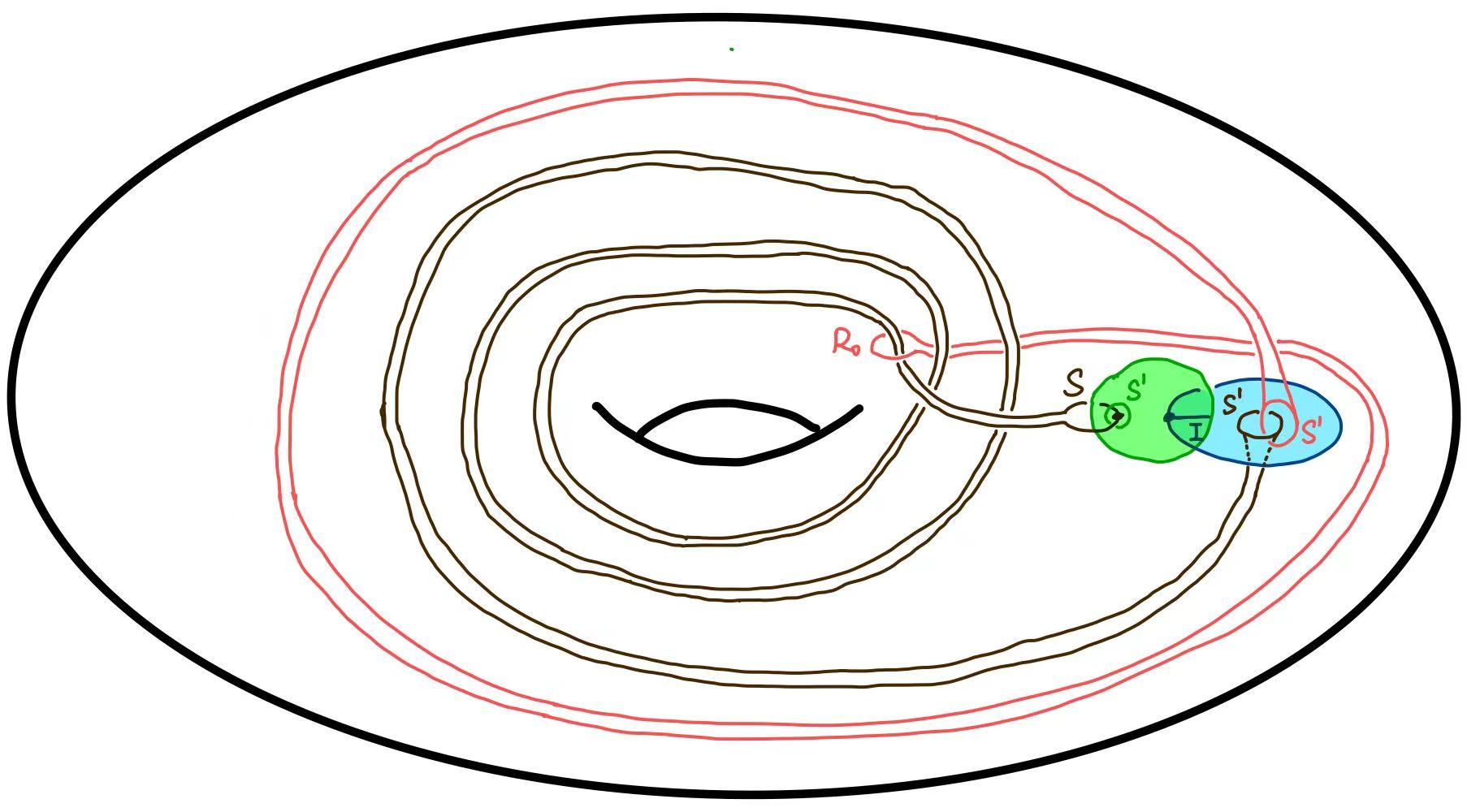}
        \caption{The mid-ball $\beta_1^\bullet$ becomes $\beta_2^\bullet=\partial(\beta_1^\bullet \cup h_{2,b} \cup h_{2, r})$ where $h_{2,b}$ is the 2-handle with attaching $S_b^1$ the brown one and core $D^2_b=\{S \text{ tube along the brown arc\}}$, $h_{2,r}$ is the 2-handle with attaching $S_r^1$ the red one and core $D^2_r=\{R \text{ tube along the red arc\}}$. As a result, $\beta_2^\bullet= (\beta_1^\bullet\setminus (S_b^1\times D^2\cup S_r^1\times D^2))\cup (D_b^2\times S^1\cup D_r^2\times S^1) $. Therefore, $\beta_2^\bullet\cap \gamma_2^\bullet=\text{the green }S^1 \text{ and the blue }I$ in the picture.}
        \label{fig:mid-ball-transformation}
    \end{figure}    
\end{example}
\section{Computing second Hatcher-Wagoner invariants for \texorpdfstring{$\beta_{2,n-2},n\geq 4$}{beta2,n-2,n>=4}}
\label{sec:computing-second-hatcher-wagoner-invariants-for-beta-2-n-2}

In this section, first we recall the definition for the second Hatcher-Wagoner invariant, namely, $\Theta: \ker\Sigma\to \Wh_1(\pi_1M, \mathbb{Z}_2\times \pi_2M)/\chi(K_3\mathbb{Z}[\pi_1 M])$. For simplicity, we restrict to the case when the Cerf diagram contains only a single eye of $(k,k+1)$-handle pair, $2\leq k\leq n-2$. For general definitions and well-definedness, see \cite{AST_1973__6__1_0} and \cite{singh2022pseudoisotopiesdiffeomorphisms4manifolds}.

Recall that $\Wh_1(\pi_1M, \mathbb{Z}_2\times \pi_2 M)=(\mathbb{Z}_2\times \pi_2 M)[\pi_1 M]/(\beta\cdot [1], \alpha\cdot [\sigma]-\alpha^\tau\cdot [\tau\sigma\tau^{-1}], \alpha,\beta\in \mathbb{Z}_2\times \pi_2 M, \tau,\sigma\in \pi_1 M)$ (see \cite{singh2022pseudoisotopiesdiffeomorphisms4manifolds} for details). 

To define $\Theta$, we need the following data: For $\{f_t,t\in I\}\in \pi_1(\mathcal{F},\mathcal{E})$ with Cerf diagram only a single eye of $(k,k+1)$-handle pair, suppose that the eye appears just before $t=\epsilon$ at the critical value $\frac{1}{2}$, and ends right after $t=1-\epsilon$ at the same critical value, then consider $V:=\bigcup_{t\in [\epsilon, 1-\epsilon]} f_t^{-1}(1/2)=M'\times[\epsilon, 1-\epsilon]$ where $M'=M\# S^k\times S^{n-k}$. For each $f_t, t\in [\epsilon,1-\epsilon]$, we call the belt $(n-k)$-sphere of the $k$-handle $B_t$, the attaching $k$-sphere of the $(k+1)$-handle $A_t$. Let $A:=\bigcup_{t\in[\epsilon, 1-\epsilon]} A_t\cong S^k\times I\subset V$, $B:=\bigcup_{t\in[\epsilon, 1-\epsilon]} B_t\cong S^{n-k}\times I\subset V$. Without loss of generality, we may assume that $A$ intersects $B$ transversely in $V$. Then we know that $A\cap B=I\sqcup_{i=1,...,k} S^1$ and denote the $i$-th $S^1$ as $S_i^1$. Now, choose a base point $*\in X\times 0\subset X\times I$, a fixed $*'\in I\subset A\cap B$ and a path $\delta_0\subset X\times I:$ $*\to *'$. For each $i$, choose a point $p_i\in S_i^1$, and two paths: $\delta^A_i\subset A$: $*'\to p_i$,
$\delta^B_i\subset B: *'\to p_i$. \emph{Denote} $\gamma_i:=\delta^B_i*(\delta^A_i)^{-1}\in \pi_1(X\times I, *')$ and $\gamma_i^*:=\gamma_i^{\delta_0}\in \pi_1(X\times I,*)$ by pulling $\gamma_i$ from the base point $*'$ back to the base point $*$ along $\delta_0$. Also note that both $A$ and $B$ are simply connected, then choose $D_i^A\looparrowright A$ such that $\partial D^A_i=S^1_i$, $D_i^B\looparrowright B$ such that $\partial D^B_i=S_i^1$, so \emph{denote} $\beta_i:=[D^A_i\cup D^B_i]\in \pi_2(X\times I, p_i)$  and $\beta_i^*:=\beta_i^{\delta_0*\delta_i^B}\in \pi_2(X\times I,*)$ by pulling $\beta_i$ from base point $p_1$ back to base point $*$ along $\delta_0*\delta_i^B$. 

To define the $\mathbb{Z}_2$ component, we consider two framings of the same $\mathbb{R}^{n-k}$-vector bundle on $S_i^1$. Consider $\nu (S_i^1, B)$(the normal bundle of $S_i^1$ in $B$), since $A$ intersects $B$ transversely in $V$, then $E=\nu (S_i^1, B)=\nu (A,V)|_{S_i^1}$ is a $(n-k)$-dimensional trivial vector bundle over $S_i^1$. We define two framings on $E$: $e_{i,B}$ is the framing naturally induced by $D^B_i\looparrowright B$, $e_{i,A}$ is induced by the attaching data for the 3-handle, namely, $A_t\times D^{n-k}\hookrightarrow M'=M'_t$. Then $e_{i,A}-e_{i,B}\in \pi_1(SO_{n-k})=\mathbb{Z}\text{ or }\mathbb{Z}_2$, we define $s_i:=e_{i,A}-e_{i,B} \text{ mod }2$. Thus for every $i$, we have defined $s_i\in \mathbb{Z}_2$.

Using all those notations above, we can define the second Hatcher-Wagoner invariant $\Theta$ for $F:=\{f_t,t\in I\}\in \pi_1(\mathcal{F},\mathcal{E})$:
\begin{definition}
    For $F:=\{f_t,t\in I\}\in \pi_1(\mathcal{F},\mathcal{E})$ with Cerf diagram only a single eye of $(k,k+1)$-handle pair with $2\leq k\leq n-2$, define $\Theta(F):=\sum_{i=1,...,k} (s_i, \beta_i^*)\cdot [\gamma_i^*]$. 
\end{definition}

\begin{example} \label{eg:calculating-second-hatcher-wagoner-invariant-for-delta-k}
    Now we calculate the second Hatcher-Wagoner invariant for implanted barbell diffeomorphism $\delta_k$: 
    
    Using results in last section, by projecting the movements to a single $S^1\times D^3\# S^1\times S^3$, and then using the dotted version and \emph{one-parameter version of dotted and 0-framed replacement} in $S^1\times D^3$, we use $\beta_t$ and $\gamma_t^\bullet$ in $S^1\times D^3$ to represent $A_t, B_t$ in $S^1\times D^3\# S^1\times S^3$. We see that the attaching sphere $A_t=\beta_t$ intersects $B_t=\gamma_t^\bullet\cup D^2$ at a single point for $t\in[0,2]$ and $\bigcup _{t\in[2,3]} A_t\cap B_t=\beta_2^\bullet \cap \gamma_2^\bullet=I\sqcup S^1$ with $I$ the blue arc in \Cref{fig:mid-ball-transformation}, and $S^1$ the green one in \Cref{fig:mid-ball-transformation}. Thus from \Cref{fig:mid-ball-transformation} we see $\gamma^*=\gamma_1^*=(k-1)^*=k-1\in \pi_1(S^1\times D^3, *)=\mathbb{Z}$. 

    To calculate $\beta^*$, we need to find $D^B\looparrowright \gamma_2^\bullet \times I $ and $D^A\looparrowright \beta_2^\bullet $ which bound $\partial D^A=\partial D^B=S^1$. By projecting to a single $S^1\times D^3$ for $t\in[2,3]$, we know that $\gamma_t^\bullet=\gamma_2^\bullet$ is the standard green one in \Cref{fig:mid-ball-transformation} which won't change, so it's enough to find $D^B\looparrowright \gamma_2^\bullet $. For $D^B$, we have an immediate choice in \Cref{fig:mid-ball-transformation}, which is just the small disk in green $\gamma_2^\bullet$ which bounds the green $S^1$. For $D^A$, we pull the green $S^1$, which is the meridian of the brown tube, along the black tube back to the meridian of brown $S^1\subset $ the blue $D^3$, which is just a parallel of red $S^1\subset$ the blue $D^3$, then capping with the red $D^2$ which is the red long tube connected with $R_0$. In all, $\beta^*=\beta_1^*=(R_0^\gamma)^*\in \pi_2(S^1\times D^3, *)=0$, here $\gamma$ is the arc in the definition of barbell, which is the arc connecting $S$ with $R_0$.

    Now we calculate $s=s_1$. For $s_1^A=\text{attaching framing of }\nu (A, V)|_{t\in [2,3], S^1\subset A}$, from the dotted version we know that $\{A_t|t\in [2,3]\}\subset \beta_2^\bullet$ is 0-framed, thus $s_1^A$ has a section which is the normal bundle of $\beta_2^\bullet$ in $S^1\times D^3$, which is exactly a normal section of $\nu (S^1, \gamma_2^\bullet)\subset \nu (S^1, \gamma_2^\bullet \times I=B|_{t\in [2,3]})$ when restricted to the green $S^1$. Thus $s^A=s^B$, $s=0$.

    Thus we obtain that $\Theta(\delta_k)=(0, (R_0^\gamma)^*)[k-1]=(0,0)[k-1]=0\in \Wh_1(\mathbb{Z}, \mathbb{Z}_2\times 0)$.

\end{example}

The calculation can be generalized similarly to any half-unknotted implanted barbell $\beta=\beta_{i,n-i}=(R_0, S, \gamma)$ with $S$ unknotted, in this section we restrict ourselves to the case of $\beta_{2,n-2}$ for $n\geq 4$, i.e. let $i=2$ in the remainder of this section. 

After doing an isotopy of $M$ we can put $S=\partial \beta_0^\bullet$ into the standard position and use strategies in \cref{sec:changing-from-i-1-i-handle-pair-to-n-i-n-i-1-handle-pair} to find a Cerf diagram containing a single eye of $(n-2,n-1)$-handle pair with $\gamma_t^\bullet, \beta_t, t\in [0,3]$. Note that for $t\in [0,2]$, $\beta_t$ intersects $\gamma_t^\bullet$ at a single point, so as the example suggests, the second Hatcher-Wagoner invariant can be calculated just using $\beta_2^\bullet$ and $\gamma_2^\bullet=\gamma_0^\bullet$. Assume we have obtained $\beta_2^\bullet\subset M$, $\beta_2^\bullet \cap \gamma_2^\bullet=I\sqcup _{i=1,...,k} S^1$, then the same argument of the example works in the general case, i.e. using notations above, we have
$$\Theta(F)=\sum_{i=1,...,k}(s_i, \beta_i^*)\cdot [\gamma_i^*]$$
where for each $i$, $\gamma_i=\delta^B_i*(\delta^A_i)^{-1}$ with $\delta^A_i\looparrowright \beta_2^\bullet$, $\delta^B_i\looparrowright \gamma_2^\bullet$ connecting $*'$ and $p_i\in S^1_i$, $\beta_i=D^A_i\cup D^B_i$ with $D_i^A\looparrowright \beta_2^\bullet$ and $D^B_i\looparrowright \gamma_2^\bullet$, $s_i=e_{i,A}-e_{i,B} \text{ mod }2$ with $E=\nu (S_i^1, \gamma_2^\bullet \times I)=\nu (\beta_2^\bullet , M\times I)|_{S_i^1}$. $\nu (S_i^1, \gamma_2^\bullet \times I)=\nu (S^1_i, \gamma_2^\bullet)\oplus E'=\nu (S_i^1, D_i^B)\oplus E'$, $\nu (\beta_2^\bullet, M\times I)|_{S_i^1}=\nu (\beta_2^\bullet, M)\oplus E''$ where $E',E''$ are two $1$-dimensional trivial bundle over $S_i^1$. $e_{i,B}$ is the framing induced by $\nu (S_i^1, \gamma_2^\bullet)$ and $e_{i,A}$ is the framing induced by $\nu (\beta_2^\bullet,M)|_{S_i^1}$. Since $\beta_2^\bullet$ intersects $\gamma_2^\bullet$ transversely at $S_i^1$ in $M$ and $codim_{\gamma_2^\bullet} S_i^1=codim_{\beta_2^\bullet} M=1$, we have $\nu (S_i^1, \gamma_2^\bullet)=\nu (\beta_2^\bullet,M)|_{S_i^1}$; thus $s_i=e_{i,A}-e_{i,B}=0$.

To calculate $\beta_i^*$ and $\gamma_i^*$, we need a smarter change: It can be very hard to visualize $\beta_2^\bullet$, especially when $R_0$ goes through $\beta_0^\bullet$ in a strange way ($R_0\cap \beta_0^\bullet$ may be knotted in $D^{n-1}=\beta_0^\bullet$). But we know that there is an isotopy $\phi_2: M\to M$ which is isotopic to identity such that $\phi_2(\beta_1^\bullet=\beta_0^\bullet)=\beta_2^\bullet$, $\phi_2(\gamma_1^\bullet)=\gamma_2^\bullet$. This means we can pullback every data we need to $\beta_1^\bullet$ and $\gamma_1^\bullet$. Then it will be easier to visualize in the given implanted barbell $\beta$. 

Note that $\beta_1^\bullet=\beta_0^\bullet$ is the given unknotted ball which bounds $S$, i.e. $S=\partial\beta_1^\bullet$. Also recall that $\gamma_1^\bullet=\gamma_0^\bullet \text{ tube}_\gamma \text{ }R_0$. Without loss of generality, we may assume that $\beta_0^\bullet \cap R_0=\sqcup_{i=1,...,k} S_i^1$ and $\text{int}(\beta_0^\bullet)\cap \gamma=\emptyset  $ (by finger-pushing $R_0$ along $\gamma$ without reaching $S$, such that the arc $\gamma$ of the implanted barbell $\beta$ can be made arbitrarily short). Let $*_0=S\cap \gamma$ be the base point, note that $*_0$ will not move during the whole time $t\in [0,3]$ (in \Cref{fig:mid-ball-transformation}, $*_0$ is the left vertex of the blue $I$). For each $i$, find a path $\delta^B_i\subset \beta=(R_0, S,\gamma)$ which is a path from $*_0$ to $p_i\in S_i^1$. Also, for each $i$, $S_i^1$ divides $R_0$ into two embedded disks $D_i$ and $D_i'$, where $D_i'$ is the one connected to the arc $\gamma$. Let $D^B_i=D_i$. Now comes the following main theorem:

\begin{theorem} \label{thm:general-calculation-of-second-hatcher-wagoner-invariant-for-beta-2-n-2}
    For a half-unknotted implanted barbell $\beta=\beta_{2,n-2}=(R_0, S, \gamma)$ with $S=\partial \beta_0^\bullet$ where $\beta_0^\bullet: D^{n-1}\hookrightarrow M$, by finger-pushing $R_0$ along the arc, we can make $\gamma$ short enough such that $\text{int}(\beta_0^\bullet)\cap \gamma =\emptyset$. Now suppose that $\beta_0^\bullet \cap R_0=\bigsqcup _{i=1}^k S_i^1$. Choose $p_i\in S_i^1$ and let $*_0=\gamma\cap S$ be the base point. For each $i$, find a path $\delta^B_i\subset \beta=(R_0, S,\gamma)$ which is a path from $*_0$ to $p_i\in S_i^1$. Also, for each $i$, $S_i^1$ divides $R_0$ into two embedded disks $D_i$ and $D_i'$, where $D_i'$ is the one connected to the arc $\gamma$. Let $D^B_i=D_i$ (see \Cref{fig:general-barbell} for an illustration). Then the $f_\beta\in \pi_0\mathcal{P}$ we constructed in \Cref{cor:changing-from-i-1-i-handle-pair-to-n-i-n-i-1-handle-pair} which results in the implanted barbell diffeomorphism with respect to $\beta$ satisfies:
    
    $$\Theta(f_\beta)=\sum_{i=1,...,k} (0, [D^B_i]^{\delta_i^B})\cdot  [\delta_i^B]$$
    Here we identify $\pi_i(M,*_0)$ with $\pi_i(M, \beta_0^\bullet)$ so that $[D_i]\in \pi_2(M,\beta_0^\bullet)$ and $\delta_i^B\in \pi_1(M,\beta_0^\bullet)$. 
\end{theorem}

\begin{proof}
    In the above discussion we already showed that $s_i=0\in \mathbb{Z}_2$ for all $i$, we only need to show that $[D_i]^{\delta_i^B}=\beta_i^*$ and $[\delta_i^B]=\gamma_i^*$. But since we choose $*_0\in M$ which won't move during the whole procedure, then $*=(*_0,0)\in M\times 0$, $*'=(*_0,1/2)\in (M')_t=M\# S^{n-2}\times S^2$. So for the calculation, we can choose $\beta_i^*$ and $\gamma_i^*$ based at $*_0\in M$. Then by definition and by changing data to $\gamma_1^\bullet=\gamma_0^\bullet \text{ tube}_\gamma \text{ }R_0$ and $\beta_1^\bullet=\beta_0^\bullet$, we pick $\delta_i^A\subset \beta_1^\bullet \text{ connecting }$ $*_0$ to $p_i$ and pick $D_i^A\looparrowright \beta_1^\bullet$ with $\partial D_i^A=S_i^1$. Then by definition $\beta_i^*=[\phi_2((D_i^A\cup D_i^B)^{\delta_i^B})]=[(D_i^A\cup D_i^B)^{\delta_i^B}]\in \pi_2(M, *_0)$, since by \Cref{lem:one-parameter-version-of-dotted-and-0-framed-replacement}, $ \phi_t: M\to M,t\in[1,2]$ is a one-parameter family of isotopies of $M$ with $\phi_1=\id$ and $\phi_t(*_0)=*_0,\forall t\in [1,2]$. But $\beta_0^\bullet$ is an embedded contractible $D^{n-1}$, so $\beta_i^*=[D_i^B]^{\delta_i^B}$. Also for the same reason, $\gamma_i^*=\delta_i^B*(\delta_i^A)^{-1}\in \pi_1(M, *_0)$ which is $\delta_i^B\in \pi_1(M, \beta_0^\bullet)$. 
\end{proof}

\begin{example}
    Using \Cref{fig:movement-of-disk} (where we draw $\gamma_1^\bullet$ and $\beta_1^\bullet=\beta_0^\bullet$ explicitly when $\beta=\delta_k$ in $S^1\times D^3$) and results in \Cref{thm:general-calculation-of-second-hatcher-wagoner-invariant-for-beta-2-n-2}, one can easily see that for the implanted barbell $\delta_k$ in $S^1\times D^3$, $\Theta(f_{\delta_k})=(0,0)\cdot [k-1]=0$, which coincides with \Cref{eg:calculating-second-hatcher-wagoner-invariant-for-delta-k}.
\end{example}

\section{Generalizations to immersed barbell diffeomorphisms}
\label{sec:generalizations-to-immersed-barbell-diffeomorphisms}

The main goal of this section is to generalize implanted $(2,2)$-barbells $\beta_{2,2}$ in $M^4$ to \emph{immersed $(2,2)$-barbells}, and then define the corresponding \emph{immersed barbell diffeomorphisms}. Following what we did in the above sections, we similarly construct $g_\beta\in \pi_0\mathcal{P}$ with its Cerf diagram being a single eye of (1,2)-handle pair and $f_\beta\in \pi_0\mathcal{P}$ with its Cerf diagram being a single eye of (2,3)-handle pair, both resulting in that immersed barbell diffeomorphism. After that we also compute the Hatcher-Wagoner invariants for $f_\beta$, thus also for $g_\beta$, by \Cref{rmk:orientation-reversing-cobordism-for-hatcher-wagoner-invariants}. By this small generalization and the theorem in the last section, we show that $\langle\Theta(f_\beta),\Theta(g_\beta),\beta \text{ half-unknotted immersed barbells}\rangle$ contains all elements in $\Wh_1(\pi_1 M; \mathbb{Z}_2\times \pi_2 M)/\chi(K_3[\mathbb{Z}\pi_1M])$ realizable by Singh's procedure in \cite{singh2022pseudoisotopiesdiffeomorphisms4manifolds}. Explicitly, 
$$\langle\Theta(f_\beta),\Theta(g_\beta) |\beta \text{ half-unknotted immersed barbells}\rangle\supset \Xi=\langle (s,\sigma)\cdot [\gamma]|s=0 \text{ or }w_2^M(\sigma)\neq 0\rangle$$ Also, we obtain parallel results for higher dimensional manifolds. 

\begin{definition}
    An \emph{immersed barbell} $\beta=(R_0,S,\gamma)$ consists of a framed immersed 2-sphere $R_0$ in $M$, a framed embedded 2-sphere $S$ in $M$, and a framed embedded arc $\gamma$ in $M$ connecting $R_0$ and $S$. By finger-pushing $R_0$ along $\gamma$ into $S$ to get another framed immersed 2-sphere $R$ in $M$ such that $R\cap S=2\text{ pts}$, we do surgery along $\nu S$ to get $M_{\nu S}$, then after surgery we get $T_R=(R\setminus D^2\times \partial I)\cup S^1\times I $ which is a framed immersed torus. Regarding $T_R$ as an $S^1$-parameterised immersions of $S^1$ in $M$, denoted as $\gamma_t,t\in S^1$, we slightly change the parametrization of $T_R$ such that all pieces $\gamma_t,t\in S^1$ are actually embeddings of $S^1$. As in \cref{sec:twin-twists-and-relation-with-barbell-diffeomorphisms}, $\nu T_R$ thus defines an element $[\nu T_R]\in \pi_1(\Emb(\nu S^1, M_{\nu S}), *)$. We define the corresponding \emph{immersed barbell diffeomorphism} as $\text{ps}_{\nu S}[\nu T_R]$, which was defined in \Cref{def:parameterised-surgery}. 
\end{definition}

In the same manner as \Cref{prop:twin-twist-is-pseudo-isotopic-to-identity-when-s-unknotted} we can show that, when $S$ is unknotted, there is a loop $g_\beta\in \pi_0\mathcal{P}$ of handle constructions of $M\times I$ containing a single (1,2)-handle pair which results in that immersed barbell diffeomorphism. And by doing \emph{one-parameter version of dotted and 0-framed replacement} stated in \Cref{lem:one-parameter-version-of-dotted-and-0-framed-replacement} we change the loop of (1,2)-handle pair $g_\beta$ to a loop of (2,3)-handle pair $f_\beta\in \pi_0\mathcal{P}$, thus we can compute the Hatcher-Wagoner invariants for $f_\beta$. 

Throughout the process, the only difference from the embedded case is that, to obtain $\gamma_t^\bullet,t\in [0,1]$ with $\gamma_t=\partial\gamma_t^\bullet$ which is the given loop of attaching spheres of the 2-handles, in the previous embedded case, we naturally choose $\gamma_t^\bullet=\gamma_0^\bullet\cup _{s\in [0,t]} \gamma_s$ (that is, pull $\gamma_0^\bullet$ along $\gamma_t,t\in [0,1]$), but now since $T_R$ is an immersed torus, it may have double self-intersection points, then there will be $t_0<t_1$ such that $p:=\gamma_{t_0}\cap \gamma_{t_1}\neq \emptyset$ (by a reparametrization of $T_R$, which won't change the resulting $g_\beta\in \pi_0\mathcal{P}$, one can assume that for any $t_0\neq t_1$, there's at most one intersection point between $\gamma_{t_0}$ and $\gamma_{t_1}$). In this case when the intersection point is reached for the second time at $t=t_1$, one must finger-push the small intersection disk $D^2=\{\text{normal disk of }T_R \text{ at point } p\}\subset \gamma_t^\bullet, t\in (t_0,t_1)$ towards the trajectory of $\gamma_t$ to get an embedded $\gamma_t^\bullet, t>t_1$. See \Cref{fig:finger-push-ensure-immersion} for an illustration.

\begin{figure}[!ht]
    \centering
    \includegraphics[width=0.5\textwidth]{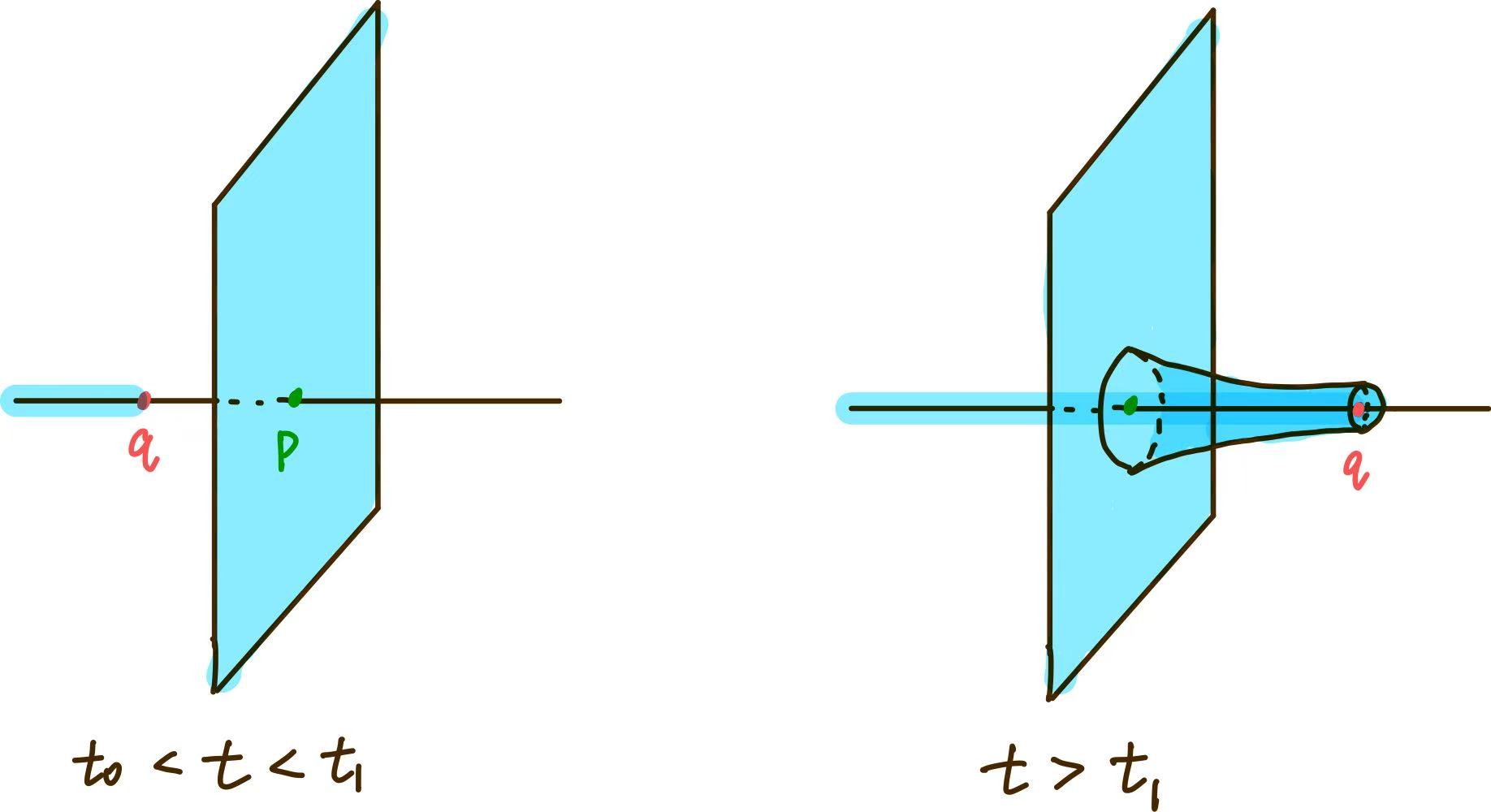}
    \caption{Here we draw a local region of $T_R$ near the self-intersection point in an $\mathbb{R}^3$-slice. $p=\gamma_{t_0}\cap \gamma_{t_1}$, $q\in \gamma_t$, and the blue region is the local part of $\gamma_t^\bullet$ in an $\mathbb{R}^3$-slice. When $t>t_1$, the blue disk $\subset \gamma_t^\bullet$ will be finger-pushed along $\gamma_t$ to ensure $\gamma_t^\bullet$ is embedded.}
    \label{fig:finger-push-ensure-immersion}
\end{figure}

After understanding that we can deduce the following theorem, which is a generalization of \Cref{thm:general-calculation-of-second-hatcher-wagoner-invariant-for-beta-2-n-2}:

\begin{theorem} \label{thm:general-calculation-of-second-hatcher-wagoner-invariant-for-immersed-beta-2-n-2}
    For a half-unknotted immersed barbell $\beta=(R_0,S,\gamma)$ with $\beta_0^\bullet: D^3\hookrightarrow M, S=\partial\beta_0^\bullet $, perturb self-intersections of $R_0$ away from $\beta_0^\bullet$. By finger-pushing $R_0$ along the arc, we can make $\gamma$ short enough such that $\text{int}(\beta_0^\bullet)\cap \gamma =\emptyset$. Now suppose that $\beta_0^\bullet \cap R_0=\sqcup _{i=1,...,k} S_i^1$. Choose $p_i\in S_i^1$ and let $*_0=\gamma\cap S$ be the base point. For each $i$, find a path $\delta^B_i\subset \beta=(R_0, S,\gamma)$ which is a path from $*_0$ to $p_i\in S_i^1$. Also, for each $i$, $S_i^1$ divides $R_0$ into two immersed disks $D_i$ and $D_i'$, where $D_i'$ is the one connected to the arc $\gamma$. Let $D^B_i=D_i$. Then the $f_\beta\in \pi_0\mathcal{P}$ we constructed which results in the immersed barbell diffeomorphism with respect to $\beta$ satisfies:
    
    $$\Theta(f_\beta)=\sum_{i=1,...,k} (0, [D^B_i]^{\delta_i^B})\cdot  [\delta_i^B]$$
    Here we identify $\pi_i(M,*_0)$ with $\pi_i(M, \beta_0^\bullet)$ so that $[D_i]\in \pi_2(M,\beta_0^\bullet)$ and $\delta_i^B\in \pi_1(M,\beta_0^\bullet)$.    
\end{theorem}

\begin{proof}
    By finger-pushing $R_0$ along $\gamma$ into $S$ to get $R$, we then attach a 1-handle, in the dotted version, we push $\beta_0^\bullet$ into $M\times I$ and cut out a neighborhood of it. On the top we just do surgery on $\nu S$ and get $T_R$ which is an immersed framed torus in $M\# S^1\times S^3$. On the left side of \Cref{fig:immersed-torus-and-standard-gamma} we draw the immersed $T_R$ in the dotted version with standard $\gamma_0^\bullet$ in green. Thus we get initial data $\gamma_0^\bullet$, $\{\gamma_t,t\in [0,1]\}\in \pi_1(\Emb(\nu S^1, M\# S^1\times S^3),*)$ and $\beta_t^\bullet=\beta_0^\bullet, t\in [0,1]$. To use \emph{one-parameter version of dotted and 0-framed replacement} and calculate $\Theta(f_\beta)$ we only need to find $\gamma_t^\bullet, t\in [0,1]$. 
    
    As in the process from the last section, we have $s_i=0$, $\beta_i^*$ and $\gamma_i^*$ can be calculated in $\gamma_1^\bullet$ and $\beta_1^\bullet=\beta_0^\bullet$. By the discussions just above the theorem, we isotope $\gamma_0^\bullet$ along $\gamma_t$ to $\gamma_t^\bullet$. Whenever we reach a self-intersection point $p=\gamma_{t_1}(v)$ ($\gamma_{t}: S^1\to T_R\subset M: w\to \gamma_{t}(w)$) the second time, the normal disk at the intersection point needs to be stretched along $l_p:=\cup_{t\in [t_1, 1]} \gamma_t(v)$. Thus in the case $T_R$ has just a single self-intersection point, $\gamma_1^\bullet=(\gamma_0^\bullet\text{ tube}_\gamma R_0-D^2) \cup S^1\times l_p \cup D_1^2 $ where $S^1\times l_p=SN(T_R)|_{l_p}$ (for $X$ immersed in $M$, we use $SN(X)$ to denote the sphere bundle of $N(X)$ which is the normal bundle of $X$ in $M$), that is, each $S^1$ is a meridian sphere, and $D_1^2\subset S^2=SN(\gamma_1)|_{\gamma_1(v)}$ which bounds $S^1=SN(T_R)|_{\gamma_1(v)}$ such that $D_1^2\cap \gamma_0^\bullet=\emptyset$ (see \Cref{fig:immersed-torus-and-standard-gamma} for an illustration). 

    When there are more double self-intersection points, by a nice reparametrization of $T_R$, we may assume the self-intersection points are $p_j=\gamma_{t_{0,j}}(w_j)\cap \gamma_{t_{1,j}}(v_j),t_{0,j}<t_{1,j}, j=1,...,n$ and $\{v_j\}$ are distinct from each other. Then let $l_{p_j}:=\cup_{t\in [t_{1,j}, 1]} \gamma_t(v_j)$. Then $\gamma_1^\bullet=(\gamma_0^\bullet\text{ tube}_\gamma R_0-\sqcup _{j=1,...,n}D^2) \cup_{j=1,...,n} (S^1\times l_{p_j}) \cup_{j=1,...,n} D_{1,n}^2$ where $S^1\times l_{p_j}=SN(T_R)|_{l_{p,j}}$ and $D_{1,j}^2\subset S^2=SN(\gamma_1)|_{\gamma_1(v_j)}$ which bounds $S^1=SN(T_R)|_{\gamma_1(v_j)}$ such that $D_j^2\cap \gamma_0^\bullet=\emptyset$.

    Then consider $\gamma_1^\bullet\cap \beta_1^\bullet$, it's $R\cap \beta_1^\bullet$ which are $k$ disjoint $S^1$, namely $S_i^1, i=1,...,k$, and several meridians of $S_i^1$ in $\beta_0^\bullet$. But for every meridian $S^1=SN(T_R)|_{x_m}, x_m=\gamma_{t_{j,m}}(v_j)\in l_{p_j}, t_{j,m}\in [t_{1,j},1]$ for some $j$ (like the one in purple in \Cref{fig:immersed-torus-and-standard-gamma}), $D^B=(SN(T_R)|_{l_{p_j}[t_{j,m},1]} \cup D_{1,j}^2)\subset \gamma_1^\bullet$ which bounds that $S^1$ is obviously null in $\pi_2 (M,\beta_0^\bullet)$, since it bounds $D^3=DN(T_R)|_{l_{p_j}[t_{j,m}, 1]}\cup D_{1,j}^3, D_{1,j}^3\subset D^3=DN(\gamma_1)|_{\gamma_1(v_j)}$ where $DN(T_R)$ denotes the normal disk bundle of $T_R$. Thus in the calculation of $\Theta$, every meridian $S^1$ corresponds to $(0,0)\cdot[\alpha]=0\in \Wh_1(\pi_1 M; \mathbb{Z}_2\times \pi_2 M)$, so it makes no contributions. So $\Theta(f_\beta)=\sum_{i=1,...,k} (0, \beta_i^*)\cdot [\gamma_i^*]$, but for the same reason, $\beta_i^*=[D_i^B]^{\delta_i^B}$, that is, stretching the disk near the intersection point along $l_{p_j}$ makes no difference in $\pi_2 (M)=\pi_2(M,\beta_0^\bullet)$. Thus, $\Theta(f_\beta)=\sum_{i=1,...,k} (0, [D^B_i]^{\delta_i^B})\cdot  [\delta_i^B]$.  
    
    \begin{figure}[!ht]
        \centering
        \includegraphics[width=0.75\textwidth]{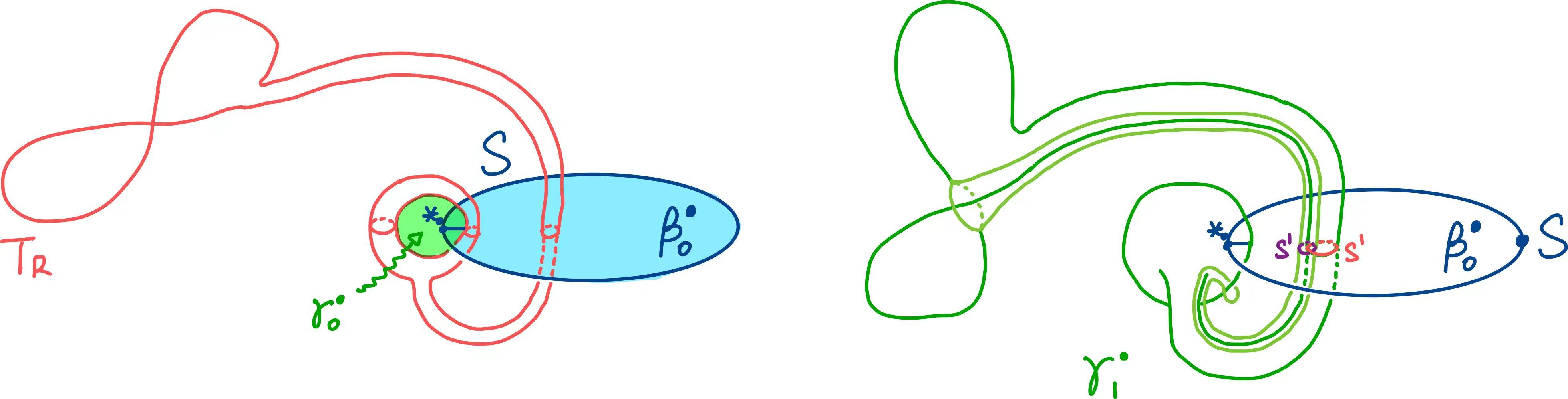}
        \caption{Here we illustrate immersed torus $T_R$ and the standard $\gamma_0^\bullet$ on the left and the desired $\gamma_1^\bullet$ on the right. If $T_R$ is embedded, $\gamma_1^\bullet\cap \beta_0^\bullet=I\cap S^1$ where $I$ is in blue and $S^1$ is in red. But since $T_R$ has a self-intersection point, by finger-pushing a small disk near the intersection point along $\gamma_t$ as we said in \Cref{fig:finger-push-ensure-immersion}, this induces another $S^1\subset \gamma_1^\bullet\cap \beta_0^\bullet$ in purple, which is exactly a meridian of the red $S^1$.}
        \label{fig:immersed-torus-and-standard-gamma}
    \end{figure}
\end{proof}

\begin{corollary} \label{cor:realize-all-elements-with-half-unknotted-immersed-barbell}
    Let $M^n$ be a smooth manifold with $n\geq 4$. For any $\sigma\in \pi_2 M$ with $w_2^M(\sigma)=0$, $\forall \alpha\in \pi_1 M$, there is a half-unknotted immersed barbell $\beta=\beta_{2,n-2}=(R, S, \gamma)$ with $S$ unknotted and $f_\beta\in \ker\Sigma\subset \pi_0\mathcal{P}$ with its Cerf diagram being a single eye of $(n-2,n-1)$-handle pair resulting in the immersed barbell diffeomorphism with respect to $\beta$ such that $\Theta(f_\beta)=(0,\sigma)\cdot [\alpha]$. In particular, when $n\geq 5$, $\beta$ is embedded.
\end{corollary}

\begin{proof}
    For $\sigma$ satisfying the condition, we first find $R$ which is an immersed 2-sphere with $q\in R$, such that $R$ represents $\sigma\in \pi_2(M,q)$. Since $w_2^M(\sigma)=0$, when $n\geq 5$, by slightly perturbing $R$ we get an embedded framed $R_\sigma$ representing $\sigma$; when $n=4$, we can induce some interior twists (see \cite[section 1.3]{9ac3db80-cd43-35c8-b9b4-a671794480ea}) locally on $R$ away from $q$ to get $R_\sigma$ which is a framed immersed 2-sphere representing the same $\sigma\in \pi_2(M,q)$. Then find a small embedded $\beta_0^\bullet: D^{n-1}\hookrightarrow M$ away from $R_\sigma$ with base point $*_0\in S:=\partial \beta_0^\bullet$. Fix $p\in int(\beta_0^\bullet)$. Find a path $\gamma_1$ from $*_0$ to $p$ such that $int(\gamma_1)\cap (\beta_0^\bullet\sqcup R_\sigma)=\emptyset$, $\gamma_1$ represents $\alpha\in \pi_1(M,\beta_0^\bullet)$. Find $\gamma_2$ from $p$ to $q$ such that $int(\gamma_2)\cap (\beta_0^\bullet\sqcup R_\sigma)=\emptyset$, $(\gamma_1*\gamma_2)^{-1}$ induces a natural isomorphism from $\pi_2(M,q)$ to $\pi_2(M,*_0)$: $\tau\to \tau^{(\gamma_1*\gamma_2)^{-1}}$, we need $\sigma^{(\gamma_1*\gamma_2)^{-1}}=\sigma$. Let $\gamma=\gamma_1*\gamma_2$, then let $\beta=(R_\sigma, S, \gamma)$. By \Cref{thm:general-calculation-of-second-hatcher-wagoner-invariant-for-beta-2-n-2} and \Cref{thm:general-calculation-of-second-hatcher-wagoner-invariant-for-immersed-beta-2-n-2}, $\Theta(f_\beta)=(0,\sigma)\cdot [\alpha]$.

    \begin{figure}[!ht]
        \centering
        \includegraphics[width=0.3\textwidth]{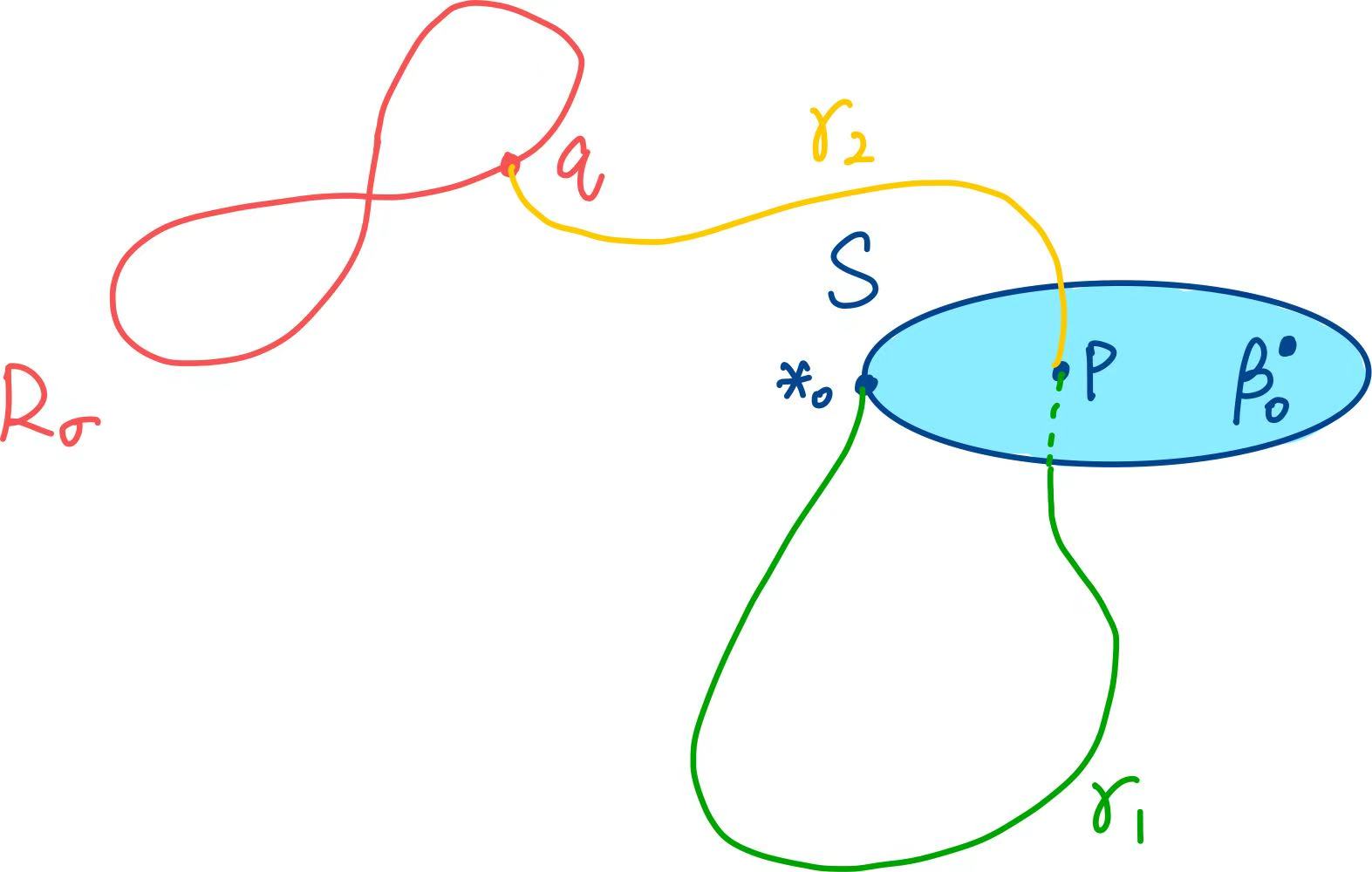}
        \caption{An illustration of $\beta=(R_\sigma, S,\gamma=\gamma_1*\gamma_2)$}
    \end{figure}
\end{proof}

When there is an odd 2-sphere $\tau\in \pi_2M$ (i.e. $w_2^M(\tau)=1$), we can do more (the key point is that the $\pi_2$ component of $\Theta$ only depends on $D_i$, not on $D_i'$):

\begin{corollary} \label{cor:realize-all-elements-with-half-unknotted-immersed-barbell-even-odd}
    If there exists $\tau\in \pi_2 (M^n)$ with $w_2^M(\tau)=1$, then for any $\sigma\in \pi_2M$, $\forall \alpha\in \pi_1 M$, there is a half-unknotted immersed barbell $\beta=\beta_{2,n-2}=(R, S, \gamma)$ and $f_\beta\in \ker\Sigma\subset \pi_0\mathcal{P}$ with its Cerf diagram being a single eye of $(n-2,n-1)$-handle pair resulting in the immersed barbell diffeomorphism with respect to $\beta$ such that $\Theta(f_\beta)=(0,\sigma)\cdot [\alpha]$. In particular, when $n\geq 5$, $\beta$ is embedded.
\end{corollary}

\begin{proof}
    For $\sigma\in \pi_2 M$ with $w_2^M(\sigma)=0$, the result just follows from \Cref{cor:realize-all-elements-with-half-unknotted-immersed-barbell}. For $w_2^M(\sigma)=1$. We find an immersed 2-sphere $R_\sigma$ representing $\sigma$, and an immersed 2-sphere $R_\tau$ representing $\tau$. Connect $R_\tau$ and $R_\sigma$ with a tube along an embedded arc $\gamma_0$ disjoint from both $R_\sigma$ and $R_\tau$ to get $R(\gamma_0)$. Then, when $n=4$, $R(\gamma_0)$ can be made into a framed immersed 2-sphere by inducing some interior twists locally; when $n\geq 5$, by slightly perturbing, $R(\gamma_0)$ can be made into a framed embedded 2-sphere. Find a small normal disk $\beta_0^\bullet$ on $\gamma_0$, let $*_0\in S=\partial \beta_0^\bullet$ be the base point. Find another embedded path $\gamma$ from $*_0$ to $p_\tau\in R_\tau$ away from $\beta_0^\bullet$ and $R(\gamma_0)$. See \Cref{fig:proof-of-corollary-6.4} for an illustration. By suitably choosing $\gamma_0$ and $\gamma$, we can find $\beta=(R(\gamma_0), S, \gamma)$ such that $\Theta(f_\beta)=(0,\sigma)\cdot[\alpha]$. 
    \begin{figure}[!ht]
        \centering
        \includegraphics[width=0.3\textwidth]{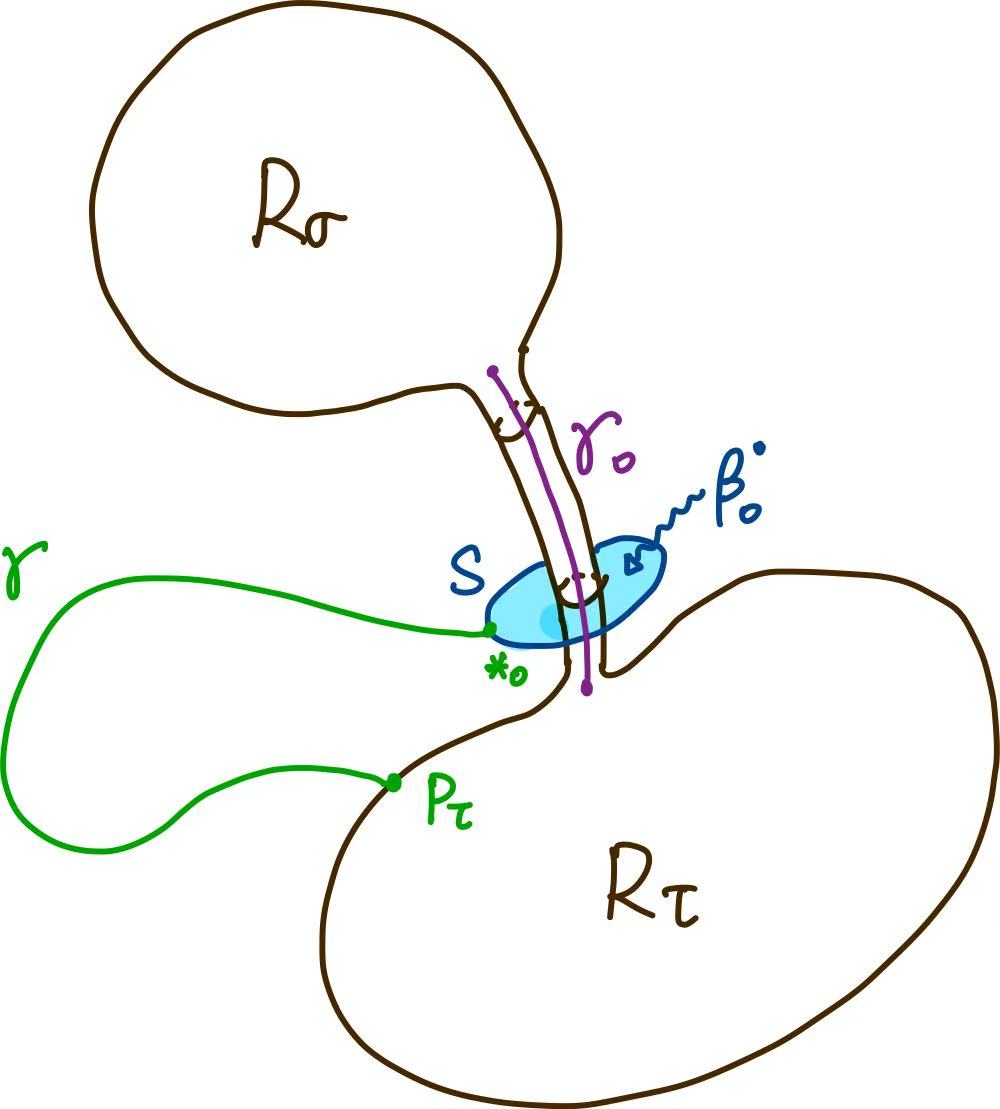}
        \caption{An illustration for the proof of \Cref{cor:realize-all-elements-with-half-unknotted-immersed-barbell-even-odd}}
        \label{fig:proof-of-corollary-6.4}
    \end{figure}
\end{proof}

Singh proved in \cite{singh2022pseudoisotopiesdiffeomorphisms4manifolds} that for any compact $M^4$, $\Xi=\langle(s,\sigma)\cdot[\gamma]| s=0 \text{ or }w_2^M(\sigma)\neq 0\rangle\subset \Wh_1(\pi_1 M; \mathbb{Z}_2\times \pi_2 M)$, we have $\Xi\subset \Theta(\ker \Sigma)$. Note that by \Cref{rmk:orientation-reversing-cobordism-for-hatcher-wagoner-invariants,rmk:involution}, if $\exists \tau\in \pi_2 M$ with $w_2^X(\tau)=1$, then for $\beta=\beta_{2,n-2}$ with $\Theta(f_\beta)=(0,\tau)\cdot [\alpha]$, $\Theta(g_\beta)=(-1)^{n+1}\bar\Theta (f_\beta)=(-1)^n(1, w_1^M(\alpha)\tau^{\alpha^{-1}})\cdot[\alpha^{-1}]$. Thus, combining the above results we get:

\begin{corollary}
    For a 4-dimensional manifold $M$, given any $(s,\sigma)\cdot[\gamma]\in \Wh_1(\pi_1 M; \mathbb{Z}_2\times \pi_2 M)$ with $s=0$ or $w_2^M(\sigma)\neq 0$,

    \begin{enumerate}[label=\em(\arabic*)]
        \item If $s=0$, there exists a half-unknotted immersed barbell $\beta$ and $f_\beta\in \ker\Sigma\subset \pi_0\mathcal{P}$ resulting in the immersed barbell diffeomorphism with respect to $\beta$, whose Cerf diagram contains a single eye of (2,3)-handle pair, satisfying $\Theta(f_\beta)=(s,\sigma)\cdot[\gamma]$.
        \item If $s=1$ and $w_2^M(\sigma)=1$, there exists a half-unknotted immersed barbell $\beta$ and $g_\beta\in \ker\Sigma\subset \pi_0\mathcal{P}$ resulting in the immersed barbell diffeomorphism with respect to $\beta$, whose Cerf diagram contains a single eye of (1,2)-handle pair, satisfying $\Theta(g_\beta)=(s,\sigma)\cdot[\gamma]$. \qedhere
    \end{enumerate}
\end{corollary}

We have the parallel version for higher dimensional manifold $M^n,n\geq 5$:

\begin{corollary} \label{cor:realize-all-elements-with-half-unknotted-immersed-barbell-even-odd-higher-dimensions}
    For a $n$-dimensional manifold $M$ with $n\geq 5$, given any $(s,\sigma)\cdot[\gamma]\in \ M; \pi_1 M)$ with $s=0$ or $w_2^M(\sigma)\neq 0$,

    \begin{enumerate}[label=\em(\arabic*)]
        \item If $s=0$, there exists a half-unknotted implanted barbell $\beta=\beta_{2,n-2}$ with the $(n-2)$-sphere unknotted in $M$ and $f_\beta\in \ker\Sigma\subset \pi_0\mathcal{P}$ resulting in the implanted barbell diffeomorphism with respect to $\beta$, whose Cerf diagram contains a single eye of $(n-2,n-1)$-handle pair, satisfying $\Theta(f_\beta)=(s,\sigma)\cdot[\gamma]$.
        \item If $s=1$ and $w_2^M(\sigma)=1$, there exists a half-unknotted implanted barbell $\beta=\beta_{2,n-2}$ and $g_\beta\in \ker\Sigma\subset \pi_0\mathcal{P}$ resulting in the implanted barbell diffeomorphism with respect to $\beta$, whose Cerf diagram contains a single eye of $(1,2)$-handle pair, satisfying $\Theta(g_\beta)=(s,\sigma)\cdot[\gamma]$. \qedhere
    \end{enumerate}
\end{corollary}

But in \cite{AST_1973__6__1_0} Hatcher and Wagoner showed that when $n\geq 5$, $\Theta: \ker\Sigma\to \Wh_1(\pi_1 M; \mathbb{Z}_2\times \pi_2 M)/\chi(K_3\mathbb{Z}[\pi_1 M])$ is surjective. Moreover, it is bijective when $n\geq 6$. To give a refinement of that result, for every $\gamma\in \pi_1 M$, we will construct a half-unknotted implanted barbell $\beta=\beta_{3,n-3}$ with $\Theta(f_\beta)=(1,0)\cdot[\gamma]$ in the next section.

In particular, for $M=(X_1\# X_2)\times I$ with $X_i$ closed, orientable, aspherical 3-manifold, Singh used $W$ and approximations for $K_3(\mathbb{Z}[\pi_1 M])$ and $\Theta(\mathcal{J}\cap \ker \Sigma)$ to show that there is $K\subset \pi_0\Diff_{PI}(M,\partial)$ and a surjection $K\to \bigoplus_{i\in \mathbb{N}} \mathbb{Z}$. Now by the above corollary, we have:

\begin{corollary}
    Let $M=(X_1\# X_2)\times I$ with $X_i$ closed, orientable, aspherical 3-manifold. In this case all $\sigma\in \pi_2 M=\mathbb{Z}[\pi_1 X_1*\pi_1 X_2]$ can be realized by embedded $S^2$ with $w_2^M(\sigma)=0$. Then $$\langle\text{implanted half-unknotted barbell diffeomorphisms}\rangle\subset \pi_0\Diff_{PI}(M,\partial)$$ is infinitely generated and of infinite $\mathbb{Z}$-rank. 
\end{corollary}
\section{A special barbell $\beta_{3,n-3}$ realizes surjectivity when $n\geq 6$}
\label{sec:a-special-barbell-beta-3-n-3-realizes-surjectivity-when-n-geq-6}

This section is devoted to constructing a half-unknotted implanted barbell $\beta=\beta_{3,n-3}$ in $M^n$ with $n\geq 6$ for every $\gamma\in \pi_1 M$, such that the associated $f_\beta\in \pi_0\mathcal{P}$ (see \Cref{cor:changing-from-i-1-i-handle-pair-to-n-i-n-i-1-handle-pair}) resulting in that barbell diffeomorphism satisfies $\Theta(f_\beta)=(1,0)\cdot[\gamma]$ (the construction is largely inspired by Hatcher's work in \cite[section 4]{hatcher1978concordance}). Together with \Cref{cor:realize-all-elements-with-half-unknotted-immersed-barbell-even-odd-higher-dimensions}, we show that every element $a\in \ker \Sigma\subset \pi_0\mathcal{P}$ has a representative which is a composition of some $f_\beta$, $\beta=\beta_{2,n-2}\text { or }\beta_{3,n-3}$ half-unknotted barbell. Consequently, this also shows that every diffeomorphism $F\in \pi_0\Diff_{PI}(M,\partial)$ satisfying that $\exists f\in \pi_0\mathcal{P}$ resulting in $F$ with vanishing first Hatcher-Wagoner invariant can be isotoped to a composition of half-unknotted barbell diffeomorphisms of index $(2,n-2)$ or $(3,n-3)$. 

First we assume $n\geq 6$. To construct the desired half-unknotted $(3,n-3)$-barbell $\beta$ in $M$, we present the following lemma:
\begin{lemma} \label{lem:embedding-disk-bundle-of-hopf-fibration-into-normal-bundle-of-sphere}
    Consider the Hopf fibration: $S^1\to S^3\to S^2$, which is the sphere bundle of a 2-dimensional vector bundle $\mathbb{R}^2\to E\to S^2$. Denote the associated disk bundle of $E$ by $D(E)$, thus $S^3=\partial D(E)$. Embed $S^2$ in a standard position $S^2\hookrightarrow \mathbb{R}^3\subset \mathbb{R}^6$, then there is a fiberwise embedding $i:D(E)\hookrightarrow \nu S^2$ over $S^2$, where we regard $\nu S^2\subset \mathbb{R}^6$ as a 4-dimensional vector bundle over $S^2$.
\end{lemma}
\begin{proof}
    A 2-dimensional vector bundle $E'$ over $S^2$ is totally determined by its Euler characteristic class $e(E')\in H^2(S^2)=\mathbb{Z}$, and a 4-dimensional vector bundle $E''$ over $S^2$ is totally determined by its second Stiefel-Whitney class $w_2(E'')\in H^2(S^2,\mathbb{Z}_2)=\mathbb{Z}_2$. We know that $e(E)=1$, then choose $\mathbb{R}^2\to E_1\to S^2$ with $e(E_1)=-1$. Let $E''=E\oplus E_1$, then $w_2(E\oplus E_1)=0$ so $E\oplus E_1$ is trivial over $S^2$. Thus $N( S^2,\mathbb{R}^6)\cong E\oplus E_1$ as a 4-dimensional vector bundle, which gives the desired fiberwise embedding $i: D(E)\hookrightarrow \nu S^2$ over $S^2$.
\end{proof}

Then we begin the construction: Choose a small region $\mathbb{R}^n\subset M$ with a standard $S^2$ in it, fiberwise embed $D(E)$ into $\nu S^2$ over $S^2$ by the above construction $i$. Let $R_0=i|_{\partial D(E)}$. Choose a normal disk $\beta_0^\bullet=D^{n-2}$ in $\nu S^2$, thus $\beta_0^\bullet\cap R_0=S^1$ which is exactly a fiber of the Hopf fibration. Let $S=\partial \beta_0^\bullet$. For every $\gamma\in \pi_1 M$, choose a path from $*_0\in S$ to $p\in R_0$ representing $\gamma$, still denoting by $\gamma$ for convenience. Then the desired half-unknotted $(3,n-3)$-barbell is $\beta=(R_0, S,\gamma)$ (see \Cref{fig:description-for-3-n-3-barbell} for an illustration).

\begin{figure}[!ht]
    \centering
    \includegraphics[width=0.4\textwidth]{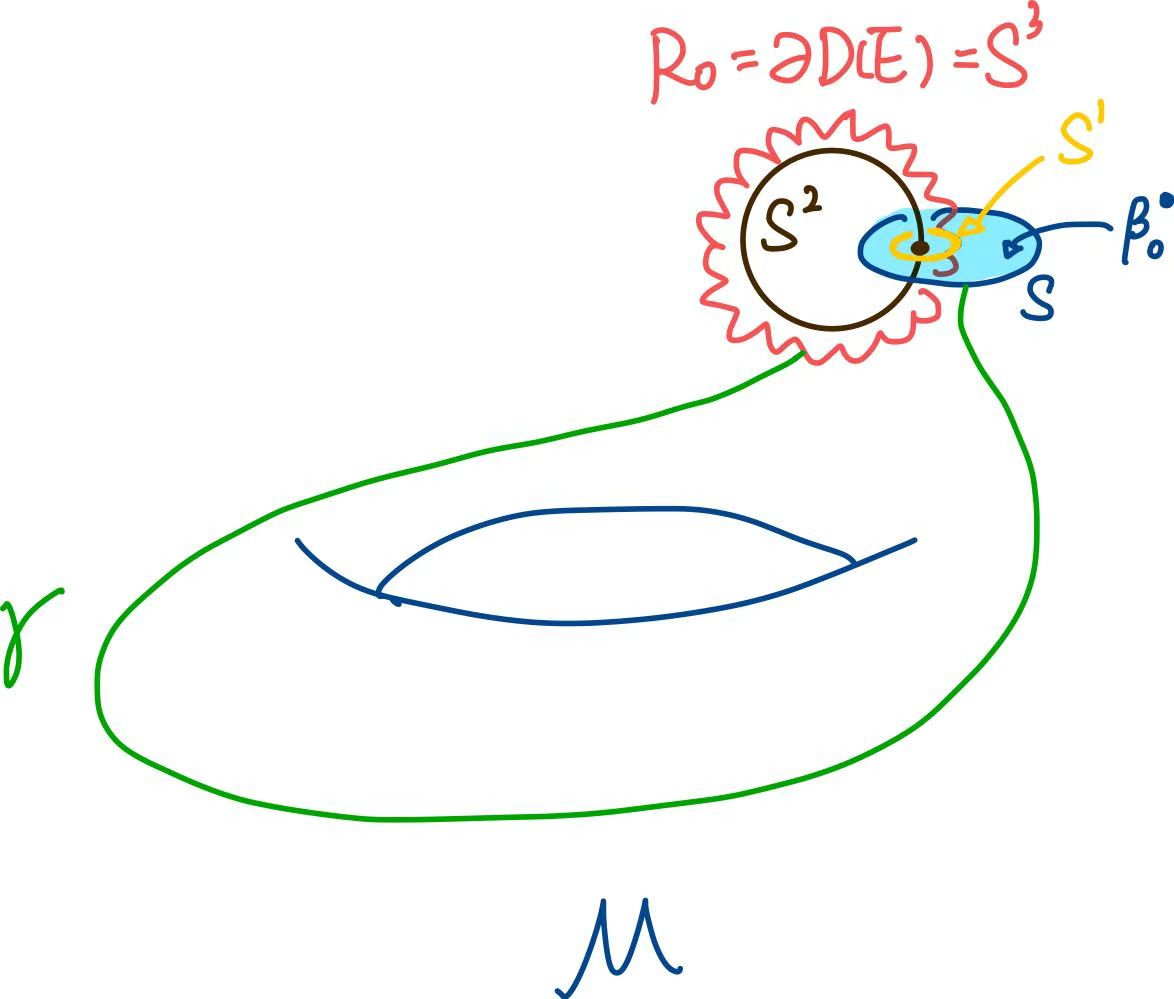}
    \caption{The description for the implanted $(3,n-3)$-barbell, with $R_0=\partial D(E)$ and $\beta_0^\bullet\cap R_0=S^1$ which is a fiber of the Hopf fibration $R_0\to S^2$.}
    \label{fig:description-for-3-n-3-barbell}
\end{figure}

\begin{proposition} \label{prop:calculate-second-hatcher-wagoner-invariant-for-beta-3-n-3}
    For the constructed $\beta=(R_0, S,\gamma)$, the $f_\beta\in \ker\Sigma$ we constructed in \cref{sec:changing-from-i-1-i-handle-pair-to-n-i-n-i-1-handle-pair}, whose Cerf diagram containing a single eye of $(n-3,n-2)$-handle pair, satisfies $\Theta(f_\beta)=(1,0)\cdot[\gamma]$. 
\end{proposition}

\begin{proof}
    We review the procedure in \cref{sec:changing-from-i-1-i-handle-pair-to-n-i-n-i-1-handle-pair} in the case $i=3$, first we obtain $g_\beta\in \pi_0\mathcal{P}$ which is represented by an element $[\nu T_R]\in \pi_1(\Emb(\nu S^2, M\#S^2\times S^{n-2}),*)$ where $T_R$ is an embedded $S^2\times S^1$ (in the dotted version, see \Cref{fig:embedded-torus} for an illustration). We choose $\gamma_0^\bullet$ to be a standard normal disk of $S$ which bounds the standard $S^2=\gamma_0$, then as the standard $S^2$ scans along $T_R$, i.e. scanning along $\gamma$ and then along $R_0$ and coming back along $\gamma^{-1}$, in the dotted version, we naturally get $\gamma_{t}^\bullet,t\in [0,1]$. We draw $\gamma_0^\bullet$ and $\gamma_1^\bullet=\gamma_0^\bullet\text{ tube}_\gamma R_0$ in \Cref{fig:embedded-torus}. The $\beta_t^\bullet=\beta_0^\bullet,t\in [0,1]$ are omitted there.

\begin{figure}[!ht]
    \centering
    \includegraphics[width=0.75\textwidth]{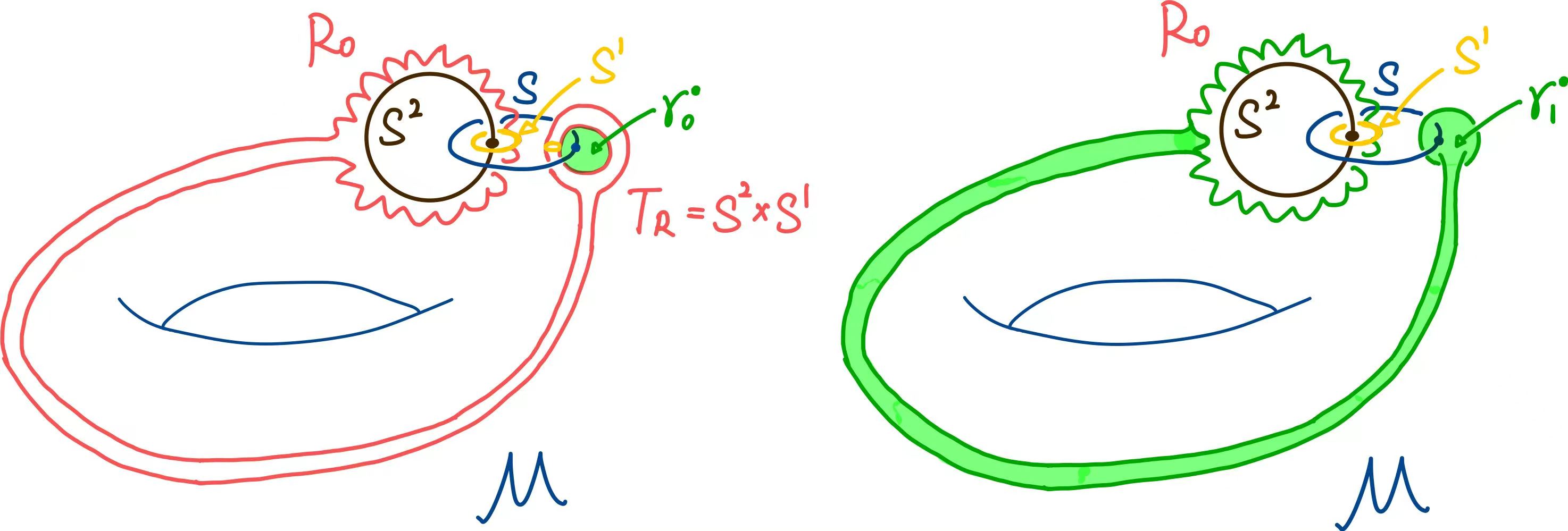}
    \caption{Here we draw $T_R=S^2\times S^1\text{ tube}_\gamma R_0, \gamma_0^\bullet$ and $\gamma_1^\bullet=\gamma_0^\bullet\text{ tube}_\gamma R_0$ in the dotted version.}
    \label{fig:embedded-torus}
\end{figure}
    
Then we apply \Cref{lem:one-parameter-version-of-dotted-and-0-framed-replacement} and use the strategy of \emph{one-parameter version of dotted and 0-framed replacement} to get $f_\beta$. To calculate $\Theta(f_\beta)$, like what we did in \cref{sec:computing-second-hatcher-wagoner-invariants-for-beta-2-n-2}, we consider $\beta_2^\bullet\cap \gamma_2^\bullet=I\sqcup S^1$ and two natural framings on $\nu (\beta_2^\bullet, M\times I)|_{S^1}=\nu (S^1, \gamma_2^\bullet\times I)$. Note that $\nu (S^1, \gamma_2^\bullet\times I)=\nu (S^1, \gamma_2^\bullet)\oplus E'$ and $\nu (\beta_2^\bullet, M\times I)=\nu (\beta_2^\bullet, M)\oplus E''$ where $E',E''$ are two trivial bundle over $S^1$, and $\nu (\beta_2^\bullet, M)|_{S^1}=\nu (S^1, \gamma_2^\bullet)$. Then $s=s_A-s_B=e_A-e_B$ where $e_A$ is the framing of $\nu (\beta_2^\bullet, M)|_{S^1}$ naturally induced by $\beta_2^\bullet$, $e_B$ is the usual $0$-framing of $\nu (S^1, \gamma_2^\bullet=D^3)$, i.e. naturally induced by a Seifert surface. Then like in \cref{sec:computing-second-hatcher-wagoner-invariants-for-beta-2-n-2}, we use the isotopy $\phi_2$ to pullback $\beta_2^\bullet=\phi_2(\beta_1^\bullet)$ to $\beta_1^\bullet$ and $\gamma_2^\bullet=\phi_2(\gamma_1^\bullet)$ to $\gamma_1^\bullet$, thus $s=e_A-e_B$ where $e_A$ is the framing of $\nu (\beta_1^\bullet, M)|_{S^1}$ naturally induced by $\beta_1^\bullet$ and $e_B$ is the usual $0$-framing of $\nu (S^1, \gamma_1^\bullet=D^3)$, i.e. naturally induced by a Seifert surface. 

To calculate $e_A$, we use the fact that the intersection $S^1$ is a fiber of the Hopf fibration, which divides $S^3=R_0$ into a $S^2$-family of $S^1$. Divide $S^3$ into two solid turi, $S^3=S^1\times D^2\cup_\partial D^2\times S^1$, and a generic fiber is just a (1,1) torus knot on the boundary $T^2=\partial(S^1\times D^2)$, in particular, it is an unknot. Let $\pi: R_0\to S^2$ be the fibration and the intersection $S^1=\pi^{-1}(p),p\in S^2$. Thus the framing of $\nu (\beta_1^{\bullet}=\beta_0^\bullet, M)|_{S^1}$ induced by $\beta_1^\bullet$, which is the standard normal disk at $p\in S^2$, is just the framing of $(\pi^*TS^2)|_{\pi^{-1}(p)}$ induced by $T_p S^2$, which is the 1-framing of the unknot $S^1\subset R_0=S^3$. See \Cref{fig:framing-for-fiber-in-hopf-fibration} for an illustration. Thus $e_A=1, e_B=0\Rightarrow s=1$.

\begin{figure}[!ht]
    \centering
    \includegraphics[width=0.75\textwidth]{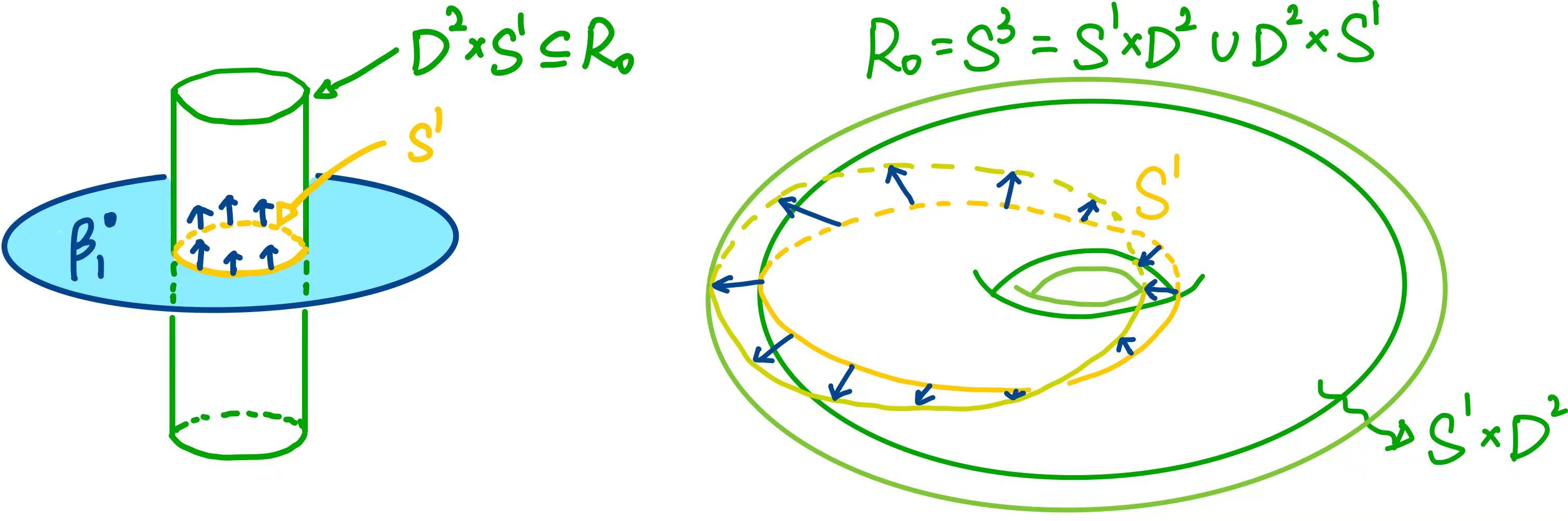}
    \caption{On the left, we draw how $R_0$ (thus $\gamma_1^\bullet)$  intersects $\beta_1^\bullet$ locally, and thus $e_A$ is induced by $(\pi^*TS^2)|_{\pi^{-1}(p)}$ where $\pi$ stands for the Hopf fibration $\pi:R_0=S^3\to S^2$. On the right we draw a section of $(\pi^*TS^2)|_{\pi^{-1}(p)}$ induced by $v\in T_p S^2$, which is the 1-framing on the (1,1) torus knot, which is the usual unknot.}
    \label{fig:framing-for-fiber-in-hopf-fibration}
\end{figure}

Since the constructed barbell is supported in $D^n\times S^1\subset M$ which has trivial $\pi_2$, this implies the $\pi_2$-component of $\Theta(f_\beta)$ must be 0. To calculate the $\pi_1$-component for $f_\beta$, by definition, we choose $q\in S^1=R_0\cap \beta_0^\bullet$, find two paths, $\delta^B\subset \gamma_1^\bullet$ from $q$ to $*_0$, $\delta^A\subset \beta_1^\bullet=\beta_0^\bullet$ from $q$ to $*_0$, thus by results from \cref{sec:computing-second-hatcher-wagoner-invariants-for-beta-2-n-2}, $\Theta(f_\beta)=(1,0)\cdot[\delta^B*(\delta^A)^{-1}]=(1,0)\cdot[\gamma]$ by the construction.

\end{proof}

\begin{remark}
    Let $M=S^1\times D^{n-1}$ and $\gamma=1\in \pi_1(M)=\mathbb{Z}$. In \cite[section 4]{hatcher1978concordance} Hatcher proposed a non-trivial pseudo-isotopy of $M$ which is originally due to Farrel (unpublished) by doing two embedded surgeries on the mid-ball $D_0=*\times D^{n-1}$ to get another embedded ball $D_1\subset M$ and considering the trace of the two surgeries, which gives a concordance from $D_0$ to $D_1$ in $M\times I$, i.e. a proper embedding of $F:D^{n-1}\times I\hookrightarrow M\times I$ with $F(D^{n-1}\times 0)=D_0, F(D^{n-1}\times 1)=D_1$. By using the h-cobordism theorem one can show that the complement of $F(D^{n-1}\times I)$ is a product, thus we get the pseudo-isotopy, which is a diffeomorphism of $M\times I$. In \Cref{fig:hatchers-example} we present a copy in Hatcher's work illustrating the two surgeries on $D_0$ to get $D_1$. As an exercise, the reader can vertify for themselves that, for $i=2,n\geq 6$, $D_1$ obtained from Hatcher's procedure is the same as the image of the mid-ball via the $(3,n-3)$-barbell diffeomorphism constructed above.   

    \begin{figure}[!ht]
        \centering
        \includegraphics[width=0.75\textwidth]{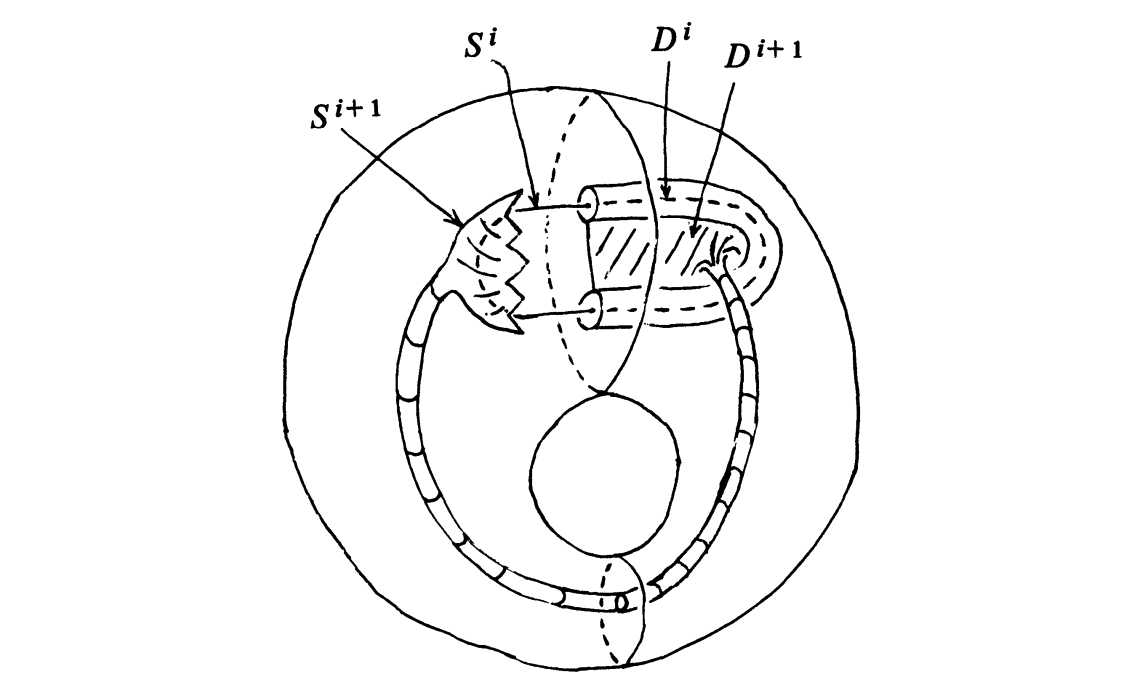}
        \caption{$D_1$ is obtained from $D_0$ by two embedded surgeries: The first one is by attaching a trivial $i$-handle $h_i$ on $D_0$, this gives the resulting ball $D_0'=D_0\#S^i\times S^{n-1-i}$, then we attach a $(i+1)$-handle $h_{i+1}$ with the attaching $D^{i+1}$ being the standard one tubed along the loop to another $S^{i+1}$, which is a perturbed Hopf fibration over $S^i$ where $S^i$ is obtained from the core of $h_i$. See \cite{hatcher1978concordance} for detailed descriptions.}
        \label{fig:hatchers-example}
    \end{figure}
\end{remark}

\begin{remark}
    In this remark, the author would like to discuss the difficulty she has encountered when generalizing the same result in dimension $n=5$. First, it can be shown that there doesn't exist an embedding $f: S^3\hookrightarrow S^2\times \mathbb{R}^3$ such that $f$ is homotopic to the Hopf map. So $R_0=S(E)$ cannot be embedded.
    
    But as we see in \cref{sec:generalizations-to-immersed-barbell-diffeomorphisms}, in the perspective of parameterised surgery, when we construct $g_\beta$ for a half-unknotted implanted $(i,n-i)$-barbell $\beta=(R_0,S,\gamma)$, $T_R=S^{i-1}\times S^1\text{ tube}_\gamma R_0$ does not need to be embedded, as long as it can be parameterised into a loop of framed embedded $S^{i-1}$, we can construct $[\nu T_R]\in \pi_1(\Emb(\nu S^{i-1}, M\# S^{i-1}\times S^{n-i+1}),*)$ which corresponds to $g_\beta\in \pi_0\mathcal{P}$ with $g_\beta$ resulting in the immersed barbell diffeomorphism with respect to $\beta$. So we may try to find an immersion $f: S^3\to S^2\times \mathbb{R}^3$, which is homotopic to Hopf map, such that by removing a standard small $D^3\subset S^3$, we can divide $S^3\setminus \text{int}(D^3)$ into an $I$-family of $D^2_t,t\in I$ with common boundary $S^1=\partial D^2_t$ such that $f|_{D^2_t}$ is an embedding for all $t\in I$. By perturbing $f$ we may assume that $f$ only has double points. Let $S\subset \text{Im} f$ be the image of all double points, i.e. $f|_{f^{-1}(S)}:f^{-1}(S)\to S$ is a 2-to-1 covering map between 1-dimensional smooth manifolds, i.e. disjoint union of some circles $S^1$. Let $f^{-1}(S)=\sqcup _{k=1}^m S_k^1$, then if $S^3$ can be divided satisfying the above property, it means that for every $k$, $f|_{S_k^1}$ must be an embedding, i.e. $f|_{f^{-1}(S)}$ must be a trivial double cover. But so far the author hasn't managed to find one.

    Here we present an explicit immersion $f$ with $S=S^1$ and $f^{-1}(S)\to S$ is the nontrivial 2-to-1 covering map, which might be useful:
    $$f: S^3\to S^2\times \mathbb{R}^3: (z_1,z_2)\mapsto ([z_1,z_2],x_1,x_2,x_1y_1+x_2y_2)$$
    where we regard $S^3=\{(z_1,z_2)\in \mathbb{C}^2||z_1|^2+|z_2|^2=1\}$ and $z_i=x_i+iy_i$.

\end{remark}

\nocite{*}
\bibliographystyle{plain}
\bibliography{references}

Q{\scriptsize IUZHEN} C{\scriptsize OLLEGE}, T{\scriptsize SINGHUA} U{\scriptsize NIVERSITY}, B{\scriptsize EIJING}, C{\scriptsize HINA} 

\emph{Email address}, Xiayu Tan: \texttt{tan-xy22@mails.tsinghua.edu.cn} 

\end{document}